\documentclass[11pt,a4paper]{article}

\usepackage[utf8]{inputenc}
\usepackage[T1]{fontenc}
\usepackage[english]{babel}
\usepackage{amsmath,amssymb,amsthm,amsfonts}
\usepackage{geometry}
\usepackage{mathrsfs}
\usepackage{bbm}
\usepackage{hyperref}
\usepackage{mathtools}
\usepackage{xcolor}
\newtheorem{lemma}{Lemma}[section]
\newtheorem{remark}{Remark}[section]

\newtheorem{proposition}{Proposition}[section]
\newtheorem{theorem}{Theorem}[section]
\theoremstyle{definition}

\newcommand{\E}{\mathbb{E}}
\newcommand{\Cbb}{\mathbb{C}}

\newcommand{\ii}{\mathrm{i}}

\newcommand{\R}{\mathbb{R}}

\newcommand{\Prob}{\mathbb{P}}
\newcommand{\Pbb}{\mathbb{P}}

\newcommand{\Cov}{\mathrm{Cov}}
\newcommand{\1}{\mathbbm{1}}
\newcommand{\tr}{{\ensuremath{\mathrm{tr}}}}
\newcommand{\Tr}{{\ensuremath{\mathrm{Tr}}}}

\newcommand\beq{\begin{equation}}
\newcommand\eeq{\end{equation}}
\newcommand{\Id}{I}
\newcommand{\C}{\mathbb{C}}

\title{Spectral Universality for Matrices with  Non-Linear Correlated Entries}
\makeatletter
\newcounter{author}
\renewcommand*\author[1]{%
  \stepcounter{author}%
  \ifnum\c@author=1
    \gdef\@author{#1}%
  \else
    \xdef\@author{\unexpanded\expandafter{\@author\and#1}}%
  \fi
  \csgdef{author@\the\c@author}{#1}}
\newcommand*\email[1]{%
  \csgdef{email@\the\c@author}{#1}}
\newcommand*\address[1]{%
  \csgdef{address@\the\c@author}{#1}}
\AtEndDocument{%
  \xdef\author@count{\the\c@author}%
  \c@author=1
  \print@authors}
\newcommand*\print@authors{%
  \ifnum\c@author>\author@count
  \else
    \print@author{\the\c@author}%
    \advance\c@author by 1
    \expandafter\print@authors
  \fi}
\newcommand*\print@author[1]{%
  \par\medskip
  \begin{tabular}{@{}l@{}}%
    \textsc{\csuse{author@#1}}\\
    \csuse{address@#1}\\
    \textit{E-mail address}:
    \href{mailto:\csuse{email@#1}}{\csuse{email@#1}}
  \end{tabular}}
\makeatother

\date{
    \today
}
\author{Marwa Banna}
\address{New York University Abu Dhabi, Division of Science, Mathematics, Abu Dhabi, UAE.}
\email{marwa.banna@nyu.edu}

\author{Issa-Mbenard Dabo}
\address{New York University Abu Dhabi, Division of Science, Mathematics, Abu Dhabi, UAE.}
\email{id2453@nyu.edu}

\author{Florence Merlev\`ede}
\address{LAMA, Univ Gustave Eiﬀel, Univ Paris Est Cr\'eteil, UMR 8050 CNRS, F-77454 Marne-La-Vall\'ee, France.}
\email{florence.merlevede@univ-eiffel.fr}

\usepackage{lipsum}
\providecommand{\keywords}[1]
{
  \small	
  \textit{Keywords:} #1
} %{} #1}
\providecommand{\subjclass}[1]
{
  \small	
  \textit{2000 Mathematics Subject Classification:} #1
}

\begin{document}

\maketitle

\begin{abstract}
We establish nonasymptotic spectral comparison results for matrices whose entries are non-linear functions of an iid random field. Under an exponential decay assumption on an $\mathbb L^\infty$-coupling coefficient, we derive high-probability bounds for the Hausdorff distance between their spectra and those of Gaussian matrices with matching covariance structures. We further derive a comparison with a covariance-matched free model and derive corresponding noncommutative Khintchine-type bounds. The results apply, in particular, to matrices whose entries are nonlinear transformations of causal linear processes, Volterra-type processes, and neural networks. The proof combines finite-memory approximation, block decomposition, Gaussian interpolation, and resolvent estimates. 
\end{abstract}
\\ \subjclass{15A52, 60G10, 60G60, 46L54, 60E15, 47A10.}

\section{Introduction and main results}
 A central question in random matrix theory is to understand spectral convergence in a strong sense, namely to control the location of the spectrum and the operator norm of functions of random matrices. Recent developments have shown that strong spectral methods are useful in a much broader range of problems than was previously accessible. They have played a role in progress on questions related to random graphs, geometric models, operator algebras, and neural networks, etc. For random graphs, strong asymptotic freeness of random permutation matrices yields sharp control of the new eigenvalues of random lifts \cite{BordenaveCollins2019}; the polynomial method recently gave quantitative strong convergence bounds and a new proof of Friedman's theorem \cite{ChenGarzaVargasTroppVanHandel2026}. In geometry, related control of permutation representations underlies near-optimal spectral gaps for random covers of hyperbolic surfaces \cite{HideMagee2023,MageePuderVanHandel2025}. % and applications to random harmonic maps and minimal surfaces \cite{Song2025}. 
In operator algebras, the GUE strong-convergence theorem was introduced to prove that $\operatorname{Ext}(C^*_{\mathrm{red}}(\mathbb F_2))$ is not a group \cite{HaagerupThorbjornsen2005}; later strong random matrix approximation results entered work on the absence of projections and on the Peterson-Thom conjecture \cite{HaagerupSchultzThorbjornsen2006,Hayes2022,BelinschiCapitaine2024,BordenaveCollins2024Norm,HayesJekelKunnawalkamElayavalli2025}. Closely related free probabilistic methods have been used to to investigate the limiting singular-value distributions of neural-network Jacobians and to control their operator norms \cite{PenningtonSchoenholzGanguli2018,CollinsHayase2023,dadoun2025stability}. Across these examples, the decisive information concerns spectral support, extreme eigenvalues, or operator norms, which cannot be recovered from convergence of normalized traces alone.

On the methodological side, the resolvent and linearization approach of Haagerup and Thorbj\o rnsen for GUE matrices, and its extensions to the GOE/GSE and Wigner/Wishart ensembles, established the first general strong-convergence results \cite{HaagerupThorbjornsen2005,Schultz2005,CapitaineDonatiMartin2007,Anderson2013}. Male and Collins--Male incorporated deterministic families and Haar matrices \cite{Male2012,CollinsMale2014}, while interpolation and smooth-functional-calculus expansions developed by Collins--Guionnet--Parraud and Parraud allow growing matrix coefficients \cite{CollinsGuionnetParraud2022,Parraud2022Haar,Parraud2023Expansion}. Other recent approaches use nonbacktracking or representation-theoretic expansions, or polynomial inequalities applied to traces that depend rationally on the dimension \cite{BordenaveCollins2024Compact,MageeDeLaSalle2026,ChenGarzaVargasTroppVanHandel2026,ChenGarzaVargasVanHandel2026}. In the nonasymptotic comparison direction most directly relevant here, general Gaussian matrices are compared with operator-valued free models \cite{BandeiraBoedihardjoVanHandel,BandeiraCipolloniSchroederVanHandel2026}, and sums of independent matrices, or of matrices with Markovian dependence, are compared with Gaussian matrices having matching covariance \cite{brailovskayaVanHandel,VanWerdeSanders2026}. The present work extends this Gaussian-replacement line to non-linear causal functionals of an iid field through finite-memory truncation, conditional blocking, and interpolation, thereby upgrading  earlier blocking-and-Lindeberg comparisons at the level of empirical spectral distributions \cite{BannaMerlevede15, MerlevedePeligradBanna2014} to nonasymptotic Hausdorff control of the full spectrum.

Strong nonasymptotic spectral estimates also yield matrix concentration inequalities.  Classical matrix concentration results, such as bounds of Bernstein or Khintchine type, give nonasymptotic control on the operator norm of random matrices. More recent developments have refined this picture by showing that, in many noncommutative and nonhomogeneous situations, sharp spectral estimates are better captured through comparison with an associated Gaussian or free-probabilistic model. In particular, the work of Bandeira, Boedihardjo and van Handel \cite{BandeiraBoedihardjoVanHandel} develops matrix concentration inequalities for Gaussian random matrices using tools from free probability, leading to bounds that are sensitive to the noncommutative structure of the model. 

This Gaussian viewpoint is especially useful when combined with comparison principles.
Rather than estimating a general random matrix directly, one compares it to a Gaussian
matrix with matching first and second moments, for which sharper spectral estimates are
available. Such arguments are particularly well suited to nonhomogeneous models, and
they also provide a natural framework for treating dependence when the covariance
structure remains tractable. In this direction, Brailovskaya and van Handel \cite{brailovskayaVanHandel} prove a nonasymptotic spectral comparison for sums of independent random matrices, showing that their spectra can be controlled through the corresponding Gaussian model with the same mean and covariance. Extending such results to the dependent setting, including the $m$-dependent case, is not an easy task.

\paragraph{The model:} The purpose of this paper is to prove strong nonasymptotic spectral estimates for a class of Wigner-type  matrices whose upper-triangular entries are generated by a causal dependence mechanism. Let $(\varepsilon_{u,v})_{u,v\in\mathbb Z}$ be an independent and identically 
distributed (iid) real-valued random field defined on a probability space
$(\Omega,\mathcal F,\mathbb P)$. Given a measurable real-valued function $f$, we define the triangular array
$(X_{k,\ell})_{k,\ell\in\mathbb Z}$ by
\beq \label{eq:Xkl-def-2D}
X_{k,\ell}
=
f\big(\varepsilon_{k-i,\ell-j}: i,j\ge 0\big),
\qquad k,\ell\in\mathbb Z .
\eeq
Throughout the paper, we assume that the entries $X_{k,\ell}$ are centered and uniformly bounded,  namely
\[
\mathbb E[X_{k,\ell}]=0,
\qquad
\|f\|_\infty\le M<\infty,
\]
where $\|.\|_\infty$ is the essential supremum so that $|X_{k,\ell}|\le M$ almost surely. The associated Wigner-type matrix is then defined by 
\beq\label{eq:X-matrix-def}
\mathbb X_n
=
\frac{1}{\sqrt n}\sum_{k=1}^n\sum_{\ell=1}^k X_{k,\ell}\,E_{k,\ell},
\eeq 
where $E_{k,\ell}= e_{k,\ell}+e_{\ell,k}$ whenever $k\neq \ell$ and $E_{k,k}= e_{k,k}$ with $e_{k,\ell}$ being the $n \times n $ unit matrix. Thus, $(\mathbb X_n)_{k\ell}=n^{-1/2}X_{k,\ell}$ for $k\ge \ell$ and $(\mathbb X_n)_{\ell k}=(\mathbb X_n)_{k\ell}$. Note that for each fixed column index $\ell$, the sequence $(X_{k,\ell})_{k\in\mathbb Z}$ is generated by a causal function of the innovation sequence $(\varepsilon_{u,v})_{u,v\in\mathbb Z}$. In particular, dependence is allowed along rows and columns of the matrix.

As in \cite{brailovskayaVanHandel}, our aim is to compare the spectrum of $\mathbb X_n$ to that of  a Gaussian Wigner-type matrix having the same covariance structure. Namely,
let $(g_{k,\ell})_{k,\ell\in\mathbb Z}$ be a centered Gaussian field such that
\[
\Cov(g_{i,j},g_{k,\ell})=\Cov(X_{i,j},X_{k,\ell}),
\]
and define
\beq \label{defofGn}
\mathbb G_n=\frac1{\sqrt n}\sum_{k=1}^n\sum_{\ell=1}^k g_{k,\ell}E_{k,\ell}.
\eeq
More precisely, the aim is to show that the spectrum of $\mathbb X_n$ is close, with high probability, to the spectrum
of $\mathbb G_n$. 

Due to the high level of technical complexity of the proofs, we shall at a first step simplify the model and consider that the entries are bounded functions of an iid random field indexed in one direction. More specifically, we will consider the following simplified version of \eqref{eq:Xkl-def-2D}:  
\beq\label{eq:Xkl-def-1D}
X_{k,\ell}
=
f\big(\varepsilon_{k-i,\ell}: i\ge 0\big),
\qquad k,\ell\in\mathbb Z .
\eeq
 In this setting, dependence occurs along the first index, while different columns are driven by disjoint innovation sequences and are therefore independent.  Note that
$\mathbb X_n$ can be rewritten as $\sum_{\ell=1}^n \mathbb Z_{\ell,n}$ where the $n \times n$ random matrices $\mathbb Z_{\ell,n} := \frac{1}{\sqrt n} \sum_{k=\ell}^n X_{k,\ell}\,E_{k,\ell}$ are independent. However, a direct application of the results of \cite{brailovskayaVanHandel} does not yield a useful estimate in the present setting, even in the case of $1$-dependent random variables. Indeed, the quantity $  \big \Vert   \max_{1 \leq \ell \leq n} \Vert \mathbb Z_{\ell,n}  \Vert \big  \Vert_{\infty}$, which is one of the main parameters governing their spectral comparison bounds, is growing with $n$. 

To quantify the dependence cross the rows, 
we shall introduce the $\mathbb L^\infty$-coupling coefficient of the model \eqref{eq:Xkl-def-1D}, which is defined as follows: let $(\varepsilon'_{u,v})_{u,v\in\mathbb Z}$ be an independent copy of $(\varepsilon_{u,v})_{u,v\in\mathbb Z}$.
For $n\in\mathbb Z$ and $j\in\mathbb Z$, define the coupled variable
\beq\label{eq:X-star-def}
X^*_{n,j}
:=
f\big(\ldots,\varepsilon'_{-2,j},\varepsilon'_{-1,j},\varepsilon'_{0,j},\varepsilon_{1,j},\ldots,\varepsilon_{n,j}\big).
\eeq
 By stationarity of the random field, the $L^{\infty}$-coupling  coefficient is then defined by
\beq\label{eq:delta-def}
\delta(n)
:=
\|X_{n,0}-X^*_{n,0}\|_\infty,
\qquad n\ge 0.
\eeq
We assume an exponential decay of dependence: there exist constants $K,c>0$ such that, for all integers $k\ge 0$,
\beq \label{conddelta}
\delta (k) \leq K e^{-c k } \, .
\eeq

One of our main results is a universality result for the spectrum ${\rm sp} ( \mathbb X_n)$ of $\mathbb X_n$ in terms of Hausdorff distance, which is defined for two subsets $A,B \subset \R$ by 
\[
d_H (A,B) = \inf \{ \varepsilon >0 : A \subseteq B + [-\varepsilon, \varepsilon ] \text{ and } B \subseteq A + [-\varepsilon, \varepsilon ] \} .
\] 

\begin{theorem}\label{thm:Hausdorff distance}
Let $(X_{k, \ell})_{k, \ell \in \mathbb Z}$ be defined by \eqref{eq:Xkl-def-1D}, $\mathbb X_n$ by \eqref{eq:X-matrix-def} and $\mathbb G_n$ by \eqref{defofGn}. Assume condition \eqref{conddelta}.  There exists a positive constant $C$ such that for any $t \geq 0$ and $\kappa >0$,
\[
\Pbb \big ( d_H ( {\rm sp} ( \mathbb X_n), {\rm sp} ( \mathbb G_n)  ) > C \varepsilon(t) \big ) \leq n {\rm e}^{-t} \, , 
\]
where 
\beq \label{defepsilont}
\varepsilon(t) =   \frac{\sigma_*}{  \sqrt n} \sqrt{t }  +  n^{-\kappa/2}  \sqrt{t }  
 +    t^{2/3} \frac{(\log n)^{1/2}}{n^{1/6}}  + t   \Big (  \frac{ \log n }{n}  \Big )^{1/6}  +   t^{5/3}   \Big ( \frac{\log n }{n}  \Big )^{1/2} , 
 \eeq
with
\[
\sigma_*^2=  \sum_{\ell  \geq 0} \big|  \Cov (  X_{0,0} ,    X_{\ell,0}  )\big|.
\]
\end{theorem}
Since $\big|  \Cov (  X_{0,0} ,    X_{\ell,0}  )\big| \leq \Vert X_{0,0} \Vert_1 \delta(\ell)$, under condition \eqref{conddelta}, $\sigma_*^2$ is finite.

\begin{remark}\label{rem:Extension-Rectangle}
   Our main results were stated under the assumption that $\mathbb X_{n,d}$ is a square self-adjoint matrix. We now explain how they extend to the more general setting in which $\mathbb X_{n,d}$ is a rectangular $n\times d$ matrix. Since such a matrix does not necessarily have a real spectrum, the natural objects to consider are its singular values. We define its Hermitian dilation by
$$
    \mathcal D(\mathbb X_{n,d})
    :=
    \begin{pmatrix}
        0 & \mathbb X_{n,d}\\
        \mathbb X_{n,d}^T & 0
    \end{pmatrix}.
$$
The eigenvalues of $\mathcal D(\mathbb X_{n,d})$ are precisely the singular values of $\mathbb X_{n,d}$ and their negatives, together with additional zero eigenvalues when $n\neq p$. More precisely, if $\operatorname{sv}(\mathbb X_{n,d})$ denotes the set of singular values of $\mathbb X_{n,d}$, then the spectrum of $\mathcal D(\mathbb X_{n,d})$ satisfies
$$
    \operatorname{sp}(\mathcal D(\mathbb X_{n,d})) \cup\{0\} =  (-  \operatorname{sv}(\mathbb X_{n,d}))
    \cup\{0\} \cup \operatorname{sv}(\mathbb X_{n,d}).
$$
Thus, the rectangular setting can be reduced to the self-adjoint one by applying our results to $\mathcal D(\mathbb X_{n,d})$. The proofs require only minor adjustments and otherwise follow the same strategy.  The arguments developed in this paper can  be adapted to the rectangular case, by replacing the representation \eqref{eq:X-matrix-def}  by
$$
    \mathcal D(\mathbb X_{n,d})
    =
    \frac{1}{\sqrt n}
    \sum_{i=1}^{n}\sum_{j=1}^{d}
    X_{i,j}\,\widehat E_{i,j}^{(n,d)},
    \qquad
    \widehat E_{i,j}^{(n,d)}
    =
    \begin{pmatrix}
        0_n & E_{i,j}^{(n,d)} \\
        \bigl(E_{i,j}^{(n,d)}\bigr)^{*} & 0_d
    \end{pmatrix}
    \in M_{n+d}(\mathbb C)_{\mathrm{sa}},
$$
where $E_{i,j}^{(n,d)}$ denotes the $n\times d$ standard matrix unit. Thus, the lower-triangular restriction $j\leq i$ disappears, and all pairs $(i,j)\in\{1,\ldots,n\}\times\{1,\ldots,p\}$ are now present. Assume that there exists positive constants $\gamma_-,\gamma_+ >0$ such that $\gamma_-\leq\frac nd\leq\gamma_+$, then $n+d,$ $n$ and $d$ are of same order. This yields rectangular analogues of Theorems~\ref{thm:Hausdorff distance} and~\ref{thm:dependent-free-comparison}, showing that the singular-value set of $\mathbb X_{n,d}$, augmented by zero when necessary, concentrates around the corresponding singular-value sets of $\mathbb G_n$ and $\mathbb X_{\mathrm{free}}$, with the same bounds as in the self-adjoint setting.

This formulation also yields spectral comparison results for the associated Gram matrices. Indeed, the nonzero eigenvalues of $MM^T$ and $M^TM$ coincide and are given by the squared singular values of $M$:
\[
\operatorname{sp}(\mathbb X_{n,d} \mathbb X_{n,d}^T)\setminus\{0\}
=
\operatorname{sp}(\mathbb X_{n,d}^T \mathbb X_{n,d})\setminus\{0\}
=
\big\{s_j(\mathbb X_{n,d})^2:1\le j\le r,\ s_j(\mathbb X_{n,d})>0\big\}.
\]
Consequently, any Hausdorff-distance estimate for the singular values transfers directly to the spectra of the corresponding Gram matrices.
\end{remark}

\noindent\textbf{Example 1.}
Let  $(\varepsilon_{u,\ell})_{u,\ell\in\mathbb Z}$ be an iid field of bounded real-valued random variables. For $0<\rho<1$ and $\kappa>0$, let $|a_i|\le \kappa\rho^i$ and define
\[
Y_{k,\ell}=\sum_{i\ge0}a_i\varepsilon_{k-i,\ell},
\qquad
X_{k,\ell}=h(Y_{k,\ell})-\mathbb E[h(Y_{k,\ell})],
\]
where $h:\mathbb R\to\mathbb R$ is Lipschitz. Then $(X_{k,\ell})$ is centered and uniformly bounded, and its coupling coefficients satisfy
\[
\delta(k)\le
4\|\varepsilon_{0,0}\|_\infty\sum_{i\ge k}|a_i|
\le
\frac{4\kappa\|\varepsilon_{0,0}\|_\infty}{1-\rho}\rho^k.
\]
Hence the assumptions of Theorem~1.1 hold.

\vspace{0.3cm} \noindent\textbf{Example 2.} \emph{(Neural Networks)} 
 Our results can also be applied to a class of convolutional neural networks considered in several studies, including \cite{BorovykhBohteOosterlee2017,han2023deep,wang2026optimal}. Fix $L\geq2$ the number of convolutional layers and $k\geq1$ the filter length. Let $N_0>L(k-1)$ be the input size and for $1\leq \ell \leq L$ denote by $ N_\ell=N_0-\ell(k-1)$ the output size of the $\ell$-th layer. Assume that the random field $(\varepsilon_{i,j})_{i,j\in\mathbb Z}$ of iid random variables is uniformly bounded; that is, there exists a constant $M<\infty$ such that $\sup_{i,j\in\mathbb Z}\|\varepsilon_{i,j}\|_\infty\leq M$. For $1\leq i \leq N_0$ and $1\leq j \leq n$ denote by $\mathbb X_{i,j}^{(0)}=\varepsilon_{i-L(k-1),j}$ the input layer; its columns indexed by $j$ are regarded as different observations in a batch of size $n$. In what follows, we assume that $k$ and $L$ are independent of $n$. For $1\leq \ell \leq L$, we define recursively the output of the $\ell$-th layer, denoted by $\mathbb X^{(\ell)}=(\mathbb X_{i,j}^{(\ell)})_{\substack{1 \leq i \leq N_\ell \\ 1 \leq j \leq n}}$, as follows 
\begin{equation}
\mathbb X_{i,j}^{(\ell)}
=
\sigma_\ell\left(
b_\ell+
\sum_{s=1}^{k}
a_s^{(\ell)}
\mathbb X_{i+k-s,j}^{(\ell-1)}
\right)-
\mathbb E\left[
\sigma_\ell\left(
b_\ell+
\sum_{s=1}^{k}
a_s^{(\ell)}
\mathbb X_{i+k-s,j}^{(\ell-1)}
\right)
\right],
\label{eq:cnn-entry-recursion}
\end{equation}
where $\bigl(a_s^{(\ell)}\bigr)_{1\leq s\leq k}$ is a size $k$ filter, $b_\ell\in\mathbb R$, and
$\sigma_\ell:\mathbb R\to\mathbb R$ is a $L_\ell$-Lipschitz activation function. Equivalently, the recursion \eqref{eq:cnn-entry-recursion} can be
written in matrix form as
$$
\mathbb X^{(\ell)}
=
\sigma_\ell\left(
A_\ell\mathbb X^{(\ell-1)}+B_\ell
\right)
-
\mathbb E\left[
\sigma_\ell\left(
A_\ell\mathbb X^{(\ell-1)}+B_\ell
\right)
\right]\in
 {\mathbb R^{N_\ell\times n}}
,
\quad 
A_\ell
=
\left(
a_{k+i-j}^{(\ell)}
\mathbf 1_{\{1\leq k+i-j\leq k\}}
\right)_
{\substack{
1\leq i\leq N_\ell\\
1\leq j\leq N_{\ell-1}
}},
$$
for $1\leq \ell \leq L$, where $B_\ell \in \mathbb R^{N_\ell\times n} $ is the matrix whose entries are all equal to $b_\ell$ and $\sigma_\ell$ is applied entrywise.

An induction over the layers shows that, for every $1\leq\ell\leq L$, there exists a measurable function $f_\ell$, independent of $n$, $i$, and $j$, such that
\begin{equation}\label{eq:Stationarity-Neurons}
\mathbb X_{i,j}^{(\ell)}
=
f_\ell\left(
\varepsilon_{i-u,j}:
(L-\ell)(k-1)\leq u\leq L(k-1)
\right).
\end{equation}
This representation follows because each weight matrix $A_\ell$ is a banded upper-triangular Toeplitz matrix. The triangular together with the initial shift $L(k-1)$ in $\mathbb X^{(0)}$ definition
 enforces $u\geq0$ in \eqref{eq:Stationarity-Neurons}. Moreover, the Toeplitz structure together with identical processing of each column, makes $f_\ell$ independent of the row and column indices. 
 
Set $X_{i,j} = \mathbb X_{i,j}^{(L)}$, this random variable is centered and uniformly bounded, and the corresponding coupling coefficients satisfy
$$
\delta (r) \leq 2M\prod_{\ell=1}^{L}\left(
L_\ell\sum_{s=1}^{k}|a_s^{(\ell)}| \right) \mathbf 1_{\{0\leq r\leq L(k-1)\}}.
$$
Hence the exponential-decay assumption \eqref{conddelta} is satisfied. Finally, according to Remark \ref{rem:Extension-Rectangle} we can extend 
Theorem \ref{thm:Hausdorff distance} to the rectangular matrix $\mathbb X_n = \frac{1}{\sqrt n} \mathbb X^{(L)}$ whenever $N_L$ and $n$ are of the same order.

\paragraph{Notation:} We use the following notation throughout the paper. We denote by $M_n (\mathbb C)_{\mathrm{sa}}$ the set of $n\times n$ self-adjoint matrices on $\mathbb C$. For any $n\times n$ matrix $M=(M_{ij})_{1\leq i,j\leq n}$, we denote its operator norm by $\|M\|$, its spectrum by $\operatorname{sp}(M)$ and its modulus by $|M| = (M^T M)^{\frac12}$. Its trace and normalized trace are denoted, respectively, by $\operatorname{Tr} M:=\sum_{i=1}^n M_{ii}$ and $\operatorname{tr} M:=\frac{1}{n}\operatorname{Tr} M$. For any scalar random variable $x$ and any $p>0$, we write $\|x\|_p:=\bigl(\mathbb E|x|^p\bigr)^{1/p}$ and denote the essential supremum of $|x|$ by $\|x\|_\infty$. Finally, we use the following convention $\mathbb E[M]^\alpha = (\mathbb E M)^\alpha$ and $\tr[M]^\alpha = (\tr M)^\alpha$. For two integers $u$ and $v$ with $u\le v$, we denote by $[\![u,v]\!]$ the integer interval notation $
[\![u,v]\!] := \{u,u+1,\dots,v\}$ 
and we adopt the convention that intervals are empty whenever $u>v$.

\section{Some intermediate results and proof strategy} \label{Sectionintermediate}

In this section, $(X_{k, \ell})_{k, \ell \in \mathbb Z}$ is defined by \eqref{eq:Xkl-def-1D}, $\mathbb X_n$ by \eqref{eq:X-matrix-def} and $\mathbb G_n$ by \eqref{defofGn}. To prove Theorem \ref{thm:Hausdorff distance},  we start by proving the following high-probability spectral inclusion.

\begin{theorem} \label{spectrumonedirection}
Assume condition \eqref{conddelta}. 
Let $\kappa >0$. There exists a positive constant $C$, depending on $K$, $c$, $\kappa$ and $M$, such that, for any $t \geq 0$,
\[
\Prob \Big (  {\rm sp } (   {\mathbb X}_n ) \subseteq     {\rm sp } (   {\mathbb G}_n )  + C  \varepsilon (t)   [-1,1]  \Big )   \geq  1-  ne^{-t} \, ,
\]
where $\varepsilon(t)$ is defined by \eqref{defepsilont}. 
\end{theorem}
The first step is to introduce a finite-memory approximation of the model. Fix an integer
$m\in\{1,\dots,n-1\}$. For $k,\ell\in\mathbb Z$, define
\beq\label{eq:Xm-def}
X^{(m)}_{k,\ell}
:=
\E\big(X_{k,\ell}\,\big|\,\sigma\big(\varepsilon_{k-i,\ell}:0\le i\le m\big)\big).
\eeq
The random variable $X^{(m)}_{k,\ell}$ is measurable with respect to the finite window $
(\varepsilon_{k,\ell},\varepsilon_{k-1,\ell},\ldots,\varepsilon_{k-m,\ell})$. Consequently, for each fixed column index $\ell$, the process
$(X^{(m)}_{k,\ell})_{k\in\mathbb Z}$ is $m$-dependent. Moreover,
\[
\E X^{(m)}_{k,\ell}=0,
\qquad
|X^{(m)}_{k,\ell}|\le M
\quad\text{almost surely}.
\]
We define the associated finite-memory Wigner-type matrix by
\beq\label{eq:Xm-matrix-def}
\mathbb X_n^{(m)}
=
\frac{1}{\sqrt n}
\sum_{k=1}^n\sum_{\ell=1}^k X^{(m)}_{k,\ell}E_{k,\ell}.
\eeq
The matrix $\mathbb X_n^{(m)}$ is the first approximation to $\mathbb X_n$ in the proof. Its approximation error is controlled by the coupling coefficient introduced above. In particular, one can easily show that
\[
\|X_{0,0}-X^{(m)}_{0,0}\|_\infty\le \delta(m),
\]
and therefore, under \eqref{conddelta}, this error decays exponentially fast in $m$. The proof of Theorem~\ref{spectrumonedirection} proceeds through the following comparison scheme:
\[
\mathbb X_n
\longrightarrow
\mathbb X_n^{(m)}
\longrightarrow
\mathbb Z^{(m)}_n
\longrightarrow
\mathbb G^{(m)}_n
\longrightarrow
\mathbb G_n .
\]
The first arrow is the finite-memory reduction described above; its spectral cost is controlled in Lemma~\ref{ineconcentration1}. The second arrow compares $\mathbb X_n^{(m)}$ with a block-Gaussian proxy $\mathbb Z^{(m)}_n$. The construction of this proxy is one of the main ingredients of the proof and is carried out in Section~\ref{Section:main proof}. We decompose the innovation sequence into blocks of length $m$, group these blocks into triple-blocks, and freeze suitable boundary blocks. Conditionally on the frozen boundary blocks, the remaining contributions from different triple-blocks become independent. Inside each triple-block, the middle part is further decomposed into a conditional expectation and a centered remainder; this decomposition produces uncorrelated components. We then replace these components by Gaussian vectors with matching covariances. Since uncorrelated Gaussian vectors are independent, the Gaussian replacement turns the frozen block decomposition into an independent Gaussian block model while preserving the relevant covariance structure. This construction allows us to adapt the interpolation argument of Brailovskaya and van Handel \cite{brailovskayaVanHandel} to the present setting. 

The block-Gaussian proxy $\mathbb Z_n^{(m)}$ is constructed in Section~\ref{Subsection:Step1}, while the resolvent comparison between $\mathbb X_n^{(m)}$ and $\mathbb Z_n^{(m)}$ is established in Section~\ref{Subsection:Step2}. Although $\mathbb Z_n^{(m)}$ is Gaussian conditionally on the frozen boundary blocks, its unconditional law retains a nontrivial random block structure through its dependence on these boundary variables. The next step is therefore to reorganize its contributions according to the frozen boundary blocks and to replace the resulting block variables by Gaussian vectors with matching covariances. This yields a Gaussian Wigner-type matrix $\mathbb G_n^{(m)}$ whose covariance structure coincides with that of the finite-memory process $(X_{i,j}^{(m)})_{i,j\in\mathbb Z}$. The comparison between $\mathbb Z_n^{(m)}$ and $\mathbb G_n^{(m)}$ is carried out in Section~\ref{Subsection:Step3} and requires a separate argument. Although $\mathbb Z^{(m)}_n$ is built from Gaussian blocks, its entries are not uniformly bounded:  conditionally on the frozen boundary blocks, they are Gaussian variables shifted by bounded terms.  Thus, the bounded-summand form of the comparison theorem of Brailovskaya and van Handel \cite{brailovskayaVanHandel} cannot be applied as a black box. Instead, we follow, in Section \ref{Subsection:Step3} their interpolation strategy and adapt it to the present setting. The final arrow compares $\mathbb G^{(m)}_n$ with the Gaussian comparison matrix $\mathbb G_n$, which has the same covariance structure as the original matrix $\mathbb X_n$. This accomplished in Section \ref{Section:Step4}.

\begin{remark}
The same approach applies to the more general two-directional causal model
\[
X_{k,\ell}
=
f\big(\varepsilon_{k-i,\ell-j}: i,j\ge 0\big),
\qquad k,\ell\in\mathbb Z,
\]
introduced in \eqref{eq:Xkl-def-2D}, where correlations may occur both across rows and across columns of $\mathbb{X}_n$. The proof in this more general setting requires a two-dimensional version of the blocking and boundary-freezing construction. Since these modifications are notationally heavier and would obscure the main ideas, we give the full proof in the one-directional case and explain the necessary non-trivial changes in Appendix~\ref{appendix:twodimensional-extension} concerning the extension to functions of iid random fields in two directions as given in Section \ref{Sectionextension2D}.
\end{remark}

The main technical estimate is a comparison of resolvent moments, obtained through the procedure described above and stated in the following proposition.

\begin{proposition}\label{propcomparaisonmomentsp}
Let $p\in\mathbb N^*$ and let $z\in\mathbb C$ with $0<\Im m (z)\le 1$. Then there exists a positive constant $C$,
depending on $\sum_{k\ge0}\delta(k)$, such that
\[
\left|
\left(\E\big[\tr\,|z\Id_n-\mathbb X_n^{(m)}|^{-2p}\big]\right)^{\frac1{2p}}
-
\left(\E\big[\tr\,|z\Id_n-\mathbb G_n|^{-2p}\big]\right)^{\frac1{2p}}
\right| \leq  U_{n,m}(p,z).
\]
 where
\beq \label{defUmn}
U_{n,m}(p,z)
=
C\left[
\frac{pn^2\delta(m)}{\Im m(z)^3}
+
\frac{M^3p^2}{\Im m(z)^4}
\left(
m\sqrt{\frac{m}{n}}+p\frac{m}{n}
\right)
+
\frac{(1+M)^{6/5}p^3}{\Im m(z)^4}\sqrt{\frac{m}{n}}
+
\frac{(1+M)^4p^5}{\Im m(z)^4}\left(\frac{m}{n}\right)^{3/2}
\left(\frac{n}{m}\right)^{1/(2p-1)}
\right].
\eeq
\end{proposition}

\medskip

The next lemma gives a preliminary localization of the spectrum of $\mathbb X_n$.
It combines the finite-memory approximation error with a rough high-probability
bound for the $m$-dependent model. Its role is not to provide the sharp spectral
scale, but rather to ensure that the spectrum of $\mathbb X_n$ lies in a
controlled deterministic interval before the resolvent comparison is applied.

\begin{lemma}\label{ineconcentration1}
There exists a universal positive constant $C$ such that, for all $t>0$ and all positive integer $m$,
\[
\Prob \left(
{\rm sp}(\mathbb X_n)
\subseteq
\left\{
2\sqrt n\,\delta(m)
+
CM\sqrt{\frac mn}\sqrt{n+t}
\right\}[-1,1]
\right)
\ge 1-e^{-t}.
\]
\end{lemma}

The final intermediate result to prove Theorem \ref{spectrumonedirection} is a resolvent estimate. It is the analogue, in the present dependent setting, of Lemma~7.3 in \cite{brailovskayaVanHandel}. It combines the resolvent moment comparison of Proposition~\ref{propcomparaisonmomentsp} with the preliminary localization estimate above, and yields a high-probability control of the resolvent of $\mathbb X_n$ in terms of the resolvent of the Gaussian comparison matrix $\mathbb G_n$.

\begin{lemma}\label{lma73} Let $z \in {\mathbb C}$ with $\Im m (z) \in ]0,1]$. Then, there exist $n_0 \in {\mathbb N}$ and a universal positive constant $C$ such that for all $n \geq n_0$, $x \geq \log n$, and every positive integer $m$, 
\begin{multline*}
\Prob  \Bigl (   \Vert  ( z I_n - {\mathbb X}_n )^{-1 } \Vert  \geq  2   \frac{\sqrt{n} \delta(m) }{ \Im m (z)^2}  +   C  \Big   \{  \Vert  ( z I_n - {\mathbb G}_n )^{-1 } \Vert   + 
  \frac{\sigma_*}{  \Im m (z)^{2} \sqrt n} \sqrt{x }  \\+  \frac{ x n^2 }{\Im m (z)^{3}} \delta(m)   
 +   \frac{ M^3 }{\Im m (z)^{4}}  x^2  \Big ( m \sqrt{\frac{m}{n}} + x  \frac{m}{n}  \Big )   +   \frac{1 }{\Im m (z)^{4}}  (1+M)^{4}  x^5  \Big ( \frac{m}{n}  \Big )^{3/2} \Big \}  \Bigr )  \leq  e^{-x} \, , 
\end{multline*}
where  $ \sigma_*^2=  \sum_{\ell  \geq 0} |  \Cov (  X_{0,0} ,    X_{\ell,0}  )|$.
\end{lemma}

Assuming Proposition~\ref{propcomparaisonmomentsp}, Lemma~\ref{ineconcentration1}, and Lemma~\ref{lma73}, the proof of Theorem~\ref{spectrumonedirection} follows by choosing $m$ of logarithmic order in $n$ and applying a standard resolvent argument. The estimates stated above prove the one-sided spectral inclusion
\[
{\rm sp}(\mathbb X_n)
\subseteq
{\rm sp}(\mathbb G_n)+C\varepsilon(t)[-1,1] ,
\]
with a high probability depending on $t$ and $\varepsilon(t)$ stated in Theorem \ref{spectrumonedirection}.

The proof of the Hausdorff-distance estimate requires one further step, namely the reverse inclusion. To obtain this we need the following result, which is the analogue of 
Proposition 7.7 in \cite{brailovskayaVanHandel} but in our dependence setting. 

\begin{proposition}\label{prop:lower-bound-X}
Let $\mathbb G_n$ be the centered Gaussian Wigner-type matrix with the same covariance structure as $\mathbb X_n$. Then there exist a universal constant $ C>0$ such that, for every $t\ge 0$ and every $n \geq 1$, 
$$
\Pbb\Bigl(
\operatorname{sp}(\mathbb G_n)\subseteq \operatorname{sp}(\mathbb X_n)+C\,\varepsilon_{n,m}(t)\,[-1,1]
\Bigr)
\ge 1-n e^{-t},
$$
where
$$
\varepsilon_{n,m}(t)
=
\alpha_{1,n,m}(t)
+
\alpha_{2,n,m}(t)^{1/2}
+
\alpha_{3,n,m}(t)^{1/3},
$$
with
$$
\alpha_{1,n,m}(t)
=
\sqrt n\,\delta(m)
+
\frac{M\sqrt m}{n^{1/4}}\sqrt t
+
M\sqrt{\frac{m}{n}}\,t,
\quad
\alpha_{2,n,m}(t)
=
n^2\delta(m)\,t
+
\frac{M^2m}{\sqrt n}\sqrt t
+
\frac{M^2m}{n}\,t,
$$
and
$$
\alpha_{3,n,m}(t)
=
M^3t^2\left(
m\sqrt{\frac{m}{n}}+t\frac{m}{n}
\right)
+
(1+M)^{6/5}t^3\sqrt{\frac{m}{n}}
+
(1+M)^4t^5\left(\frac{m}{n}\right)^{3/2}.
$$
\end{proposition}
This proposition implies the following counter part of Theorem \ref{spectrumonedirection}.
\begin{theorem} \label{spectrumotherdirection}
Assume condition \eqref{conddelta}. 
Let $\kappa >0$. There exists a positive constant $C$, depending on $K$, $c$,
and $M$, such that, for any $t \geq 0$ and every $\kappa $,
\[
\Prob \Big (  {\rm sp } (   {\mathbb G}_n ) \subseteq     {\rm sp } (   {\mathbb X}_n )  + C  \varepsilon (t)   [-1,1]  \Big )   \geq  1-  ne^{-t} \, ,
\]
where $\varepsilon(t)$ is defined by \eqref{defepsilont}. 
\end{theorem}

These intermediate estimates are proved in Section \ref{Sectionintermediate}.

\section{A free comparison theorem for the model given by \eqref{eq:Xkl-def-1D}}\label{Sectionfree}

\subsection{Free comparison}
We now derive a free-probabilistic comparison theorem for the dependent Wigner-type matrix $\mathbb X_n$ defined by \eqref{eq:X-matrix-def} where its entries are given by  \eqref{eq:Xkl-def-1D}. The aim is to compare the spectrum of $\mathbb X_n$ with that of a noncommutative model built from semicircular variables and having the same covariance structure. This result is in the spirit of \cite[Theorem~2.16]{brailovskayaVanHandel}, but the argument requires one important modification. The sharp Gaussian comparison theorem of  \cite[Theorem~2.1]{BandeiraBoedihardjoVanHandel} cannot be applied directly to $\mathbb X_n$, since the entries of $\mathbb X_n$ are
nonlinear causal functionals and are dependent along each column. Instead, we first compare $\mathbb X_n$ with its covariance-matched Gaussian analogue $\mathbb G_n$, and only then apply the Gaussian free comparison theorem to $\mathbb G_n$.

Recall that \cite[Theorem~2.1]{BandeiraBoedihardjoVanHandel} applies to a self-adjoint Gaussian matrix of the form
\begin{equation}\label{eq:H_Gaussian}
    H
    =
    A_0+\sum_{k=1}^N \xi_k A_k,
\end{equation}
where $A_0,A_1,\dots,A_N\in M_n(\mathbb C)_{\mathrm{sa}}$ are deterministic matrices and $\xi_1,\dots,\xi_N$ are independent standard real Gaussian variables. The associated free model is
\begin{equation}\label{eq:H_Free}
    H_{\mathrm{free}}
    =
    A_0\otimes 1
    +
    \sum_{k=1}^N A_k\otimes s_k,
\end{equation}
where $s_1,\dots,s_N$ is a free standard semicircular family in a tracial $C^\ast$-probability space. The theorem compares the spectrum of $H$ with that of $H_{\mathrm{free}}$. Assume that $(X_{k, \ell})_{k,\ell\in\mathbb Z}$ are given by \eqref{eq:Xkl-def-1D} and recall that the covariance-matched Gaussian matrix is
$$
    \mathbb G_n
    =
    \frac1{\sqrt n}\sum_{1\leq j \leq i \leq n}g_{i,j} E_{i,j},
$$
where the centered Gaussian field $(g_{i,j})_{1\leq j \leq i \leq n}$ has covariance
$$
    \operatorname{Cov}(g_{i,j},g_{i',j'})
    =
    \operatorname{Cov}(X_{i,j},X_{i',j'}),
$$
for $1\leq j \leq i \leq n$ and $1\leq j' \leq i' \leq n$.
The preceding coordinate representation of $\mathbb G_n$ is not yet written in the Gaussian matrix-series form required by \cite[Theorem~2.1]{BandeiraBoedihardjoVanHandel}, because the random variables $(g_{i,j})_{1\leq j \leq i \leq n}$ may not be independent. This is, however, only
a matter of representation which we  make explicit and use it to define the corresponding free semicircular model. We start by noting that there exists  iid standard Gaussian r.v.'s $(\xi_{i,j})_{1\leq i,j\leq n}$, such that 
$$
g_{i,j}
=
\sum_{k=1}^{n}
(\Gamma_n^{1/2})_{i,k}\xi_{k,j}
\qquad\text{almost surely}.
$$
where  $\Gamma_n = \big ( {\rm Cov} (g_{i,1}, g_{j,1} ) \big )_{1 \leq i,j \leq n}$. We then  define the matrix 
$$ 
A_{(k,\ell), n }=\frac1{\sqrt n}\sum_{i=\ell}^n ( \Gamma^{1/2}_n )_{i,k} E_{i,\ell}\in M_n(\mathbb C)_{\mathrm{sa}},
$$ 
and write
\begin{equation}\label{eq:Gaussian-matrix-representation}
    \mathbb G_n
    =
    \frac1{\sqrt n}\sum_{1\leq j \leq i \leq n} g_{i,j} E_{i,j}
    =
    \sum_{ k,\ell=1}^n  \xi_{k,\ell} A_{(k,\ell), n }.
\end{equation}
Now $\mathbb G_n$ is exactly of the form required by
\cite[Theorem~2.1]{BandeiraBoedihardjoVanHandel}, therefore one can apply this theorem in order to compare $\mathbb G_n$ and the associated free model. We may therefore define the free model associated with $\mathbb X_n$, or equivalently with $\mathbb G_n$, as follows. Let $(\mathcal A,\tau)$ be a tracial $C^\ast$-probability space, and let $(s_{k,\ell})_{1\leq k,\ell\leq n}$ be a free standard semicircular family in $\mathcal A$. Set
\beq \label{defofXfree}
    \mathbb X_n^{\mathrm{free}}
    :=
    \sum_{ k,\ell =1}^n A_{(k,\ell),n}\otimes s_{k,\ell}
    \in M_n(\mathbb C)_{\mathrm{sa}}\otimes\mathcal A .
\eeq
The spectrum $\operatorname{sp}(\mathbb X_n^{\mathrm{free}})$ is the $C^\ast$-spectrum of the
self-adjoint element $\mathbb X_n^{\mathrm{free}}$, and is deterministic. Define
$$
    \tilde{\sigma}_n^2
    =
    \sup_{\|u\|=\|v\|=1}
    \mathbb E\bigl[
        |\langle u,\mathbb G_n v\rangle|^2
    \bigr], \quad \sigma_n^2
    =
    \left\|\mathbb E[\mathbb G_n^2]\right\|, \quad 
    v_n^2
    =
    \|\operatorname{Cov}(\mathbb G_n)\|,
$$
where $\operatorname{Cov}(\mathbb G_n) $ denotes the $n^2\times n^2$ covariance matrix of the entries of $\mathbb G_n$.
\begin{lemma}
\label{lem:Gaussian-parameter-bounds}
Let $\sigma_*^2=  \sum_{\ell  \geq 0} \big|  \Cov (  X_{0,0} ,    X_{\ell,0}  )\big|$. Then there exists a universal constant $C>0$ such that
\begin{equation}
\label{UBquantities}
\tilde{\sigma}_n^2
\leq
\frac{C\sigma_\ast^2}{ n},
\qquad
\sigma_n^2
\leq
C\sigma_\ast^2,
\qquad
v_n^2
\leq
\frac{C\sigma_\ast^2}{ n}.
\end{equation}
\end{lemma}
The proof of the lemma is postponed to Appendix \ref{Appendix:free}.  Note that since $\big|  \Cov (  X_{0,0} ,    X_{\ell,0}  )\big| \leq \Vert X_{0,0} \Vert_1 \delta(\ell)$, therefore under condition \eqref{conddelta}, $\sigma_*^2$ is finite. Taking into account our Theorem \ref{spectrumonedirection} and applying  \cite[Theorem~2.1]{BandeiraBoedihardjoVanHandel} to the Gaussian matrix $\mathbb G_n$, we get the following comparison result.

\begin{theorem}[Free comparison for the dependent model]
\label{thm:dependent-free-comparison}
Let $(X_{k, \ell})_{k, \ell \in \mathbb Z}$ be defined by \eqref{eq:Xkl-def-1D},  $\mathbb X_n$ by \eqref{eq:X-matrix-def} and $\mathbb X_n^{\mathrm{free}}$ by \eqref{defofXfree}. Let $\varepsilon(t)$ be defined by \eqref{defepsilont}, and fix $\kappa>0$. Then under condition \eqref{conddelta} there exists a constant
$C=C(K,c,M,\kappa)>0$ such that, for every $t\ge0$ and every $n\geq 1$,
\begin{equation}\label{eq:sp-free}
    \mathbb P\!\left(
        \operatorname{sp}(\mathbb X_n)
        \subseteq
        \operatorname{sp}(\mathbb X_n^{\mathrm{free}})
        +
        C \Big (\varepsilon(t)
    +
     \sigma_* \frac{(\log n)^{3/4}}{n^{1/4}}
   \Big )[-1,1]
    \right)
    \ge
    1-ne^{-t}.
\end{equation}
In particular,
\begin{equation}\label{eq:lam-free}
    \mathbb P\!\left(
        \lambda_{\max}(\mathbb X_n)
        \le
        \lambda_{\max}(\mathbb X_n^{\mathrm{free}})
        +
       C \Big (\varepsilon(t)
    +
     \sigma_* \frac{(\log n)^{3/4}}{n^{1/4}}
   \Big )
    \right)
    \ge
    1-ne^{-t},
\end{equation}
and
\begin{equation}
\label{eq:dependent-Khintchine-expectation}
    \mathbb E\|\mathbb X_n\|
    \leq
    \|\mathbb X_n^{\mathrm{free}}\|+C\varepsilon(3\log n)
    +Mn^{-3/2}+C\sigma_*\frac{(\log n)^{3/4}}{n^{1/4}}.
\end{equation}
\end{theorem}

\begin{proof}  By Theorem~\ref{spectrumonedirection}, for every $t\ge0$,
$$
    \mathbb P\!\left(
        \operatorname{sp}(\mathbb X_n)
        \subseteq
        \operatorname{sp}(\mathbb G_n)
        +
        C\varepsilon(t)[-1,1]
    \right)
    \ge
    1-ne^{-t}.
$$
Denote the  event inside the probability by $\Omega_1(t)$. As shown above, the centered Gaussian matrix $\mathbb G_n$ admits the
representation \eqref{eq:Gaussian-matrix-representation}. 
Therefore \cite[Theorem~2.1]{BandeiraBoedihardjoVanHandel} applies to $\mathbb G_n$. Its
associated free model is precisely given by \eqref{defofXfree}. 
Consequently, for every $t\ge0$,
$$
    \mathbb P\!\left(
        \operatorname{sp}(\mathbb G_n)
        \subseteq
        \operatorname{sp}(\mathbb X_n^{\mathrm{free}})
        +
        C\Bigl\{
            v_n^{1/2}\sigma_n^{1/2}(\log n)^{3/4}
            +
            \tilde{\sigma}_n \sqrt t
        \Bigr\}[-1,1]
    \right)
    \ge
    1-e^{-t}.
$$
Denote by $\Omega_2(t)$ the event appearing inside the probability. Then, on the event $\Omega_1(t)\cap\Omega_2(t)$, we obtain the following bound:
$$
\begin{aligned}
    \operatorname{sp}(\mathbb X_n)
    &\subseteq
    \operatorname{sp}(\mathbb G_n)
    +
    C\varepsilon(t)[-1,1] \\
    &\subseteq
    \operatorname{sp}(\mathbb X_n^{\mathrm{free}})
    +
    C\Bigl\{
        \varepsilon(t)
        +
        v_n^{1/2}\sigma_n^{1/2}(\log n)^{3/4}
        +
        \tilde{\sigma}_n\sqrt t
    \Bigr\}[-1,1].
\end{aligned}
$$
By the union bound,
$$
    \mathbb P(\Omega_1(t)\cap\Omega_2(t))
    \ge
    1-(n+1)e^{-t}.
$$
Together with the upper bounds \eqref{UBquantities}, this proves \eqref{eq:sp-free}.

Let $A$ and $B$ be self-adjoint, if
$
    \operatorname{sp}(A)
    \subset
    \operatorname{sp}(B)+\eta[-1,1],
$
then
$
    \lambda_{\max}(A)
    \le
    \lambda_{\max}(B)+\eta.
$
Applying this results with
 $A=\mathbb X_n$, $B=\mathbb X_n^{\mathrm{free}}$ and  $\eta=C \Big (\varepsilon(t)+ \sigma_* \frac{(\log n)^{3/4}}{n^{1/4}}\Big)$, gives \eqref{eq:lam-free}.

Finally, we fix $n\geq2$ and set $t_n=3\log n$. On the event $\Omega_1(t_n)$ introduced
above, the first spectral comparison and self-adjointness give
$$
    \|\mathbb X_n\|
    \leq
    \|\mathbb G_n\|+C\varepsilon(t_n),
$$
while 
\beq \label{boundofomega1c}
    \mathbb P(\Omega_1(t_n)^c)
    \leq
    ne^{-t_n}=n^{-2}.
\eeq
As the entries of $\mathbb X_n$ are bounded by $M$, we also have 
\beq \label{boundofnorm2Xn}
    \|\mathbb X_n\|^2
    \leq
    \Tr [\mathbb X_n^{T} \mathbb X_n]
    =\frac1n\sum_{i=1}^n X_{i,i}^2
    +\frac2n\sum_{1\leq\ell<i\leq n}X_{i,\ell}^2
    \leq M^2n.
\eeq
Splitting the expectation over $\Omega_1(t_n)$ and its complement and taking into account \eqref{boundofomega1c} and \eqref{boundofnorm2Xn} yield
\begin{equation}
\label{eq:expectation-X-to-G}
    \mathbb E\|\mathbb X_n\|
    \leq
    \mathbb E\|\mathbb G_n\|
    +C\varepsilon(3\log n)
    +Mn^{-3/2}.
\end{equation}
On the other hand, \cite[Corollary~2.2]{BandeiraBoedihardjoVanHandel} and
\eqref{UBquantities} imply
\[
    \mathbb E\|\mathbb G_n\|\leq
    \|\mathbb X_n^{\mathrm{free}}\|
    +C(v_n\sigma_n)^{1/2}(\log n)^{3/4}\leq
    \|\mathbb X_n^{\mathrm{free}}\|
    +C\sigma_*\frac{(\log n)^{3/4}}{n^{1/4}}.
\]
Finally, combining the two bounds above proves 
\eqref{eq:dependent-Khintchine-expectation}. 
\end{proof}

\subsection{Khintchine-type consequences of free comparison}
We conclude by placing Theorem~\ref{thm:dependent-free-comparison} in the
context of noncommutative Khintchine theory, initiated by Lust-Piquard
\cite{LustPiquard1986} and further developed by Lust-Piquard and Pisier
\cite{LustPiquardPisier1991}. This theory extends the classical Khintchine
inequalities to random series with matrix coefficients. Let $H$ be the centered self-adjoint Gaussian matrix defined in
\eqref{eq:H_Gaussian}, and set $\sigma(H)^2=\bigl\|\mathbb E[H^2]\bigr\|$, the classical noncommutative Khintchine inequality yields
\begin{equation}
\label{eq:classical-noncommutative-Khintchine}
c\sigma(H)
\leq
\mathbb E\|H\|
\leq
C\sigma(H)\sqrt{\log(2n)},
\end{equation}
where $c,C>0$ are universal constants; see, for instance,
\cite[Section~1.1]{BandeiraBoedihardjoVanHandel}. Applying the upper bound to
the covariance-matched Gaussian model $\mathbb G_n$ and using
\eqref{eq:expectation-X-to-G}, we obtain
$$
    c\sigma_n\leq \mathbb E\|\mathbb X_n\|
    \leq
    C\sigma_n\sqrt{\log(2n)}
    +C\varepsilon(3\log n)
    +Mn^{-3/2},
$$
where $\varepsilon (t)$ is defined in \eqref{defepsilont} for $t \geq 0$.

The logarithmic factor multiplying $\sigma_n$ is unavoidable in general, in particular for diagonal series, where the problem reduces to estimating the maximum of $d$ Gaussian random variables. It may nevertheless be suboptimal when the coefficient matrices do not commute. In our setting, it turns out that free comparison allows us to remove this factor from the
leading term. Indeed, the degree-one free Khintchine inequality gives
\begin{equation}
\label{eq:free-Khintchine}
\sigma_n
\leq
\|\mathbb X_n^{\mathrm{free}}\|
\leq
2\sigma_n;
\end{equation}
see \cite[p.~208]{PisierOperatorSpaces} and \cite[Lemma~2.5]{BandeiraBoedihardjoVanHandel}. The absence of a logarithmic factor in this estimate makes the free model a natural benchmark for sharper norm estimates.

In the Gaussian setting, Bandeira, Boedihardjo, and Van Handel \cite{BandeiraBoedihardjoVanHandel} made this principle quantitative. Their free comparison theorem \cite[Theorem~2.1]{BandeiraBoedihardjoVanHandel} compares $H$ with its free analogue $H_{\mathrm{free}}$, defined in \eqref{eq:H_Free}, and yields
$$
    c\sigma(H)\mathbb E\|H\|
    \leq
    \|H_{\mathrm{free}}\|
    +
    C\,v(H)^{1/2}\sigma(H)^{1/2}(\log n)^{3/4} \, 
    \text{ where } \, 
    v(H)^2:=\|\operatorname{Cov}(H)\|.
$$
Consequently, along any sequence for which
$
    \frac{v(H)}{\sigma(H)}
    =
    o\bigl((\log n)^{-3/2}\bigr),
$
we have
$$
    c\sigma(H)\mathbb E\|H\|
    \leq
    \|H_{\mathrm{free}}\|+o(\sigma(H))
    \leq
    2\sigma(H)+o(\sigma(H)).
$$
Brailovskaya and van Handel \cite{brailovskayaVanHandel} extended this comparison mechanism to sums $Y=\sum_{i=1}^{N}Z_i$ of independent, centered, self-adjoint random matrices. Their argument first compares $Y$ with a Gaussian matrix having the same covariance structure and then invokes the Gaussian-to-free comparison, with
additional error terms accounting for the sizes of the individual summands.

Our dependent model follows the same two-step strategy, although the entries of $\mathbb X_n$ are nonlinear causal functionals and are neither independent nor jointly Gaussian. The comparison established in the previous subsection first transfers $\mathbb X_n$ to its covariance-matched Gaussian analogue $\mathbb G_n$, as quantified by \eqref{eq:expectation-X-to-G}. The subsequent Gaussian-to-free comparison yields \eqref{eq:dependent-Khintchine-expectation}, which, together with \eqref{eq:free-Khintchine}, gives
\begin{equation}
\label{eq:dependent-explicit-Khintchine}
c\sigma_n\leq\mathbb E\|\mathbb X_n\|
\leq
2\sigma_n
+C\varepsilon(3\log n)
+Mn^{-3/2}
+C\sigma_*\frac{(\log n)^{3/4}}{n^{1/4}}.
\end{equation}
In our setting, $\varepsilon(3\log n)=o(1)$ and $M$ and $\sigma_*$ are uniformly bounded in $n$, it follows that
$$
    c\sigma_n\leq\mathbb E\|\mathbb X_n\|
    \leq
    2\sigma_n+o(1)
    \qquad\text{as }n\to\infty.
$$

Thus, the dimension-dependent loss is confined to a vanishing additive remainder rather than multiplying the leading variance scale. Moreover, the spectral inclusion \eqref{eq:sp-free} controls more than the extreme spectral scale: it also excludes eigenvalues lying farther than the comparison error from the free spectrum, including possible outliers inside spectral gaps. Theorem~\ref{thm:dependent-free-comparison} therefore provides a dependent analogue of the preceding structure-sensitive estimates, with the norm of the covariance-matched free model as its principal term and the dependent, non-Gaussian structure contributing only an additive error.

\section{Extension to functions of iid random fields in two directions} \label{Sectionextension2D}

For reader convenience and for simplification we have stated and proved the results in case where the entries of $\mathbb X_n$ are  functions  of an iid random field in one direction, so when they satisfy the equation \eqref{eq:Xkl-def-1D}. Assume now that  the entries of $\mathbb X_n$ are  functions  of an iid random field in two directions, so they are given by the equation \eqref{eq:Xkl-def-2D}. Analyzing the proofs of the results stated in Section \ref{Sectionintermediate}, we have first to suitably approximate the field $(X_{k, \ell})_{k, \ell \in \mathbb Z}$ by a $m$-dependent random field $(X^{(m)}_{k, \ell})_{k, \ell \in \mathbb Z}$.  Hence we set 
\[
X^{(m)}_{k, \ell} =  \E \big ( X_{k, \ell} | \sigma(\varepsilon_{k-i, \ell-j} , 0 \leq i \leq m, 0 \leq j \leq m)  \big )\, .
\] 
We define now the ${\mathbb L}^{\infty}$-coupling coefficients associated with the random field $(X_{k, \ell})_{k, \ell \in \mathbb Z}$. Let $(\varepsilon'_{k, \ell})_{k, \ell \in \mathbb Z}$ be an independent copy of $(\varepsilon_{k, \ell})_{k, \ell \in \mathbb Z}$. Then define 
\[
X^{(*,1)}_{k,\ell}
:=
f\big( (\varepsilon_{k-i, \ell-j})_{0 \leq i \leq k-1 \atop  0 \leq j } , (\varepsilon'_{k-i, \ell-j})_{ i \geq k \atop  j \geq 0 }  \big) \ , \  X^{(*,2)}_{k,\ell} 
:=
f\big( (\varepsilon_{k-i, \ell-j})_{0 \leq i \atop  0 \leq j  \leq \ell-1 } , (\varepsilon'_{k-i, \ell-j})_{ i \geq 0 \atop  j \geq \ell }  \big)
\]
and
\beq  \label{defdeltak2D}
\delta_1 (k) = \Vert  X_{k, 0}  -X^{(*,1)}_{k,0}  \Vert_{\infty} \ , \  \delta_2 (k) = \Vert  X_{0, k}  -X^{(*,2)}_{0, k }  \Vert_{\infty} \text{ and } \delta(k) = \max ( \delta_1 (k) , \delta_2 (k) )  .
\eeq
With the definition of these coefficients and taking into account stationarity, we get in particular that 
\[
\Vert X_{k, \ell}  - X^{(m)}_{k, \ell}  \Vert_{\infty} = \Vert X_{0, 0}  - X^{(m)}_{0, 0}  \Vert_{\infty}  \leq \delta(m)
\]
and 
\[
\big | {\rm Cov} ( X_{k, \ell} ,   X_{k', \ell'}  ) \big |  \leq  \Vert X_{0,0} \Vert_1 \min \big ( \delta_1 (|k-k'|) ,  \delta_2 (|\ell-\ell'|)  \big )  \leq   \Vert X_{0,0} \Vert_1  \delta ( |k-k'| \vee |\ell-\ell'| ) \, .
\]
Analyzing the proof of Theorem \ref{thm:Hausdorff distance} and in particular the intermediate results stated in Section \ref{Sectionintermediate}, we infer that the following results holds: 
\begin{theorem}\label{thm:Hausdorff distance2D}
Let $(X_{k, \ell})_{k, \ell \in \mathbb Z}$ be defined by \eqref{eq:Xkl-def-2D} and $\mathbb X_n$ by \eqref{eq:X-matrix-def}.  Let $\delta(k)$ be defined by \eqref{defdeltak2D}. Assume 
that $(\delta(k))_{k \geq 0}$ decreases geometrically fast. Then, there exists a universal positive constant $C$ such that for any $t \geq 0$ and $\kappa >0$,
\[
\Pbb \big ( d_H ( {\rm sp} ( \mathbb X_n), {\rm sp} ( \mathbb G_n)  ) > C \varepsilon(t) \big ) \leq n {\rm e}^{-t} \, , 
\]
where 
\[
  \varepsilon(t) =   \frac{\sigma_*}{  \sqrt n} \sqrt{t }  +  n^{-\kappa/2}  \sqrt{t }  
 +    t^{2/3} \frac{\log n}{n^{1/6}}  + t     \frac{ (\log n)^{1/3} }{n^{1/6}}   +   t^{5/3}    \frac{\log n }{\sqrt{n}}   , 
\]
with
\[
\sigma_*^2=  \sum_{k,\ell  \in {\mathbb Z}} \big|  \Cov (  X_{0,0} ,    X_{k,\ell}  )\big|.
\]
\end{theorem}
The proof of this the theorem follows the same strategy as that of Theorem \ref{thm:Hausdorff distance}, essentially replacing $m$ with $m^2$ in the statements and proofs of the results in Section \ref{Sectionintermediate} (not replacing however $\delta(m)$ by $\delta(m^2)$). We do not give here the details because the proofs are too technical. However to convince the reader we give in Section \ref{appendix:twodimensional-extension} the construction of the Gaussian-type Wigner matrix whose spectrum is close to that of $\mathbb X_n$.  

Let us give an example of processes to which our results apply.

\noindent {\textbf{Example 3.} Another class of nonlinear random field is the Volterra process which plays an important role in the nonlinear system theory.  For any $(k,\ell) \in {\mathbb Z}^2$, define
\begin{equation} \label{Volterra}
X_{k , \ell}= \sum_{(i,j,r,s) \in {\mathbb Z}^4} a_{i,r} a_{j,s} \xi_{k-i, \ell -j} \xi_{k-r, \ell -s }\, ,
\end{equation}
where $a_{i,j}$ are real coefficients such that $a_{i,r}=0$ if $i=r$, and $(\xi_{i,j})$ is a sequence of iid centered real-valued random variables in ${\mathbb L}^{\infty}$. This model is a particular case of a more general Volterra process defined for ${\bf k} \in {\mathbb Z}^2$ by: 
$$
X_{{\bf k}}= \sum_{{\bf s_1}, {\bf s_2} \in {\mathbb Z}^2}  b_{{\bf s_1}, {\bf s_2} } \xi_{{\bf k- s_1}} \xi_{{\bf k- s_2}} \, ,
$$
when $b_{{\bf s_1}, {\bf s_2} } =a_{i,r} a_{j,s}$ if ${\bf s_1}=(i,j)$ and  ${\bf s_2}=(r,s)$.  If $|a_{i,j}| \leq  \kappa \rho^{i+j}$ for some $\rho \in (0,1)$ and $\kappa>0$, we infer that the coupling coefficients $(\delta(k))_{k>0}$ decrease exponentially fast. Hence Theorem \ref{thm:Hausdorff distance2D} applies when the entries of the underlying symmetric matrix are given by \eqref{Volterra}.
}

\section{Proofs of the results of Section \ref{Sectionintermediate}}

\subsection{Proof of Proposition \ref{propcomparaisonmomentsp} }\label{Section:main proof}

We shall decompose the innovation field into blocks of length $m$ and group these blocks
into \emph{triple-blocks} of length $3m$. The key idea is to ``freeze'' the
boundary innovation blocks by replacing them with a deterministic
sequence ${\bf a}_{\bf 3}$, which yields conditional independence between different triple-blocks. The middle part of each triple-block remains unchanged. This idea goes back to Berkes-Liu-Wu \cite{BerkesLiuWu} to get an extension to KMT results for functions of iid random variables.

\subsubsection{Step 1: Construction a Gaussian proxy $\mathbb Z^{(m)}$ of $\mathbb X^{(m)}$}\label{Subsection:Step1}
\paragraph{Triple-block decomposition of $\mathbb X^{(m)}$:}For each column index $j\in\mathbb Z$ and block index $k\in\mathbb Z$, define the
length-$m$ innovation block
\[
\overrightarrow{{\bf \eta}_{k} }(j) = \big( \varepsilon_{(k-1) m +1,j }, \dots,  {\varepsilon_{k m,j }} \big),
\]
and, for a deterministic sequence $(a_{i,j})_{i,j\in\mathbb Z}$, the corresponding
deterministic block
\[
{\bf a}_{k} (j) = \big(a_{(k-1) m +1,j }, \dots, {a_{k m ,j }} \big).
\]
We will condition on or "freeze" \emph{every third} innovation block, and therefore set
\[
\overrightarrow{{\bf \eta}_{\bf 3} }
=
\Big\{ \overrightarrow{{\bf \eta}_{3k} }(j) : (k,j)\in{\mathbb Z}^2 \Big\},
\qquad
{\bf a}_{\bf 3}
=
\Big\{{\bf a}_{3k} (j) : (k,j)\in{\mathbb Z}^2 \Big\}.
\]
Intuitively, these are the boundary blocks that will separate the triple-blocks. Let $k_{n,m} = [n/(3m)] +1$. Setting $X_{i,j} = 0$ if $i >n$, we can write
\[
{\mathbb X}^{(m)}
=
\frac{1}{\sqrt n}
\sum_{k=1}^{k_{n,m}}
\sum_{i= (3k-3)m+1}^{3km}
\sum_{j =1}^i X_{i,j}^{(m)} E_{i,j}.
\]
Thus the $i$-indices are partitioned into consecutive triple-blocks
\[
[1,3m],\ [3m+1,6m],\ \ldots,\ [3(k-1)m+1,3km],\ \ldots
\]
each of which is itself split into three sub-blocks of length $m$:
\[
\underbrace{[(3k-3)m+1,(3k-2)m]}_{\text{left/outer}}
\quad\cup\quad
\underbrace{[(3k-2)m+1,(3k-1)m]}_{\text{middle}}
\quad\cup\quad
\underbrace{[(3k-1)m+1,3km]}_{\text{right/outer}}.
\]
Recall that
\[
X^{(m)}_{i,j}
=
\E \Big( X_{i, j}\,\Big|\, \sigma ( \varepsilon_{i-k,j},  0 \leq k \leq m ) \Big)
:= g_m  (  \varepsilon_{i-m,j} , \dots, \varepsilon_{i,j}),
\]
for some measurable real-valued function $g_m$. In particular, for fixed $(i,j)$ the variable $X^{(m)}_{i,j}$ depends only on the
$m\!+\!1$ innovations $(\varepsilon_{i-m,j},\ldots,\varepsilon_{i,j})$.

\medskip
We now fix $k\ge 1$ and explain precisely how we  freeze the boundary innovations inside each triple-block.

\noindent\emph{(i) Left sub-block: $(3k-3)m+1 \leq i \leq (3k-2)m$.}
In this range, the window $(\varepsilon_{i-m,j},\ldots,\varepsilon_{i,j})$ overlaps the boundary block
$\overrightarrow{{\bf \eta}_{3k-3}}(j)$ and the next block $\overrightarrow{{\bf \eta}_{3k-2}}(j)$.
We freeze the boundary part by replacing it with ${\bf a}_{3k-3}(j)$ and define
\[
Y^{(3k-2)}_{i,j}   ( {\bf a}_{3k -3} (j)  )
=
g_m  \big(  a_{i-m,j} , \dots, a_{(3k-3) m , j} , \varepsilon_{(3k-3) m +1 , j} , \dots,  \varepsilon_{i,j}\big).
\]
We also define its mean
\beq \label{eq:centering term 1}
 M^{(3k-2)}_{i,j}   ( {\bf a}_{3k -3} (j)  )
=
\E \big( Y^{(3k-2)}_{i,j}   ( {\bf a}_{3k -3} (j)  ) \big).
\eeq

\smallskip
\noindent\emph{(ii) Middle sub-block: $(3k-2)m+1 \leq i \leq (3k-1)m$.}
Here the window of length $m\!+\!1$ is entirely contained in the two \emph{interior} innovation blocks
$\overrightarrow{{\bf \eta}_{3k-2}}(j)$ and $\overrightarrow{{\bf \eta}_{3k-1}}(j)$, so no freezing is needed:
\[
Y^{(3k-1)}_{i,j}
=
X_{i,j}^{(m)}
:=
H_i\big(\overrightarrow{{\bf \eta}_{3k-2}}(j), \overrightarrow{{\bf \eta}_{3k-1}}(j)\big),
\]
where $H_i$ is a measurable function.

\smallskip
\noindent\emph{(iii) Right sub-block: $(3k-1)m+1 \leq i \leq 3km$.}
In this range, the window overlaps the interior block $\overrightarrow{{\bf \eta}_{3k-1}}(j)$ and the
\emph{next boundary block} $\overrightarrow{{\bf \eta}_{3k}}(j)$. We freeze the boundary block by replacing it
with ${\bf a}_{3k}(j)$ and define
\[
Y^{(3k)}_{i,j}   ( {\bf a}_{3k} (j)  )
=
g_m  \big(  \varepsilon_{i-m,j} , \dots, \varepsilon_{(3k-1) m , j} , a_{(3k-1) m +1 , j}, \dots,   a_{i,j}\big),
\]
with mean
\beq \label{eq:centering term 2}
 M^{(3k)}_{i,j}   ( {\bf a}_{3k} (j)  )
=
\E \big( Y^{(3k)}_{i,j}   ( {\bf a}_{3k} (j)  ) \big).
\eeq
Now to keep the modified variables centered, we define
\[
{\tilde Y}^{(3k-2)}_{i,j}   ( {\bf a}_{3k -3} (j)  )
=
Y^{(3k-2)}_{i,j}   ( {\bf a}_{3k -3} (j)  )
-
M^{(3k-2)}_{i,j}   ( {\bf a}_{3k -3} (j)  ),
\]
and
\[
 {\tilde Y}^{(3k)}_{i,j}   ( {\bf a}_{3k } (j)  )
=
Y^{(3k)}_{i,j}   ( {\bf a}_{3k} (j)  )
-
M^{(3k)}_{i,j}   ( {\bf a}_{3k} (j)  ).
\]
Define then the modified partially frozen and centered variables
\[
{\tilde X}^{(m)}_{i,j}   ( {\bf a}_{3k -3} (j) ,  {\bf a}_{3k } (j)  )
=
\begin{cases}
{\tilde Y}^{(3k-2)}_{i,j}   ( {\bf a}_{3k -3} (j)  ) ,
& \text{ for }  (3k-3)m+1 \leq i \leq (3k-2)m, \\[0.25em]
X^{(m)}_{i,j},
& \text{ for }  (3k-2)m+1 \leq i \leq (3k-1)m,\\[0.25em]
 {\tilde Y}^{(3k)}_{i,j}   ( {\bf a}_{3k } (j)  ) ,
& \text{ for }  (3k-1)m+1 \leq i \leq 3km .
\end{cases}
\]
 
\medskip
This construction leads to independence since within the $k$-th triple-block, the only dependence across different triple-blocks comes from the boundary
innovation blocks $\overrightarrow{{\bf \eta}_{3k-3}}(j)$ and $\overrightarrow{{\bf \eta}_{3k}}(j)$ that are shared
with neighboring triple-blocks. Hence by freezing these boundary blocks to deterministic values
${\bf a}_{3k-3}(j)$ and ${\bf a}_{3k}(j)$, i.e by  conditioning on
$\overrightarrow{{\bf \eta}_{\bf 3}}={\bf a}_{\bf 3}$, the remaining randomness in the $k$-th triple-block is carried
only by the \emph{interior} blocks $\overrightarrow{{\bf \eta}_{3k-2}}(j)$ and $\overrightarrow{{\bf \eta}_{3k-1}}(j)$,
which are independent across different $k$'s. 

We now collect these centered outer parts into vectors. Hence the vectors
\[
\overrightarrow{Y}^{(3k-2)}  ( {\bf a}_{3k -3} (j)  )
=
\big \{ {\tilde Y}^{(3k-2)}_{i,j}   ( {\bf a}_{3k -3} (j)  ) \big \}_{i \in [\![ 3(k-1)m+1, n \wedge (3k-2)m  ] \!] }
\]
and
\[
\overrightarrow{Y}^{(3k)}  ( {\bf a}_{3k} (j)  )
:=
\big \{ {\tilde Y}^{(3k)}_{i,j}   ( {\bf a}_{3k} (j)  ) \big \}_{i \in [\![ (3k-1)m+1, n \wedge 3km  ] \!] }
\]
are independent. Moreover, independence across different $j$ follows from independence of
distinct columns of the innovation field. We now decompose the middle sub-block as follows:
\[
X_{i,j}^{(m)} = \E \big  (  X_{i,j}^{(m)} |  \overrightarrow{{\bf \eta}_{3k-2} }(j)   \big) +   \big \{X_{i,j}^{(m)}  -\E   \big (  X_{i,j}^{(m)} |  \overrightarrow{{\bf \eta}_{3k-2} }(j)  \big )   \big \}  , 
\]
 and define the following vectors supported on the middle sub-block:
\[
\overrightarrow{U}^{(k)}  ( j )  :=  \big \{ \E \big  (  X_{i,j}^{(m)} |  \overrightarrow{{\bf \eta}_{3k-2} }(j)   \big)  \big \}_{i \in [\![ (3k-2)m+1,  n \wedge(3k-1)m  ] \!] }  
\]
and
\[
\overrightarrow{V}^{(k)}  ( j )  :=  \big \{  X_{i,j}^{(m)}  - \E \big  (  X_{i,j}^{(m)} |  \overrightarrow{{\bf \eta}_{3k-2} }(j)   \big)  \big \}_{i \in [\![ (3k-2)m+1,  n \wedge(3k-1)m  ] \!] }   .
\]
By construction,
\[
\E\big(\overrightarrow{V}^{(k)}(j)\,\big|\,\overrightarrow{{\bf \eta}_{3k-2}}(j)\big)=\overrightarrow{0},
\qquad
\text{and hence}\qquad
\E\big[\overrightarrow{U}^{(k)}(j)\,\big(\overrightarrow{V}^{(k)}(j)\big)^t\big]=0,
\]
so the vectors $\overrightarrow{U}^{(k)}(j)$ and $\overrightarrow{V}^{(k)}(j)$ are uncorrelated.
This is the key reason for splitting the middle block in this way. Next, we define random column vectors in $\mathbb R^n$ by placing the previously defined pieces into their natural coordinate positions, and filling the remaining coordinates with zeros. Here
$\overrightarrow{0}_{r,s}$ denotes the zero vector filling coordinates $r,\dots,s$ and is considered empty
if $r>s$. Define:
\[
\overrightarrow{B}^{(1)}_k   ( {\bf a}_{3k -3} (j)  )
=
\Big ( \overrightarrow{0}_{1,(3k-3)m},   \overrightarrow{Y}^{(3k-2)}  ( {\bf a}_{3k -3} (j)  ) ,  \overrightarrow{U}^{(k)}  ( j ) , \overrightarrow{0}_{(3k-1)m+1, n }  \Big )^t , 
\]
\[
\overrightarrow{B}^{(2)}_k   ( {\bf a}_{3k} (j)  )
=
\Big ( \overrightarrow{0}_{1, (3k-2)m},  \overrightarrow{V}^{(k)}  ( j ) ,   \overrightarrow{Y}^{(3k)}  ( {\bf a}_{3k} (j)  ) ,  \overrightarrow{0}_{3km+1, n }  \Big )^t ,
\]
and
\[
\overrightarrow{B_k}   ( {\bf a}_{3k -3} (j) ,   {\bf a}_{3k} (j)  )
=
\overrightarrow{B}^{(1)}_k   ( {\bf a}_{3k -3} (j)  )
+
\overrightarrow{B}^{(2)}_k   ( {\bf a}_{3k } (j)  ) . 
\]

\noindent
Note that the vector $\overrightarrow{B}^{(1)}_k$ contains the \emph{left outer} centered block and the conditioned part
$\overrightarrow{U}^{(k)}$ of the middle block. The vector $\overrightarrow{B}^{(2)}_k$ contains $\overrightarrow{V}^{(k)}$ of the middle block and the \emph{right outer} centered block. In particular,
$\overrightarrow{B_k}$ reconstructs the full triple-block contribution up to the mean terms. It is important, however, to note that $\overrightarrow{B}^{(1)}_k   ( {\bf a}_{3k -3} (j)  )$ and
$\overrightarrow{B}^{(2)}_k   ( {\bf a}_{3k } (j)  )$ are not correlated and that the random vectors
$ \big ( \overrightarrow{B_k}   ( {\bf a}_{3k -3} (j) ,   {\bf a}_{3k} (j)  ) \big )_{k \geq 1, j \geq 1}
$
are centered and independent.

\paragraph{Covariances and Gaussian counterparts.}
We introduce a Gaussian block model that matches the covariance structure of
$\overrightarrow{B_k}$. The main advantage is that, in the Gaussian setting,
\emph{uncorrelatedness implies independence}. To do so, we consider the covariance matrices
\[
\begin{aligned}
\Gamma^{(1)}_k  ( {\bf a}_{3k -3} (j)  )
&=
\E \Big ( \overrightarrow{B}^{(1)}_k   ( {\bf a}_{3k -3} (j)  )
\big (  \overrightarrow{B}^{(1)}_k   ( {\bf a}_{3k -3} (j)  ) \big )^t  \Big ), \\  \Gamma^{(2)}_k  ( {\bf a}_{3k } (j)  )
&=
\E \Big ( \overrightarrow{B}^{(2)}_k   ( {\bf a}_{3k } (j)  )
\big (  \overrightarrow{B}^{(2)}_k   ( {\bf a}_{3k } (j)  )  \big )^t  \Big )  \, . 
\end{aligned}
\]
Note that the dependence on ${\bf a}_{3k-3}(j)$ and ${\bf a}_{3k}(j)$ reflects the boundary freezing, the remaining
randomness in $\overrightarrow{B}^{(1)}_k$ and $\overrightarrow{B}^{(2)}_k$ is carried by interior innovation blocks.

Let $(Z^{(1)}_{i,j})_{1 \leq i,j \leq n}$ and $(Z^{(2)}_{i,j})_{1 \leq i,j \leq n}$ be two sequences of iid
standard gaussian random variables that are independent between them and independent of
$(\overrightarrow{\eta}, {\bf a})$. We define the following independent Gaussian random vectors
\[
\overrightarrow{G}^{(1)}_k   ( {\bf a}_{3k -3} (j)  )
=
\Big (\Gamma^{(1)}_k  ( {\bf a}_{3k -3} (j)  )  \Big )^{1/2} ( Z^{(1)}_{1,j}, \dots, Z^{(1)}_{n,j} )^t ,
\]
and
\[
\overrightarrow{G}^{(2)}_k   ( {\bf a}_{3k } (j)  )
=
\Big (\Gamma^{(2)}_k  ( {\bf a}_{3k } (j)  )  \Big )^{1/2} ( Z^{(2)}_{1,j}, \dots, Z^{(2)}_{n,j} )^t  .
\]
Set
\[
\overrightarrow{G_k}   ( {\bf a}_{3k -3} (j) ,   {\bf a}_{3k} (j)  )
=
\overrightarrow{G}^{(1)}_k   ( {\bf a}_{3k -3} (j)  )
+
\overrightarrow{G}^{(2)}_k   ( {\bf a}_{3k } (j)  ) . 
\]
The vectors $\overrightarrow{G}^{(1)}_k$ and $\overrightarrow{G}^{(2)}_k$ are Gaussian by construction, centered,
and independent across $k$ and $j$. Moreover, because
$\overrightarrow{B}^{(1)}_k$ and $\overrightarrow{B}^{(2)}_k$ are uncorrelated, we have
\[
\mathrm{Cov}\big(\overrightarrow{B_k}\big)
=
\Gamma^{(1)}_k + \Gamma^{(2)}_k
=
\mathrm{Cov}\big(\overrightarrow{G_k}\big),
\]
so $\overrightarrow{G_k}$ and $\overrightarrow{B_k}$ share the same covariance matrix for fixed boundary values ${\bf a}_{3k-3}(j),{\bf a}_{3k}(j)$. Moreover, $\overrightarrow{G_k}$ has  the stronger structural property that the relevant block components are independent.

\paragraph{Construction of Gaussian matrix $\mathbb Z^{(m)}$.}
Next we set, for any integer $i \in [3(k-1)m+1 , 3km]$, 
\[
 A_i  \big(    {\bf a}_{3(k-1)} (j) ,  {\bf a}_{3k} (j) \big)  : =  \Big ( \overrightarrow{B_k}  \big(   {\bf a}_{3(k-1)} (j) ,   {\bf a}_{3k} (j) \big)  \Big )_{i}   \,  
\]
and collect the deterministic mean corrections \eqref{eq:centering term 1} and \eqref{eq:centering term 2} introduced in the freezing procedure by setting 
\[
M_{i,j}\big({\bf a}_{3(k-1)}(j),{\bf a}_{3k}(j)\big)
=
M^{(3k-2)}_{i,j}\big({\bf a}_{3(k-1)}(j)\big)
+
M^{(3k)}_{i,j}\big({\bf a}_{3k}(j)\big).
\]
With the notations above, the following identity is valid:
\[
{\mathbb X}^{(m)} = {\mathbb X}^{(m)}  ( \overrightarrow{{\bf \eta}_{\bf 3} } )
=
\frac{1}{\sqrt n} \sum_{k=1}^{k_{n,m}}  \sum_{i= (3k-3)m+1}^{3km}  \sum_{j =1}^i
\Big (
A_i  \big(    \overrightarrow{{\bf \eta}_{3(k-1)} }(j) ,   \overrightarrow{{\bf \eta}_{3k} }(j) \big)
+
M_{i,j}  \big(    \overrightarrow{{\bf \eta}_{3(k-1)} }(j) ,   \overrightarrow{{\bf \eta}_{3k} }(j) \big)
\Big )
E_{i,j}.
\]
In the same manner, for any integer $i \in [3(k-1)m+1 , 3km]$, we set
\[
 Z_i \big(    {\bf a}_{3(k-1)} (j) ,  {\bf a}_{3k} (j) \big)
: =
\Big ( \overrightarrow{G_k}     \big(    {\bf a}_{3(k-1)} (j) ,  {\bf a}_{3k} (j) \big)    \Big )_{i}    
\]
and define 
\[
{\mathbb Z}^{(m)} =  {\mathbb Z}^{(m)}  (\overrightarrow{{\bf \eta}_{\bf 3} } ) 
=
\frac{1}{\sqrt n} \sum_{k=1}^{k_{n,m}}  \sum_{i= (3k-3)m+1}^{3km}  \sum_{j =1}^i
\Big (
Z_i  \big(    \overrightarrow{{\bf \eta}_{3(k-1)} }(j) ,   \overrightarrow{{\bf \eta}_{3k} }(j) \big)
+
M_{i,j}  \big(    \overrightarrow{{\bf \eta}_{3(k-1)} }(j) ,   \overrightarrow{{\bf \eta}_{3k} }(j) \big)
\Big )
E_{i,j}  \, .
\]
Note that the matrix $\mathbb Z^{(m)}$ is obtained from $\mathbb X^{(m)}$ by replacing, within each triple-block, the
centered random vector $\overrightarrow{B_k}$ by its Gaussian counterpart $\overrightarrow{G_k}$ with the same
covariance for fixed boundary values, while keeping the same mean correction $M_{i,j}$. Hence, $\mathbb Z^{(m)}$ is a Gaussian proxy for $\mathbb X^{(m)}$ that preserves conditional covariance structure and is built from independent blocks.

\subsubsection{Step 2: Gaussian interpolation and reduction to an independent-block model}\label{Subsection:Step2}

We want now to give an upper bound for the quantity 
\[
Q_1= \Bigl|\bigl(\E\bigl[\tr\,|z\Id-{\mathbb X}^{(m)}|^{-2p}\bigr]\bigr)^{\frac{1}{2p}}-\bigl(\E\bigl[\tr\,|z\Id- {\mathbb Z}^{(m)} |^{-2p}\bigr]\bigr)^{\frac{1}{2p}}\Bigr| .
\]
The strategy is to compare ${\mathbb X}^{(m)}$ and  ${\mathbb Z}^{(m)}$ by interpolation  following the approach used in the proof of Theorem~6.8 in \cite{brailovskayaVanHandel}.
With this aim, we introduce the interpolating matrix ${\mathbb D}^{(m)}(t)$ that continuously deforms the
non-Gaussian block variables into Gaussian blocks while keeping the deterministic mean corrections unchanged.
Controlling the derivative of the resolvent moment along this path will yield the desired estimate for $Q_1$. We start by defining for any $t \in [0,1]$,
\[
D_i  \big(   {\bf a}_{3(k-1)} (j) ,  {\bf a}_{3k} (j) \big) (t) = \sqrt{t}  A_i \big(    {\bf a}_{3(k-1)} (j) ,  {\bf a}_{3k} (j) \big) + \sqrt{1-t}  Z_i \big(    {\bf a}_{3(k-1)} (j) ,  {\bf a}_{3k} (j) \big) ,
\]
and set
\[
{\mathbb D}^{(m)}  (t) =  {\mathbb D}^{(m)}   (\overrightarrow{{\bf \eta}_{\bf 3} } ) (t) =    \frac{1}{\sqrt n} \sum_{k=1}^{k_{n,m}}  \sum_{i= (3k-3)m+1}^{3km}  \sum_{j =1}^i  \Big (  D_i   \big(    \overrightarrow{{\bf \eta}_{3(k-1)} }(j) ,   \overrightarrow{{\bf \eta}_{3k} }(j) \big)  (t)+  M_{i,j}  \big(    \overrightarrow{{\bf \eta}_{3(k-1)} }(j) ,   \overrightarrow{{\bf \eta}_{3k} }(j) \big)  \Big )   E_{i,j}. 
\]
By construction, the endpoints of the interpolation satisfy
$
{\mathbb D}^{(m)}(1)={\mathbb X}^{(m)}$ and $
{\mathbb D}^{(m)}(0)={\mathbb Z}^{(m)}$. Also note that 
\[
\E\big[{\mathbb D}^{(m)}  (1)\mid \overrightarrow{{\bf \eta}_{\bf 3} } ={\bf a}_{\bf 3} \big] =    \frac{1}{\sqrt n} \sum_{k=1}^{k_{n,m}}  \sum_{i= (3k-3)m+1}^{3km}  \sum_{j =1}^i   M_{i,j}  \big(    {\bf a}_{3(k-1)} (j) ,   {\bf a}_{3k} (j) \big)     E_{i,j} 
\]
As in \cite{brailovskayaVanHandel}, we shall analyze the quantity
\[
 \Big |  \frac{d}{dt} \E\bigl[\tr\,|z\Id-{\mathbb D}^{(m)}  (t) |^{-2p}\bigr]\ \Big | \, .
\]
The key point is that the block construction from Step~1 yields conditional independence once we freeze the
boundary blocks $\overrightarrow{{\bf \eta}_{\bf 3}}$ by conditioning on
$\overrightarrow{{\bf \eta}_{\bf 3}} ={\bf a}_{\bf 3}$ and to exploit the resulting independent structure. With this aim, we start by noticing that 
\[
 \Big |  \frac{d}{dt} \E\bigl[\tr\,|z\Id-{\mathbb D}^{(m)}  (t) |^{-2p}\bigr]\ \Big |
 \leq
 \E  \Big |  \frac{d}{dt} \E_{\overrightarrow{{\bf \eta}_{\bf 3} } ={\bf a}_{\bf 3} }\bigl[\tr\,|z\Id-{\mathbb D}^{(m)}  (t) |^{-2p}\bigr]\ \Big | ,
\]
where the notation $\E_{\overrightarrow{{\bf \eta}_{\bf 3} } ={\bf a}_{\bf 3} } ( \cdot)$ means
$\E \big ( \cdot \mid \overrightarrow{{\bf \eta}_{\bf 3} } ={\bf a}_{\bf 3}  \big )$. Under the independence structure, we have
\[
 \E_{\overrightarrow{{\bf \eta}_{\bf 3} } ={\bf a}_{\bf 3} }\bigl[\tr\,|z\Id-{\mathbb D}^{(m)}  (t) |^{-2p}\bigr]
 =
 \E\bigl[\tr\,|z\Id- {\mathbb D}^{(m)}   ({\bf a}_{\bf 3}  )   (t) |^{-2p}\bigr].
\]
Once the boundary blocks are frozen, ${\mathbb D}^{(m)}(\overrightarrow{{\bf \eta}_{\bf 3}})(t)$
becomes a matrix whose remaining randomness comes from interior innovation blocks only; the latter are independent
across different triple-blocks, and therefore we can treat ${\mathbb D}^{(m)}({\bf a}_{\bf 3})(t)$ as a matrix with
independent block contributions conditionally on ${\bf a}_{\bf 3}$. Next, we note that, conditionally on
$\overrightarrow{{\bf \eta}_{\bf 3}}={\bf a}_{\bf 3}$, the random vectors
\[
\left(
\left(
A_i\big(
{\bf a}_{3(k-1)}(j),
{\bf a}_{3k}(j)
\big)
\right)_
{i\in[\![\,\max\{j,(3k-3)m+1\},\,n\wedge 3km\,]\!]}
\right)_{k,j}
\]
are independent and define
\[
\widetilde{{\mathbb X}}^{(m)} ( {\bf a}_{\bf 3}  )   
=
\frac{1}{\sqrt n} \sum_{k=1}^{k_{n,m}}  \sum_{i= (3k-3)m+1}^{3km}  \sum_{j =1}^i
\Big (  A_i  (   {\bf a}_{3(k-1) }(j) ,   {\bf a}_{3k} (j) )   \Big )   E_{i,j} \, .
\]
We will regard $\widetilde{{\mathbb X}}^{(m)}({\bf a}_{\bf 3})$ as the independent-block component of the model, while
the mean terms $M_{i,j}(\cdot,\cdot)$ are handled separately since they are deterministic under the conditioning. Now, we set
\[
\sigma^2 \big( \widetilde{{\mathbb X}}^{(m)} ( {\bf a}_{\bf 3}  )  \big)
=
\Big\Vert  \E  \big [ \big ( \widetilde{{\mathbb X}}^{(m)} ( {\bf a}_{\bf 3}   ) \big )^2 \big  ]\Big\Vert
\]
and 
\[
R \big( \widetilde{{\mathbb X}}^{(m)} ( {\bf a}_{\bf 3}  )  \big)
=
\frac{1}{\sqrt{n}}\Big  \Vert \max_{1 \leq k \leq k_{n,m} \atop 1 \leq j \leq n}  
\Big  \Vert
 \sum_{i= \max \{ j, (3k-3)m+1\}}^{n \wedge 3km}  A_i  (   {\bf a}_{3(k-1) }(j) ,   {\bf a}_{3k} (j) )   E_{i,j}
\Big  \Vert  \Big \Vert_{\infty}.
\]
These quantities are the natural analogues of the terms quantifying the variance and uniform block-size in \cite{brailovskayaVanHandel} for sums of independent random matrices. The term $\sigma^2$ is the variance proxy of the frozen centered matrix, while $R$ is the maximal operator norm of one independent block contribution.  Finally, setting $x_{i,j}$ to be the $(i,j)$-entry of the matrix $\sqrt{n}\widetilde{{\mathbb X}}^{(m)} ( {\bf a}_{\bf 3}  )$ and taking into account the symmetry of the matrix,
we note that if $a \leq b$ then
\[
 \big [ \big ( \widetilde{{\mathbb X}}^{(m)} ( {\bf a}_{\bf 3}   ) \big )^2 \big  ]_{a,b}
=
\frac{1}{n} \sum_{j=1}^a x_{a,j} x_{b,j}
+
\frac{1}{n} \sum_{j=a+1}^b x_{j,a} x_{b,j}
+
\frac{1}{n} \sum_{j=b+1}^n x_{j,a} x_{j,b} .
\]
Taking the expectation and using independence, we get 
\[
\E  \big [ \big ( \widetilde{{\mathbb X}}^{(m)} ( {\bf a}_{\bf 3}   ) \big )^2 \big  ]_{a,b}  = \frac{1}{n} \sum_{j=1}^{a\wedge b} \E (x_{a,j} x_{b,j} ) + \frac{1}{n}  \sum_{j=a+1}^n  \E (x^2_{j,a}  ) {\mathbf 1}_{a=b} .
\]
But $ \E (x_{a,j} x_{b,j} ) = 0$ if $a \in [3(k-1)m+1 , 3km]$ and $b  \in [3(k'-1)m+1 , 3k'm]$ with $k \neq k'$. In addition,   $ |  \E (x_{a,j} x_{b,j} ) | \leq 4M^2$. Hence, for $u=(u_1, \dots, u_n)^t$  such that $\Vert u \Vert=1$, 
\begin{align*}
 \Big \Vert \E  \big [ \big ( \widetilde{{\mathbb X}}^{(m)} ( {\bf a}_{\bf 3}   ) \big )^2 \big ] \cdot u   \Big \Vert^2 &=  \sum_{a=1}^n \Big (  \sum_{b=1}^n \E  \big [ \big ( \widetilde{{\mathbb X}}^{(m)} ( {\bf a}_{\bf 3}   ) \big )^2 \big  ]_{a,b}   u_b \Big )^2  \\
 &=  \sum_{k=1}^{k_{n,m}} \sum_{a=3(k-1)m+1}^{3km} \Big ( \sum_{b=3(k-1)m+1}^{3km}  \E  \big [ \big ( \widetilde{{\mathbb X}}^{(m)} ( {\bf a}_{\bf 3}   ) \big )^2 \big  ]_{a,b}   u_b \Big )^2   
\end{align*}
which implies that 
\beq \label{boundonsigma}
\sigma^2 ( \widetilde{{\mathbb X}}^{(m)} ( {\bf a}_{\bf 3}  )  )  \leq 12 m   M^2  .
\eeq
\noindent
On the other hand, for any $k \geq 1$ and any $j \geq 1$, 
\[
 \Big  \Vert \sum_{i= \max \{ j, (3k-3)m+1\}}^{n \wedge 3km}  A_i  (   {\bf a}_{3(k-1) }(j) ,   {\bf a}_{3k} (j) )   E_{i,j}   \Big  \Vert
 =
 \sup_{\Vert u \Vert=1}  \Big  \Vert  \sum_{i= \max \{ j, (3k-3)m+1\}}^{n \wedge 3km} A_i  (   {\bf a}_{3(k-1) }(j) ,   {\bf a}_{3k} (j) )   E_{i,j}   . u  \Big  \Vert .  
\]
For $i\neq j$ and for any vector
$u=(u_1,\dots,u_n)^t$, one has
$E_{i,j}\,u = u_j\,e_i + u_i\,e_j,$
where $(e_1,\dots,e_n)$ denotes the canonical basis of $\mathbb R^n$. 
Hence, for $u=(u_1, \dots, u_n)^t$ such that $\Vert u \Vert=1$,
\begin{multline*}
 \Big  \Vert  \sum_{i= \max \{ j, (3k-3)m+1\}}^{n \wedge 3km}  A_i  (   {\bf a}_{3(k-1) }(j) ,   {\bf a}_{3k} (j) )   E_{i,j}   . u  \Big  \Vert^2  \\
 \leq u_j^2  \sum_{i= \max \{ j, (3k-3)m+1\}}^{n \wedge 3km}  A^2_i  (   {\bf a}_{3(k-1) }(j) ,   {\bf a}_{3k} (j) )  + 
  \Big (  \sum_{i= \max \{ j, (3k-3)m+1\}}^{n \wedge 3km}   u_i A_i  (   {\bf a}_{3(k-1) }(j) ,   {\bf a}_{3k} (j) )  \Big)^2  
 \leq 6 m (2 M )^2  \, 
\end{multline*}
where the second inequality follows by Cauchy-Schwarz and the fact that 
$|A_i(\cdot,\cdot)|\le 2M$. Hence
\beq \label{boundonR}
R ( \widetilde{{\mathbb X}}^{(m)} ( {\bf a}_{\bf 3}  )  ) \leq 2 \sqrt{6} M \sqrt{\frac{m}{n}}  .
\eeq
\noindent
The quantity $R(\widetilde{{\mathbb X}}^{(m)}({\bf a}_{\bf 3}))$ controls the size of the largest
independent block contribution in operator norm, and \eqref{boundonR} shows that each such contribution is small
when $m\ll n$.

Following the lines of the proof of Proposition 6.10 in 
\cite{brailovskayaVanHandel} with ${\mathbb D}^{(m)}   ({\bf a}_{\bf 3}  )   (t)$ replacing $X(t)$,  taking into account the upper bounds
\eqref{boundonsigma} and \eqref{boundonR}, which do not depend  on  $ {\bf a}_{\bf 3}$ and integrating with respect
to the law of  $\overrightarrow{{\bf \eta}_{\bf 3}}$, we infer that, for any $z\in\C$ with $\Im m(z)>0$, there exists an
universal positive constant $C$  such that
\begin{multline*}
 \Big |  \frac{d}{dt} \E\bigl[\tr\,|z\Id-{\mathbb D}^{(m)}  (t) |^{-2p}\bigr]\ \Big |  
 \leq   C  M^3 m  \sqrt{\frac{m}{n}}   \frac{ p^3}{\Im m (z)^{4}} \max \left \{  \E\bigl[\tr\,|z\Id-{\mathbb D}^{(m)}  (t) |^{-2p}\bigr]^{1- \frac{1}{2p}} ,   \frac{  \big (  p M m^{1/2} n^{-1/2} \big )^{6p-3}  }{\Im m (z)^{8p-4}}  \right \} . 
\end{multline*}
Taking into account Lemma 6.6 in \cite{brailovskayaVanHandel}, we derive that there exists an universal  positive constant $C$ such that 
\beq \label{firstboundtrace}
Q_1 = \Bigl|\Bigl(\E\bigl[\tr\,|z\Id-{\mathbb X}^{(m)}|^{-2p}\bigr]\Bigr)^{\frac{1}{2p}}-\Bigl(\E\bigl[\tr\,|z\Id- {\mathbb Z}^{(m)} |^{-2p}\bigr]\Bigr)^{\frac{1}{2p}}\Bigr| \
\leq   C    \frac{ M^3 p^2}{\Im m (z)^{4}} \Big (  m \sqrt{\frac{m}{n}}   +   p   \frac{m  }{n}  \Big )  .
\eeq
\subsubsection{Step 3. Covariance matching and construction of ${\mathbb G}^{(m)}$.}\label{Subsection:Step3} We shall now give an upper bound of 
\beq \label{defofQ2}
Q_2= \Bigl|\Bigl(\E\bigl[\tr\,|z\Id-{\mathbb Z}^{(m)}|^{-2p}\bigr]\Bigr)^{\frac{1}{2p}}-\Bigl(\E\bigl[\tr\,|z\Id- {\mathbb G}^{(m)} |^{-2p}\bigr]\Bigr)^{\frac{1}{2p}}\Bigr| ,
\eeq
where $ {\mathbb G}^{(m)}$ is a Wigner-type matrix with Gaussian entries that we construct below.
The purpose of this step is to replace the conditionally \emph{block-Gaussian} matrix ${\mathbb Z}^{(m)}$, which is Gaussian only conditionally on the frozen boundary blocks but retains a nontrivial block structure, by a Gaussian Wigner-type matrix ${\mathbb G}^{(m)}$ whose  entries have the same covariance structure as the process $(X^{(m)}_{i,j})_{i,j \in \mathbb{Z}}$. With this aim, we first write
\begin{align*}
{\mathbb Z}^{(m)} &=   \frac{1}{\sqrt n} \sum_{k=1}^{k_{n,m}}  \sum_{i= (3k-3)m+1}^{3km}  \sum_{j =1}^i  \Big ( Z_k^{(1)} ( i,  \overrightarrow{{\bf \eta}_{3(k-1)} }(j)) + Z_k^{(2)} ( i,  \overrightarrow{{\bf \eta}_{3k} }(j))  \Big  )   E_{i,j}  \\
&=   \frac{1}{\sqrt n}  \sum_{i= 1}^{m}  \sum_{j =1}^i  Z_1^{(1)} ( i,  \overrightarrow{{\bf \eta}_{0} }(j))  E_{i,j}   +  \frac{1}{\sqrt n}  \sum_{i= 3 ( k_{n,m} -1)m+ 1}^{n}  \sum_{j =1}^i  Z_{k_{n,m}}^{(2)} ( i,  \overrightarrow{{\bf \eta}_{3 k_{n,m} } }(j))  E_{i,j}  \\
 & \quad+  \frac{1}{\sqrt n} \sum_{k=1}^{k_{n,m} -1} \sum_{j =1}^n  \Big \{  \sum_{i=  \max(j, (3k-3)m+1)}^{3km} 
Z_k^{(2)} ( i,  \overrightarrow{{\bf \eta}_{3k} }(j))  E_{i,j} +   \sum_{i=  \max(j, 3km+1)}^{3(k+1)m} 
Z_{k+1}^{(1)} ( i,  \overrightarrow{{\bf \eta}_{3k} }(j))   E_{i,j} \Big \}     \, ,
\end{align*}
where 
\[
Z_k^{(1)} ( i,  \overrightarrow{{\bf \eta}_{3(k-1)} }(j)) :=   \big [ \overrightarrow{G}^{(1)}_k   (   \overrightarrow{{\bf \eta}_{3(k-1)} }(j) )  \big ]_i +  M^{(3k-2)}_{i,j}  (    \overrightarrow{{\bf \eta}_{3(k-1)} }(j)  )  
\]
and
\[
Z_k^{(2)} ( i,  \overrightarrow{{\bf \eta}_{3k} }(j)) :=    \big [  \overrightarrow{G}^{(2)}_k   (   \overrightarrow{{\bf \eta}_{3k} }(j) )  \big ]_i +  M^{(3k)}_{i,j}  (    \overrightarrow{{\bf \eta}_{3k} }(j)  ) .
\]
Note that the first line is nothing but the definition of ${\mathbb Z}^{(m)}$ expressed in terms of the two Gaussian components associated
with the left and right boundaries of each triple-block. The grouping in the second line  allows us to represent ${\mathbb Z}^{(m)}$ as a sum of independent blocks indexed by the boundary blocks.  Now, we shall embed the relative terms in the same way into $\mathbb R^n$ with the correct support as in Step 1, by filling with zeros outside the corresponding sub-blocks. More precisely, we set for any $k=1, \dots, k_{m,n}-1$
\[
\begin{aligned}
    \overrightarrow{Z}_{k}^{(2)} (  \overrightarrow{{\bf \eta}_{3k} }(j))  &:=  \Big ( \overrightarrow{0}_{1,3(k-1)m}  ,  \big (  Z_{k}^{(2)} ( i,  \overrightarrow{{\bf \eta}_{3k} }(j)) \big )_{i \in [\! [  3(k-1)m+1, 3km ] \!]} ,   \overrightarrow{0}_{3km+1, n }  \Big ),
    \\ \overrightarrow{Z}_{k+1}^{(1)} (  \overrightarrow{{\bf \eta}_{3k} }(j))  &:=  \Big ( \overrightarrow{0}_{1,3km}  ,  \big (  Z_{k+1}^{(1)} ( i,  \overrightarrow{{\bf \eta}_{3k} }(j)) \big )_{i \in [\! [  3km+1, 3(k+1)m ] \!]} ,   \overrightarrow{0}_{3(k+1)m+1, n }  \Big ).
\end{aligned}
 \]
Finally, define for $k=0$ and $k=k_{m,n}$
\[
\begin{aligned}
\overrightarrow{Z}_{0} (  \overrightarrow{{\bf \eta}_{0} }(j)) &:=  \Big (   \big (   Z_1^{(1)} ( i,  \overrightarrow{{\bf \eta}_{0} }(j)) \big )_{i \in [\! [  1, m ] \!]} , \overrightarrow{0}_{m+1,n } \Big )   ,\\
  \overrightarrow{Z}_{k_{n,m} } (  \overrightarrow{{\bf \eta}_{3k_{n,m} } }(j)) &:=  \Big (  \overrightarrow{0}_{ 3 ( k_{n,m} -1)m ,n } ,   \big (    Z_{k_{n,m}}^{(2)} ( i,  \overrightarrow{{\bf \eta}_{3 k_{n,m} } }(j)) \big )_{i \in [\! [  3 ( k_{n,m} -1)m+ 1, n ] \!]} \Big ) .  
\end{aligned}
\] 
and, for any integer $k \in [1, k_{n,m} -1]$, set 
\[
\overrightarrow{Z_{k}} (  \overrightarrow{{\bf \eta}_{3k} }(j)):= \overrightarrow{Z}_{k}^{(2)} (  \overrightarrow{{\bf \eta}_{3k} }(j))  + \overrightarrow{Z}_{k+1}^{(1)} (  \overrightarrow{{\bf \eta}_{3k} }(j)) .
\]
For $1\le k\le k_{n,m}-1$, the boundary block $\overrightarrow{{\bf \eta}_{3k}}(j)$ influences the right outer part
of the $k$-th triple-block and the left outer part of the $(k+1)$-th triple-block. The definitions of $\overrightarrow{Z_0}$ and $\overrightarrow{Z_{k_{n,m}}}$ handle the
incomplete triple-blocks at the beginning and at the end of the index range. Clearly the random vectors
$ \big ( \overrightarrow{Z_{k}} (  \overrightarrow{{\bf \eta}_{3k} }(j)) \big )_{k \geq 0, j \geq 1} $
are centered and  independent since $((  \overrightarrow{{\bf \eta}_{3k} }(j))_{k \geq 0, j \geq 1} $ are.
In addition, if we denote by ${\bf Z}_n (j)$ the vector of ${\mathbb R}^n$ defined by 
\[
{\bf Z}_n (j)=   \sum_{k=0}^{k_{n,m}}\overrightarrow{Z_{k}} (  \overrightarrow{{\bf \eta}_{3k } }(j))  \, ,
\]
we have
\beq  \label{identitycov}
\Cov \Big (  \big [ {\bf Z}_n (j) \big ]_i ,  \big [ {\bf Z}_n (j) \big ]_{\ell}   \Big ) = 
\Cov \big ( X^{(m)}_{i,j} , X^{(m)}_{\ell,j}   \big )  .
\eeq
The proof of \eqref{identitycov} is non-trivial and is postponed to Appendix~\ref{appendix:covariance-identity} below. It shows that, for each fixed $j$, the Gaussian vector ${\bf Z}_n(j)$ reproduces the covariance structure of the original $m$-dependent column process $\big(X^{(m)}_{i,j}\big)_{1\le i\le n}$. This matching of second moments is the key input for constructing the Gaussian Wigner-type matrix ${\mathbb G}^{(m)}$ in the next step.
Let  now    $  \big ( G^{(m)}_{k,j}  \big )_{k \geq 0, j \geq 1} $  be independent  ${\mathbb R}^n $ Gaussian random vectors such that, for any $k$ and $j$  they are centered and have the same covariance structure as 
$\overrightarrow{Z_{k}} (  \overrightarrow{{\bf \eta}_{3k} }(j))$. Now setting $ g^{(m)}_{i,j}  :=  \big [  G^{(m)}_{k,j}  \big ]_i $ the $i$-th coordinate of the vector $  \big ( G^{(m)}_{k,j}  \big )$, we define 
\beq \label{defofGm}
{\mathbb G}^{(m)} =   \frac{1}{\sqrt n} \sum_{k=0}^{k_{n,m}}  \sum_{i= (3k-3)m+1}^{3(k+1)m}  \sum_{j =1}^i  g^{(m)}_{i,j}    E_{i,j}  . 
\eeq
To simplify the notations, we also write
\beq \label{defofZm}
{\mathbb Z}^{(m)} =   \frac{1}{\sqrt n} \sum_{k=0}^{k_{n,m}}  \sum_{i= (3k-3)m+1}^{3(k+1)m}  \sum_{j =1}^i  z^{(m)}_{i,j}    E_{i,j}  . 
\eeq
where for $k$ fixed 
\[
z^{(m)}_{i,j} =\big[\overrightarrow{Z_k}(\overrightarrow{{\bf \eta}_{3k}}(j))\big]_i= Z_k^{(2)} ( i,  \overrightarrow{{\bf \eta}_{3k} }(j))    {\bf 1}_{i \in [\![ \max(j, (3k-3)m+1), 3km ]\!]} +   
Z_{k+1}^{(1)} ( i,  \overrightarrow{{\bf \eta}_{3k} }(j))   {\bf 1}_{i \in [\![ \max(j, 3km+1),  3(k+1)m  ]\!]} . 
\]
Now we handle the quantity $Q_2$ defined in \eqref{defofQ2} with ${\mathbb G}^{(m)} $ constructed in \eqref{defofGm}.  Since the  random vectors $ \big ( \overrightarrow{Z_k} (  \overrightarrow{{\bf \eta}_{3k} }(j)) \big )_{k \geq 0, j \geq 1} $ are  independent, ${\mathbb Z}^{(m)}$ is a Wigner-type matrix such that the columns  constructed with the entries above the diagonal are independent and the rows are independent provided they are spaced by more than  $6m$.  To give a suitable upper bound for $Q_2$, we cannot use Theorem 6.8 in \cite{brailovskayaVanHandel} directly. Indeed, the underlying random variables are not bounded since their law conditionally to $  \overrightarrow{{\bf \eta}_{\bf 3} } = {\bf a}_{\bf 3} $ is the law of a Gaussian random variable plus a bounded random variable. We then have to adapt to our context the proof of their Theorem 6.8 and then modify the proof of their Proposition 6.10. With this aim, for any $t \in [0,1]$, let 
\[
{\mathbb Z}^{(m)} (t) = \sqrt{t}{\mathbb Z}^{(m)} + \sqrt{1-t}{\mathbb G}^{(m)}  \, .
\]
For any $q \geq 1$, define
\[
R_q ({\mathbb Z}^{(m)} ) = \frac{1}{\sqrt{n}} \Big (   \sum_{k=0}^{k_{n,m}} \sum_{j =1}^n  \E \Big [  \tr \Big | \sum_{i= \max( j , (3k-3)m+1)}^{3(k+1)m  }   z^{(m)}_{i,j}  E_{i,j} \Big |^q\Big ] \Big )^{1/q}
\]
and
\[
\sigma_q ({\mathbb Z}^{(m)} ) = \Big (  \tr \Big (  \E \big  ( {\mathbb Z}^{(m)} \big )^2  \Big )^{q/2} \Big )^{1/q}  . 
\]
Following the proof of  Proposition 6.10 in \cite{brailovskayaVanHandel} by taking $q \geq 6p$ and by using their Proposition 5.1 by modifying their  $p_j$'s by multiplying them by the constant $q/(q-k)$, we derive that 
\begin{multline} \label{inetracepropinter}
\left | \frac{d}{dt} \E\bigl[\tr\,|z\Id-Z^{(m)} (t)|^{-2p}\bigr] \right |  \leq C \frac{(16p)^{6p}}{ \Im m (z)^{8p}} \Big [  R_{6p} ({\mathbb Z}^{(m)} )   \Big ]^{6p} \\
+ \frac{1}{2} \sum_{k=3}^{6p-1} (16p)^k   \Big [  R_q ({\mathbb Z}^{(m)} )   \Big ]^{(k-2)q/(q-2)}   \Big [  \sigma_q ({\mathbb Z}^{(m)} )   \Big ]^{2 (q-k)/(q-2)}  \Big (  \E\bigl[\tr\,|z\Id-Z^{(m)} (t)|^{-2p-k}  \Big )^{(q-k)/q} \, .
\end{multline}
We note now that 
\[
 \sigma^2_q ({\mathbb Z}^{(m)} )  \leq  \sigma^2 ({\mathbb Z}^{(m)} ) :=   \big  \Vert  \E  \big [ \big ( {{\mathbb Z}}^{(m)}  \big )^2 \big  ] \big \Vert .
\]
By \eqref{identitycov},  $ \big| \Cov \big (  \big [ {\bf Z}_n (j) \big ]_a ,  \big [ {\bf Z}_n (j) \big ]_{b}   \big ) \big| \leq M   \delta(|b-a|){\bf 1}_{|b-a| \leq m} $. This entails that 
\[
 \Big|\E  \big [ \big ( {{\mathbb Z}}^{(m)}  \big )^2 \big  ]_{a,b}\Big| \leq M \delta(|b-a|){\bf 1}_{|b-a| \leq m}  \, .
\]
It follows that, for $u=(u_1, \dots, u_n)^t$  such that $\Vert u \Vert=1$
\begin{multline*}
 \Big \Vert \E  \big [ \big ( {{\mathbb Z}}^{(m)} \big )^2 \big ] \cdot u   \Big \Vert^2 =  \sum_{a=1}^n \Big (  \sum_{b=1}^n \E  \big [ \big ({{\mathbb Z}}^{(m)}\big )^2 \big  ]_{a,b}   u_b \Big )^2  \\
 \leq   M^2  \sum_{a=1}^n \Big ( \sum_{b=1 }^{n}   \delta(|b-a|)  |u_b | \Big )^2   \leq    M^2  \sum_{a=1}^n  \sum_{b=1 }^{n}    \sum_{b'=1 }^{n}   \delta(|b-a|)  \delta(|b'-a|)   |u_b | |u_{b'} |  \\
 \leq   M^2  \sum_{a=1}^n  \sum_{b=1 }^{n}    \sum_{b'=1 }^{n}   \delta(|b-a|)  \delta(|b'-a|) u^2_b  \leq   4  M^2  \Big ( \sum_{k =0}^{\infty}    \delta( k)  \Big )^2 \sum_{b=1 }^{n}   u^2_b  =     4  M^2  \Big ( \sum_{k =0}^{\infty}    \delta( k)  \Big )^2  \, ,
\end{multline*}
where we have used that $2|u_b | |u_{b'} | \leq u_b^2 + u_{b'}^2 $ for the third inequality. Consequently 
\beq  \label{boundforsigma2}
 \sigma^2 ({\mathbb Z}^{(m)} )  \leq 2 M \sum_{i \geq  0} \delta(i) < \infty \, .
\eeq
We give now an upper bound for $R_q ({\mathbb Z}^{(m)} ) $. Assume that $q $ is even (so $q= 2q'$) and let
\[
B := \Big| \sum_{i= \max( j , (3k-3)m+1 ) }^{3(k+1)m  }   z^{(m)}_{i,j}  E_{i,j} \Big|^2.
\]
Then $B$ is a real symmetric matrix of order $n$ that is also positive semidefinite. By the inequality for the trace of a product of a real symmetric  positive semidefinite matrix, we have 
\[
\Tr ( B^{q'} ) \leq \Big (  \Tr ( B) \Big )^{q'} .
\]
But
\[
 \Tr (B) \leq 2 \sum_{i= (3k-3)m+1}^{3(k+1)m  }   \big (  z^{(m)}_{i,j}  \big )^2   . 
\]
Hence, for any $q =2q'$ with $q'$ a positive integer, 
\[
 \E \Big [  \tr \Big | \sum_{i= \max( j , (3k-3)m+1)}^{3(k+1)m  }   z^{(m)}_{i,j}  E_{i,j} \Big |^q\Big ]   \leq n^{-1}\Big ( 2  \sum_{i= (3k-3)m+1}^{3(k+1)m  }   \Vert  z^{(m)}_{i,j}  \Vert_q^2 \Big )^{q'} . 
\]
But, for $ i \in [\![ \max(j, (3k-3)m+1), 3(k+1)m ]\!]$, 
\[
\Vert  z^{(m)}_{i,j}  \Vert_q \leq  \Vert Z_k^{(2)} ( i,  \overrightarrow{{\bf \eta}_{3k} }(j))   \Vert_q +  \Vert
Z_{k+1}^{(1)} ( i,  \overrightarrow{{\bf \eta}_{3k} }(j))  \Vert_q . 
\]
Next, by the definition of $Z_k^{(2)}$, it follows that
\[
 \Vert Z_k^{(2)} ( i,  \overrightarrow{{\bf \eta}_{3k} }(j))   \Vert_q \leq M +  \big\Vert \big [  \overrightarrow{G}^{(2)}_k   (   \overrightarrow{{\bf \eta}_{3k} }(j) )  \big ]_i \big\Vert_q \, ,
\]
and the law of $\big [  \overrightarrow{G}^{(2)}_k   (   \overrightarrow{{\bf \eta}_{3k} }(j) )  \big ]_i $ given $\overrightarrow{{\bf \eta}_{\bf 3} }={\bf a}_{\bf 3}$ is the law of a centered Gaussian random variable with variance, say $c^2$, 
which is bounded by $M^2$.  Now, if $N \sim {\mathcal N} ( 0,c^2)$,
\[
\Vert N \Vert^q_q  = \frac{c^{q}  q ! }{ 2^{q'} q' !}\leq   c^q e^{-q/2}  q^{1+q/2} \, , 
\]
where we used the fact that $e n^n e^{-n} \leq n! \leq  e n^n e^{-n}  n $. It follows that 
\[
 \Vert Z_k^{(2)} ( i,  \overrightarrow{{\bf \eta}_{3k} }(j))   \Vert_q \leq M +  M e^{-1/2}  q^{(q+2)/(2q)}   \, .
\]
A similar bound is valid for $ \Vert
Z_{k+1}^{(1)} ( i,  \overrightarrow{{\bf \eta}_{3k} }(j))  \Vert_q$. So, 
for $ i \in [\![ \max(j, (3k-3)m+1), 3(k+1)m ]\!]$, 
\[
\Vert  z^{(m)}_{i,j}  \Vert_q \leq     2 M +   2  e^{-1/2}  M  q^{(q+2)/(2q)}   \leq 4M  q^{(q+2)/(2q)}  . 
\]
All the above computations imply that, for any positive and even integer $q$, there exists a positive  universal constant $C$ such that 
\beq \label{boundonRq}
R_q ({\mathbb Z}^{(m)} ) \leq C  M     q^{1/2}  \Big ( \frac{m}{n} \Big )^{ (q-2)/(2q)}   \, .
\eeq
Let now 
\[
A_q(m,n) =   C  (M  +1)  \Big ( \frac{n}{m} \Big )^{\frac1q }   \Big ( \frac{m q}{n} \Big )^{\frac12 } .
\]
Following the lines of the proof of Proposition 6.10 in \cite{brailovskayaVanHandel} and using the upper bounds \eqref{boundforsigma2} and \eqref {boundonRq},  the following intermediate proposition holds:
\begin{proposition} \label{propcomparaisonmomentspinter}Let $p\in\mathbb{N}^*$ and $q=12p^2$. Let $z\in\C$ with $\Im m(z) \in ]0,1]$. Define $Q_2$ by \eqref{defofQ2}. Then
\[
Q_2
 \ll  \frac{1 }{\Im m (z)^{4}}  p^2  A^{q/(q-2)}_q(m,n)  \Big ( 1 + p  A^{2q/(q-2)}_q(m,n) \Big )  +   \frac{p^3 }{\Im m (z)^{4}} A^{-\frac{q}{(q-2)(2p-1)}}_q(m,n) A^{6p/(2p-1)}_{6p}(m,n)   \, .
\]
\end{proposition}
\noindent{\textbf{Proof of Proposition \ref{propcomparaisonmomentspinter}.}}  We start from  \eqref{inetracepropinter} and note that, for a nonnegative real number $\kappa$ such that $\kappa \leq  2p$ we have 
\beq \label{Upperboundjensen}
 \Big (  \E\bigl[\tr\,|z\Id-\mathbb Z^{(m)}(t)|^{-2p-k} \bigr] \Big )^{(q-k)/q} \leq \Big (  \frac{1}{\Im m (z)^{ \kappa  + k }} \Big )^{(q-k)/q}   \Big (  \E\bigl[\tr\,|z\Id-\mathbb Z^{(m)}(t)|^{-2p} \bigr]  \Big )^{\beta (q-k) /q } \, , 
\eeq
where $\beta =1 -  \frac { \kappa }{2p}  $.  This upper bound follows from the fact that 
\[
\Vert (z\Id-\mathbb Z^{(m)}(t))^{-1} \Vert  \leq  ( \Im m (z) )^{-1}
\]
and that, by comparison of the Schatten norms and Jensen's inequality, 
\[
 \E\bigl[\tr\,|z\Id-\mathbb Z^{(m)}(t)|^{-2p+ \kappa} \bigr] \leq \Big (  \E\bigl[\tr\,|z\Id-\mathbb Z^{(m)}(t)|^{-2p}  \bigr]  \Big )^{\beta } \, .
\]
Next, the upper bound \eqref{Upperboundjensen} together with \eqref{boundforsigma2} and the fact that $\Im m (z) \in ]0,1]$ imply
\begin{multline}\label{eq:derivative-start-Q2} \left | \frac{d}{dt} \E\bigl[\tr\,|z\Id-\mathbb Z^{(m)}(t)|^{-2p} \bigr] \right |  \leq C \frac{(16p)^{6p}}{\Im m (z)^{8p}} R_{6p}^{6p} ({\mathbb Z}^{(m)} ) \\
+ \frac{C}{2} \sum_{k=3}^{6p-1} (16p)^k   \Big [  R_q ({\mathbb Z}^{(m)} )   \Big ]^{ \frac{(k-2)q}{q-2}}     \frac{1}{\Im m (z)^{ \kappa + k }} \Big (  \E\bigl[\tr\,|z\Id-\mathbb Z^{(m)}(t)|^{-2p} \bigr]  \Big )^{\frac{\beta (q-k)}{q}} \, .
\end{multline}
where $C$ is a positive constant. We continue by setting
\[
T_m(t):=\E\bigl[\tr\,|z\Id-Z^{(m)}(t)|^{-2p}\bigr].
\]
Thus \eqref{eq:derivative-start-Q2} yields, for each fixed $\kappa\in[0,2p]$,
\[
\left|
\frac{d}{dt}T_m(t)
\right|
\le
C\,\frac{(16p)^{6p}}{\Im m (z)^{8p}}
A_{6p}^{6p}(m,n)
+
\frac{C}{2}
\sum_{k=3}^{6p-1}
(16p)^k
\Big[A_q(m,n)\Big]^{\frac{(k-2)q}{q-2}}
\frac{1}{\Im m (z)^{k+\kappa}}
\Big(T_m(t)\Big)^{\beta\frac{q-k}{q}} .
\]
We now use the endpoint convexity argument. For fixed $\kappa$ and $\beta$, the logarithm of the
$k$-th summand in the last sum is affine in $k$. Hence the
summand is bounded by the maximum of its endpoint values at $k=3$ and $k=6p$. After absorbing the
number of terms in the sum into the constants
, we get
\begin{multline}\label{eq:endpoint-kappa-Q2}
\frac{1}{2}
\sum_{k=3}^{6p-1}
(16p)^k
\Big[A_q(m,n)\Big]^{\frac{(k-2)q}{q-2}}
\frac{1}{\Im m (z)^{k+\kappa}}
\Big(T_m(t)\Big)^{\beta\frac{q-k}{q}}
\\
\le
\frac{1}{8}
\max\Bigg\{
(32p)^3
\Big[A_q(m,n)\Big]^{\frac{q}{q-2}}
\frac{1}{\Im m (z)^{3+\kappa}}
\Big(T_m(t)\Big)^{\beta\frac{q-3}{q}},
\\
(32p)^{6p}
\Big[A_q(m,n)\Big]^{\frac{(6p-2)q}{q-2}}
\frac{1}{\Im m (z)^{6p+\kappa}}
\Big(T_m(t)\Big)^{\beta\frac{q-6p}{q}}
\Bigg\}.
\end{multline}
On the other hand, selecting 
\[
\kappa = \frac{q -6p}{q-3} = \frac{2p}{2p+1} 
\]
 then $\beta  (q-3)/q= 1-1/(2p)$, and hence 
 we get \begin{multline}\label{eq:derivative-main-Q2}
\left|
\frac{d}{dt}T_m(t)
\right|
\le
C\,\frac{(16p)^{6p}}{\Im m (z)^{8p}}
A_{6p}^{6p}(m,n)
+
C\,\frac{p^3 A_q^{\frac{q}{q-2}}(m,n)}{\Im m (z)^4}
\max\left\{
\Big(T_m(t)\Big)^{1-\frac{1}{2p}},
\left(
32p\,A_q^{\frac{q}{q-2}}(m,n)
\right)^{6p-3}
\frac{1}{\Im m (z)^{8p-4}}
\right\}.
\end{multline}
It remains to incorporate the first term into the preceding differential inequality and use Lemma 6.6 in \cite{brailovskayaVanHandel} to end the proof.

\subsubsection{Step 4. Comparison between $\mathbb G_n^{(m)}$ and $\mathbb G_n$}
\label{Section:Step4}

We now compare the Gaussian matrix $\mathbb G_n^{(m)}$, whose covariance structure is that of the finite-memory process
$(X^{(m)}_{k,\ell})_{k,\ell \in \mathbb{Z}}$, with the Gaussian comparison matrix $\mathbb G_n$, whose covariance structure is that of the original process
$(X_{k,\ell})_{k,\ell \in \mathbb{Z}}$. We aim to find an upper bound for the quantity
\[ 
Q_3 = \Bigl|\Bigl(\E\bigl[\tr\,|z\Id-{\mathbb G}^{(m)}|^{-2p}\bigr]\Bigr)^{\frac{1}{2p}}-\Bigl(\E\bigl[\tr\,|z\Id- {\mathbb G}|^{-2p}\bigr]\Bigr)^{\frac{1}{2p}}\Bigr| .
\]
We shall interpolate between the moments of the resolvents of $ {\mathbb G}^{(m)}$ and ${\mathbb G}$ in a similar fashion as before. Since these two matrices are both centered Gaussian matrices and differ only through the covariance structure of their entries, we use a Gaussian covariance interpolation between the two Gaussian fields. Differentiating the corresponding resolvent moment along this interpolation and applying Gaussian integration by parts reduces the comparison to a perturbative estimate involving only the difference of the two covariance kernels. More precisely, we obtain:
\beq\label{eq:main-bound}
\Bigl|\Bigl(\E\bigl[\tr\,|z\Id-G|^{-2p}\bigr]\Bigr)^{\frac{1}{2p}}-\Bigl(\E\bigl[\tr\,|z\Id-G^{(m)}|^{-2p}\bigr]\Bigr)^{\frac{1}{2p}}\Bigr|
\ \le\
\frac{2p}{n\,\Im m (z)^{3}}\ \Delta_{m,n} \, , 
\eeq
with $\Delta_{m,n}$ being the total covariance discrepancy
\[
\Delta_{m,n}
:=
\sum_{j=1}^n
\sum_{i=j}^n
\sum_{k=j}^n
\big|
\Cov(X_{i,j},X_{k,j})
-
\Cov(X^{(m)}_{i,j},X^{(m)}_{k,j})
\big|.
\]
Since the variables are centered, we can write
\[
\Cov(X_{i,j},X_{k,j})
-
\Cov(X^{(m)}_{i,j},X^{(m)}_{k,j})
=
\E\big[(X_{i,j}-X^{(m)}_{i,j})X_{k,j}\big]
+
\E\big[X^{(m)}_{i,j}(X_{k,j}-X^{(m)}_{k,j})\big].
\]
By Cauchy-Schwarz, stationarity, and the contraction property of conditional expectation in $L^2$,
\[
\begin{aligned}
\big|
\Cov(X_{i,j},X_{k,j})
-
\Cov(X^{(m)}_{i,j},X^{(m)}_{k,j})
\big|
&\le
\|X_{i,j}-X^{(m)}_{i,j}\|_2\,\|X_{k,j}\|_2
+
\|X^{(m)}_{i,j}\|_2\,\|X_{k,j}-X^{(m)}_{k,j}\|_2
\le
2\|X_{0,0}\|_2\,\delta(m).
\end{aligned}
\]
This yields that 
\[
\Delta_{m,n}  \leq 2 n^3  \Vert X_{0,0} \Vert_2  \delta (m)  \, ,
\]
and hence the upper bound 
\[
Q_3 \leq \frac{2p n^2}{\Im m (z)^{3}}  \Vert X_{0,0} \Vert_2  \delta (m) \, .
\]

\subsubsection{Step 5. Conclusion of the proof of Proposition~\ref{propcomparaisonmomentsp}}

We now collect the bounds for $Q_1$, $Q_2$ and $Q_3$ obtained in Steps~2, 3 and~4. It remains only to simplify the powers of
$A_q(m,n)$ appearing in the estimate for $Q_2$. Note that for any $r \geq 4$,
\[
\frac{r}{2(r-2)} \leq \frac12 + \frac{1}{\log (q) } \, .
\]
Hence
\[
r^{\frac{r}{2(r-2)}} \leq  e+ \sqrt{r} \text{ and then } r^{\frac{r}{2(r-2)}} \leq  C \sqrt{r}\, .
\]
Since $q \geq 12$, it follows that 
\[
A^{q/(q-2)}_q(m,n) \ll  ( 1 + M)^{6/5} \sqrt{q}  \sqrt{\frac{m}{n}} \, .
\]
On another hand, for $p \geq 2$, 
\[
A^{-\frac{q}{(q-2)(2p-1)}}_q(m,n) A^{6p/(2p-1)}_{6p}(m,n)  \leq (1+M)^4  p^2  \Big ( \frac{m}{n} \Big )^{3/2}  \Big ( \frac{n}{m} \Big )^{ 1/(2p -1) }  \, .
\]
Combining the estimates for $Q_1$, $Q_2$ and $Q_3$, we obtain the desired stated bounds.

\subsection{Proofs of the remaining intermediate results.}
\paragraph{Proof of Lemma \ref{ineconcentration1}.} 
We start by noticing that 
\[
\Vert {\mathbb X}  - {\mathbb X} ^{(m)} \Vert = \sup_{u,v \in {\mathbb S}^{n-1} } \Big |  \langle u , ( {\mathbb X}  - {\mathbb X}^{(m)}) v \rangle \Big |   \, .
\]
But, setting, for $i,j  \in \{1, \dots, n \}$, 
\[
a_{i,j} = u_i v_j + u_j v_i \text{ if } i\neq j \text{ and } a_{ii} = u_i v_i \, ,
\]
and 
$a_{i,j} =0$ otherwise, 
we have 
\[
 \Big |  \langle u , ( {\mathbb X}  - {\mathbb X}^{(m)}) v \rangle \Big | \leq  \frac{1}{ \sqrt{n}} \sum_{i=1}^n \sum_{j=1}^i  |X_{i,j}-X^{(m)}_{i,j}  \vert |a_{i,j} |  \leq  \frac{2}{ \sqrt{n}} \delta(m)  \sum_{i=1}^n |u_i|  \sum_{j=1}^n |v_j|  \, .
\]
Hence
\beq \label{comparisonXXm}
\Vert {\mathbb X}  - {\mathbb X} ^{(m)} \Vert    \leq  2 \sqrt{n} \delta (m) \, .
\eeq
Next, we proceed as in the proof of Lemma 7.2 in \cite{brailovskayaVanHandel}. Let us write
\[
 \langle u , {\mathbb X}^{(m)} v \rangle = n^{-1/2} \sum_{i=1}^n \sum_{j=1}^i X^{(m)}_{i,j} a_{i,j} .
 \]
 Setting ${\tilde X}^{(m)}_{i,j} =0$ if $i >n$ or if $j >i$, we start with the following decomposition 
\[
 \sum_{i=1}^n \sum_{j=1}^i X^{(m)}_{i,j} a_{i,j}  =   \sum_{\ell=1}^{[2^{-1} k_{n,m}] + 1} \sum_{j=1}^n \sum_{i = (2 \ell -1)m +1}^{2 \ell m}  {\tilde X}^{(m)}_{i,j} a_{i,j} +  \sum_{\ell=0}^{[2^{-1} k_{n,m}] }  \sum_{j=1}^n \sum_{i = 2 \ell m +1}^{(2 \ell +1)  m} {\tilde X}^{(m)}_{i,j} a_{i,j} \, ,
\]
where $k_{n,m} = [n/m] +1$.  Let $B_{\ell, j} = \sum_{i = ( \ell -1)m +1}^{ \ell m}  {\tilde X}^{m}_{i,j}a_{i,j} $. We then have
\[
 \sum_{i=1}^n \sum_{j=1}^i X^{(m)}_{i,j} a_{i,j}   =  \sum_{\ell=1}^{[2^{-1} k_{n,m}] + 1}  \sum_{j=1}^n  B_{2\ell, j} + \sum_{\ell=0}^{[2^{-1} k_{n,m}] }  \sum_{j=1}^n  B_{2\ell +1 , j}  \, .
\]
But the random variables $ ( B_{2\ell, j} ) _{\ell, j \geq 1}$ are independent as well as the random variables $ ( B_{2\ell +1, j} ) _{\ell \geq 0, j \geq 1}$. Note that 
\[
\Vert  B_{2\ell, j} \Vert_{\infty} \leq M  \sum_{i = ( 2\ell -1)m +1}^{ 2\ell m}  |a_{i,j} | \leq M \sqrt{m}  \sqrt{ \sum_{i = (2 \ell -1)m +1}^{ 2\ell m}  a^2_{i,j} }  \, .
\]
Hence, by using Hoeffding's inequality, we get that, for any $y  >0$ and any $u , v $ such that $\Vert u \Vert= \Vert v \Vert=1$, 
\[
\Prob \Big ( \Big |  n^{-1/2}  \sum_{\ell=1}^{[2^{-1} k_{n,m}] + 1}  \sum_{j=1}^n  B_{2\ell, j}    \Big | \geq   y  \Big )  \leq  2 \exp \left ( -  \frac{y^2 n }{ 8  m M^2 }  \right )  ,
\]
since 
\[
 \sum_{\ell=1}^{ k_{n,m}}  \sum_{j=1}^n    \sum_{i = ( \ell -1)m +1}^{ \ell m}  a^2_{i,j}  =  \sum_{i=1}^n \sum_{j=1}^n a^2_{i,j} \leq 4  \Vert u \Vert^2  \Vert v \Vert^2=4  \, .
\]
Hence, for any $ x>0$, 
\[
\Prob \Big ( \Big |  n^{-1/2}  \sum_{\ell=1}^{[2^{-1} k_{n,m}] + 1}  \sum_{j=1}^n  B_{2\ell, j}    \Big | \geq     \frac{ 2 \sqrt{2 m }M  }{ \sqrt n} \sqrt{ x}   \Big )  \leq  2 e^{-x}  \, .
\]

We can get a similar for the deviation probability of the odd blocks $ \sum_{\ell=0}^{[2^{-1} k_{n,m}] }  \sum_{j=1}^n  B_{2\ell +1 , j} $.  So, overall, for a universal constant $\alpha$ large enough, we get that 
\[
\Prob \Big ( \Vert   {\mathbb X}  \Vert  \geq   2 \sqrt{n} \delta (m) +   16 \sqrt{2  }   M{  \sqrt  \frac{m}{n}} \sqrt{\alpha n + t }    \Big )  \leq  e^{-t} \, .
\]
We end the proof as the one given in Lemma 7.2 in \cite{brailovskayaVanHandel}.

\paragraph{Proof of Lemma \ref{lma73}}
Using the identity $A^{-1} - B^{-1}= A^{-1}(B-A)B^{-1}$ together with the fact that $ \Vert  ( z \1 - A )^{-1 } \Vert \leq   \frac{1 }{ |\Im m (z)|} $ and inequality \eqref{comparisonXXm},  we have
\beq \label{comparisonResXXm}
  \Bigl  |   \Vert  ( z I_n - {\mathbb X} )^{-1 } \Vert  -   \Vert  ( z I_n - {\mathbb X}^{(m)} )^{-1 } \Vert  \Bigr |  \leq  \frac{1 }{ \Im m (z)^2} \Vert {\mathbb X}  -  {\mathbb X}^{(m)} \Vert
  \leq  2   \frac{\sqrt{n} \delta(m) }{ \Im m (z)^2} . 
\eeq
Hence, by Markov's inequality, 
\begin{multline*}
\Prob  \Bigl (   \Vert  ( z I_n - {\mathbb X} )^{-1 } \Vert  \geq  2   \frac{\sqrt{n} \delta(m) }{ \Im m (z)^2}  +   e  \E^{1/(2p)} \big (   \Vert  ( z I_n - {\mathbb X} ^{(m)} )^{-1 } \Vert^{2p} \big )  \Bigr ) \\
\leq  \Prob  \Bigl (   \Vert  ( z I_n - {\mathbb X}^{(m)} )^{-1 }  \Vert  \geq   e  \E^{1/(2p)} \big (   \Vert  ( z I_n - {\mathbb X} ^{(m)} )^{-1 } \Vert^{2p} \big )  \Bigr ) 
\leq e^{-2p} .
\end{multline*}
Next, we continue the proof as the one of Lemma 7.3  in \cite{brailovskayaVanHandel} by using our Proposition \ref{propcomparaisonmomentsp} instead of their Theorem 6.8.  

\paragraph{Proof of Theorem \ref{spectrumonedirection}}
First we choose $m = c^{-1} ( 2+\kappa) \log (n)$ with $\kappa>0$. The proof follows the lines of the proof of Proposition 7.4   in \cite{brailovskayaVanHandel} using Lemmas \ref{ineconcentration1} and  \ref{lma73} instead of their lemmas 7.2 and 7.3.

\subsection{Proof of Proposition \ref{prop:lower-bound-X}}

The proof is based on the same scheme as that of \cite[Proposition~7.7]{brailovskayaVanHandel}. The only essential modification is that Proposition~\ref{prop:X-from-Xm} replaces \cite[Proposition~5.5]{brailovskayaVanHandel}. We first derive a comparison between the resolvents of $\mathbb X$ and $\mathbb G$, then discretize the set $\operatorname{sp}(\mathbb G)+\ii\varepsilon$, and finally invoke \cite[Lemma~7.1]{brailovskayaVanHandel} to conclude.

\medskip
\noindent
\textbf{Resolvent comparison:}
Fix $z\in\Cbb$ with $\Im m(z)\in(0,1]$, and let $p\in\mathbb N^*$.
By Markov's inequality,
$$
\Pbb\Bigl( \|(zI_n-\mathbb G)^{-1} \| \ge e\,\E\,[\|(zI_n-\mathbb G)^{-1}\|^{2p}]^{\frac{1}{2p}}\Bigr)\le e^{-2p}.
$$
Using this inequality along with Proposition \ref{propcomparaisonmomentsp} yields
$$
\E\,[\|(zI_n-\mathbb G)^{-1}\|^{2p}]^{\frac{1}{2p}}
\le
n^{\frac{1}{2p}}\E\,[\|(zI_n-\mathbb X^{(m)})^{-1}\|^{2p}]^{\frac{1}{2p}}
+
n^{\frac{1}{2p}}U_{n,m}(p,z).
$$
with $U_{n,m}(p,z)$ defined by \eqref{defUmn}. Now, by inequality \eqref{comparisonResXXm}, 
$$
\E\,[\|(zI_n-\mathbb X^{(m)})^{-1}\|^{2p}]^{\frac{1}{2p}}
\le
\E\,[\|(zI_n-\mathbb X)^{-1}\|^{2p}]^{\frac{1}{2p}}
+
\frac{2\sqrt n\,\delta(m)}{\Im m(z)^2}.
$$
In the spirit of the proof of \cite[Lemma 7.5]{brailovskayaVanHandel}, we use \cite[Theorem 2.3]{BLM} in order to control $\E\,[\|(zI_n-\mathbb X)^{-1}\|^{2p}]^{\frac{1}{2p}}$ by the concentration estimate in Proposition \ref{prop:X-from-Xm}, there exists a universal constant $C_1 >0$ such that
$$
\E\,[\|(zI_n-\mathbb X)^{-1}\|^{2p}]^{\frac{1}{2p}}
\le
\E \|(zI_n-\mathbb X)^{-1}\|
+
C_1\xi_{n,m}(p,z),
$$
with $\xi_{n,m}(p,z)$ defined in Proposition \ref{prop:X-from-Xm}. Still in the spirit of \cite{brailovskayaVanHandel}, combining the previous bounds gives
$$
\E\,[\|(zI_n-\mathbb G)^{-1}\|^{2p}]^{\frac{1}{2p}}
\le
n^{1/2p}\E \|(zI_n-\mathbb X)^{-1}\|
+
C_2 n^{1/2p}\left(
\frac{\alpha_{1,n,m}(p)}{\Im m(z)^2}
+
\frac{\alpha_{2,n,m}(p)}{\Im m(z)^3}
+
\frac{\alpha_{3,n,m}(p)}{\Im m(z)^4}
\right),
$$
with $C_2 > 0$ a universal constant and $\alpha_{1,n,m},\alpha_{2,n,m},\alpha_{3,n,m}$ exactly as in the statement.
 A second application of Proposition \ref{prop:X-from-Xm} gives
$$
\Pbb\Bigl(
\E \|(zI_n-\mathbb X)^{-1}\|\ge \|(zI_n-\mathbb X)^{-1}\|+
C_3\xi_{n,m}(p,z)\Bigr)\le  4e^{-c p},
$$
for $C_3,c > 0$ universal constants. Note that $n^{\frac{1}{2p}},(n/m)^{\frac{1}{2p-1}} \leq e^{2}$ for $p\geq \log (n)$; thus combining the above estimates, we obtain
$$
\Pbb\Bigl(
\|(zI_n-\mathbb G)^{-1}\|\ge
C\left\{
\|(zI_n-\mathbb X)^{-1}\|
+
\frac{\alpha_{1,n,m}(p)}{\Im m(z)^2}
+
\frac{\alpha_{2,n,m}(p)}{\Im m(z)^3}
+
\frac{\alpha_{3,n,m}(p)}{\Im m(z)^4}
\right\}
\Bigr)
\le  e^{-2p} + 4e^{-cp},
$$
for every $p\ge \log n$. Finally, by setting $t = \lceil Lp\rceil$ for a sufficiently large constant $L$ we get 
$$
\Pbb\Bigl(
\|(zI_n-\mathbb G)^{-1}\|\ge
C\left\{
\|(zI_n-\mathbb X)^{-1}\|
+
\frac{\alpha_{1,n,m}(t)}{\Im m(z)^2}
+
\frac{\alpha_{2,n,m}(t)}{\Im m(z)^3}
+
\frac{\alpha_{3,n,m}(t)}{\Im m(z)^4}
\right\}
\Bigr)
\le  5e^{-t},
$$
for every $t\ge \log n$.\\

\medskip
\noindent
\textbf{Uniform resolvent inequality on $\operatorname{sp}(\mathbb G)+i\varepsilon$:} We now follow exactly the same discretization argument as in the proof of Proposition~7.7 of
\cite{brailovskayaVanHandel}. One first uses the Gaussian spectral localization lemma
(their Lemma~7.2) to replace $\operatorname{sp}(\mathbb G)$ by a deterministic set $\Omega_x$ with probability
$1-e^{-x}$, then chooses an $\varepsilon$-net $N_x\subset \Omega_x$, and finally uses
$$
A^{-1}-B^{-1}=A^{-1}(B-A)B^{-1}
$$
to pass from the grid to all points in $\operatorname{sp}(\mathbb G)+i\varepsilon$. This yields
$$
\Pbb\Bigl(
\|(zI_n-\mathbb G)^{-1}\|\le
C\left\{
\|(zI_n-\mathbb X)^{-1}\|
+
\frac{\alpha_{1,n,m}(t)}{\varepsilon^2}
+
\frac{\alpha_{2,n,m}(t)}{\varepsilon^3}
+
\frac{\alpha_{3,n,m}(t)}{\varepsilon^4}
\right\}
\ \text{for all }z\in \operatorname{sp}(\mathbb G)+i\varepsilon
\Bigr)
\ge 1-n e^{-t},
$$
for all $t\ge \log n$, after enlarging the constants if necessary. Note that this inequality also holds when $t<\log n$, since $1-ne^{-t} <0$ in this case.

\medskip
\noindent
\textbf{Application of the deterministic spectral lemma:} Applying Lemma~7.1 of \cite{brailovskayaVanHandel} with $A=\mathbb G$ and $B=\mathbb X$, we obtain for all $t \geq 0$
$$
\Pbb\Bigl(
\operatorname{sp}(\mathbb G)\subset \operatorname{sp}(\mathbb X)+C\,\varepsilon_{n,m}(t)\,[-1,1]
\Bigr)
\ge 1-n e^{-t},
$$
where
$$
\varepsilon_{n,m}(t)
=
\alpha_{1,n,m}(t)+\alpha_{2,n,m}(t)^{1/2}+\alpha_{3,n,m}(t)^{1/3}.
$$
This proves the proposition.

\paragraph{Proof of Theorem \ref{spectrumotherdirection}}
The theorem follows by applying Proposition \ref{prop:lower-bound-X}, selecting, once again,   $m = c^{-1} ( 2+\kappa) \log (n)$ with $\kappa>0$. 

\paragraph{Proof of Theorem \ref{thm:Hausdorff distance}}
The theorem follows by applying Theorems \ref{spectrumonedirection} and \ref{spectrumotherdirection} (see the proof of \cite[Theorem 2.6]{brailovskayaVanHandel} for more details). 

\appendix

\section{Proof of the covariance identity \eqref{identitycov}}
\label{appendix:covariance-identity}

Let us verify \eqref{identitycov} for a few couples of integers $(i,\ell) $ in $[\! [ 1, n ] \! ]^2$. We start by noticing that by construction 
\[
\Cov \Big (  \big [ {\bf Z}_n (j) \big ]_i ,  \big [ {\bf Z}_n (j) \big ]_{\ell}   \Big )  = 0 \text{ if } |i-\ell | >m \, .
\]
Next,  assume  that $(i, \ell)$ is in $[\! [ 3km+1, (3k+1)m ] \! ]^2$ for some non negative $k$.  We have 
\begin{multline*}
\Cov \Big (  \big [ {\bf Z}_n (j) \big ]_i ,  \big [ {\bf Z}_n (j) \big ]_{\ell}   \Big )  = \E \Big (  \big [ \overrightarrow{Z}_{k+1}^{(1)} (  \overrightarrow{{\bf \eta}_{3k} }(j)) \big ]_i   \big [ \overrightarrow{Z}_{k+1}^{(1)} 
(  \overrightarrow{{\bf \eta}_{3k} }(j)) \big ]_{\ell} \Big )
\\
= \E \left [  E_{ \overrightarrow{{\bf \eta}_{\bf 3}} = {\bf a }_{\bf 3} } \Big (  \big [ \overrightarrow{Z}_{k+1}^{(1)} (  \overrightarrow{{\bf \eta}_{3k} }(j)) \big ]_i   \big [ \overrightarrow{Z}_{k+1}^{(1)} (  \overrightarrow{{\bf \eta}_{3k} }(j)) \big ]_{\ell} \Big ) \right ] \\
=   \int  \E \Big [   \big [ \overrightarrow{G}^{(1)}_{k+1}   (    {\bf a}_{3k } (j)    )  \big ]_i  \big [ \overrightarrow{G}^{(1)}_{k+1}   (    {\bf a}_{3k } (j)    )  \big ]_\ell   \Big ]  dP_{\overrightarrow{{\bf \eta}_{\bf 3}} }  (  {\bf a }_{\bf 3}  )  \\
=  \int  \E \Big [  Y^{(3k+1)}_{i,j}   ( {\bf a}_{3k } (j)  )   Y^{(3k+1)}_{\ell,j}   ( {\bf a}_{3k } (j)  )   \Big ]  dP_{\overrightarrow{{\bf \eta}_{\bf 3}} }  (  {\bf a }_{\bf 3}  )  
=  \E  \big ( X^{(m)}_{i,j}  X^{(m)}_{\ell,j}   \big ) \, .
  \end{multline*}
Consider now the case where $  3km+1 \leq i \leq  (3k+1)m$ and   $(3k+1)m  + 1 \leq  \ell \leq  (3k+2)m$ for some non negative $k$.  We have 
\begin{multline*}
\Cov \Big (  \big [ {\bf Z}_n (j) \big ]_i ,  \big [ {\bf Z}_n (j) \big ]_{\ell}   \Big )  = \E \Big (  \big [ \overrightarrow{Z}_{k+1}^{(1)} (  \overrightarrow{{\bf \eta}_{3k} }(j))  \big ]_i   \big [  \overrightarrow{Z}_{k+1}^{(1)} (  \overrightarrow{{\bf \eta}_{3k} }(j)) +  \overrightarrow{Z}_{k+1}^{(2)} 
(  \overrightarrow{{\bf \eta}_{3(k+1)} }(j)) \big ]_{\ell} \Big ) \\
=  \E \Big (  \big [  \overrightarrow{Z}_{k+1}^{(1)} (  \overrightarrow{{\bf \eta}_{3k} }(j)) \big ]_i   \big [ \overrightarrow{Z}_{k+1}^{(1)} 
(  \overrightarrow{{\bf \eta}_{3k} }(j)) \big ]_{\ell} \Big ) \, ,
  \end{multline*}
where we have used the  independence between $ \overrightarrow{Z}_{k+1}^{(1)} (  \overrightarrow{{\bf \eta}_{3k} }(j)) $ and $  \overrightarrow{Z}_{k+1}^{(2)} (  \overrightarrow{{\bf \eta}_{3(k+1)} }(j)) $. Now, since 
 $  3km+1 \leq i \leq  (3k+1)m$ and   $(3k+1)m  + 1 \leq  \ell \leq  (3k+2)m$,  taking the conditional expectation with respect to $\overrightarrow{{\bf \eta}_{\bf 3}} = {\bf a }_{\bf 3}$, we get 
\begin{multline*}
 \E \Big (  \big [  \overrightarrow{Z}_{k+1}^{(1)} (  \overrightarrow{{\bf \eta}_{3k} }(j)) \big ]_i   \big [ \overrightarrow{Z}_{k+1}^{(1)} 
(  \overrightarrow{{\bf \eta}_{3k} }(j)) \big ]_{\ell} \Big )  =  \int  \E \Big [   \big [ \overrightarrow{G}^{(1)}_{k+1}   (    {\bf a}_{3k } (j)    )  \big ]_i  \big [ \overrightarrow{G}^{(1)}_{k+1}   (    {\bf a}_{3k } (j)    )  \big ]_\ell   \Big ]  dP_{\overrightarrow{{\bf \eta}_{\bf 3}} }  (  {\bf a }_{\bf 3}  )  \\
=  \int  \E \Big [  Y^{(3k+1)}_{i,j}   ( {\bf a}_{3k } (j)  )   \E \big  (  X_{\ell,j}^{(m)} |  \overrightarrow{{\bf \eta}_{3k+1} }(j)   \big)   \Big ]  dP_{\overrightarrow{{\bf \eta}_{\bf 3}} }  (  {\bf a }_{\bf 3}  )  
=  \E  \big ( X^{(m)}_{i,j}  X^{(m)}_{\ell,j}   \big ) \, ,
  \end{multline*}
where for the last equality we have used that since $Y^{(3k+1)}_{i,j}   ( {\bf a}_{3k } (j)  ) $ is a measurable function of $\overrightarrow{{\bf \eta}_{3k+1} }(j) $, 
\[
 \E \Big [  Y^{(3k+1)}_{i,j}   ( {\bf a}_{3k } (j)  )   \E \big  (  X_{\ell,j}^{(m)} |  \overrightarrow{{\bf \eta}_{3k+1} }(j)   \big)   \Big ]  =  \E \Big [  Y^{(3k+1)}_{i,j}   ( {\bf a}_{3k } (j)  )   X_{\ell,j}^{(m)}   \Big ] \, .
\]
Consider now the case where $  (3k+1)m+1 \leq i \leq  (3k+2)m$ and   $(3k+1)m  + 1 \leq  \ell \leq  (3k+2)m$ for some non negative $k$.  We have 
\begin{multline*}
\Cov \Big (  \big [ {\bf Z}_n (j) \big ]_i ,  \big [ {\bf Z}_n (j) \big ]_{\ell}   \Big )  = \E \Big (  \big [ \overrightarrow{Z}_{k+1}^{(1)} (  \overrightarrow{{\bf \eta}_{3k} }(j))  +   \overrightarrow{Z}_{k+1}^{(2)} 
(  \overrightarrow{{\bf \eta}_{3(k+1)} }(j)) \big ]_i   \big [  \overrightarrow{Z}_{k+1}^{(1)} (  \overrightarrow{{\bf \eta}_{3k} }(j)) +  \overrightarrow{Z}_{k+1}^{(2)} 
(  \overrightarrow{{\bf \eta}_{3(k+1)} }(j)) \big ]_{\ell} \Big ) \\
=  \E \Big (  \big [  \overrightarrow{Z}_{k+1}^{(1)} (  \overrightarrow{{\bf \eta}_{3k} }(j)) \big ]_i   \big [ \overrightarrow{Z}_{k+1}^{(1)} 
(  \overrightarrow{{\bf \eta}_{3k} }(j)) \big ]_{\ell} \Big ) +  \E \Big (  \big [  \overrightarrow{Z}_{k+1}^{(2)} (  \overrightarrow{{\bf \eta}_{3(k+1)} }(j)) \big ]_i   \big [ \overrightarrow{Z}_{k+1}^{(2)} 
(  \overrightarrow{{\bf \eta}_{3(k+1)} }(j)) \big ]_{\ell} \Big ) \, ,
  \end{multline*}
where we have used the  independence between $ \overrightarrow{Z}_{k+1}^{(1)} (  \overrightarrow{{\bf \eta}_{3k} }(j)) $ and $  \overrightarrow{Z}_{k+1}^{(2)} (  \overrightarrow{{\bf \eta}_{3(k+1)} }(j)) $. Hence, by definition, since $(i, \ell)$ is in $[\! [ (3k+1)m  + 1, (3k+2)m ] \! ]^2$, 
\begin{multline*}
\Cov \Big (  \big [ {\bf Z}_n (j) \big ]_i ,  \big [ {\bf Z}_n (j) \big ]_{\ell}   \Big )  = \E \Big (  \big [  \E \big  (  X_{i,j}^{(m)} |  \overrightarrow{{\bf \eta}_{3k+1} }(j)   \big)   \E \big  (  X_{\ell,j}^{(m)} |  \overrightarrow{{\bf \eta}_{3k+1} }(j)   \big)  \Big )  \\
+  \E \Big (  \big \{  X_{i,j}^{(m)} -  \E \big  (  X_{i,j}^{(m)} |  \overrightarrow{{\bf \eta}_{3k+1} }(j)   \big)   \big \}    \big \{  X_{\ell,j}^{(m)}   -  \E \big  (  X_{\ell,j}^{(m)} |  \overrightarrow{{\bf \eta}_{3k+1} }(j)   \big) \big \} \Big )  
=   \E  \big ( X^{(m)}_{i,j}  X^{(m)}_{\ell,j}   \big ) \, .
  \end{multline*}
Finally, we consider the case where $  (3k+1)m+1 \leq i \leq  (3k+2)m$ and   $(3k+2)m  + 1 \leq  \ell \leq  (3k+3)m$ for some non negative $k$.  We have 
\begin{multline*}
\Cov \Big (  \big [ {\bf Z}_n (j) \big ]_i ,  \big [ {\bf Z}_n (j) \big ]_{\ell}   \Big )  = \E \Big (  \big [ \overrightarrow{Z}_{k+1}^{(1)} (  \overrightarrow{{\bf \eta}_{3k} }(j))  +  \overrightarrow{Z}_{k+1}^{(2)} 
(  \overrightarrow{{\bf \eta}_{3(k+1)} }(j))  \big ]_i   \big [   \overrightarrow{Z}_{k+1}^{(2)} 
(  \overrightarrow{{\bf \eta}_{3(k+1)} }(j)) \big ]_{\ell} \Big ) \\
= \E \Big (  \big [ \overrightarrow{Z}_{k+1}^{(2)} 
(  \overrightarrow{{\bf \eta}_{3(k+1)} }(j))  \big ]_i   \big [   \overrightarrow{Z}_{k+1}^{(2)} 
(  \overrightarrow{{\bf \eta}_{3(k+1)} }(j)) \big ]_{\ell} \Big )  \, ,
  \end{multline*}
where we have used the  independence between $ \overrightarrow{Z}_{k+1}^{(1)} (  \overrightarrow{{\bf \eta}_{3k} }(j)) $ and $  \overrightarrow{Z}_{k+1}^{(2)} (  \overrightarrow{{\bf \eta}_{3(k+1)} }(j)) $. Now, since 
$  (3k+1)m+1 \leq i \leq  (3k+2)m$ and   $(3k+2)m  + 1 \leq  \ell \leq  (3k+3)m$,  taking the conditional expectation with respect to $\overrightarrow{{\bf \eta}_{\bf 3}} = {\bf a }_{\bf 3}$, we get 
\begin{multline*}
  \E \Big (  \big [ \overrightarrow{Z}_{k+1}^{(2)} 
(  \overrightarrow{{\bf \eta}_{3(k+1)} }(j))  \big ]_i   \big [   \overrightarrow{Z}_{k+1}^{(2)} 
(  \overrightarrow{{\bf \eta}_{3(k+1)} }(j)) \big ]_{\ell} \Big )   =  \int  \E \Big [   \big [ \overrightarrow{G^{(2)}_{k+1}}   (    {\bf a}_{3(k+1) } (j)    )  \big ]_i  \big [ \overrightarrow{G^{(2)}_{k+1}}   (    {\bf a}_{3(k+1) } (j)    )  \big ]_\ell   \Big ]  dP_{\overrightarrow{{\bf \eta}_{\bf 3}} }  (  {\bf a }_{\bf 3}  )  \\
=  \int  \E \Big [  \big \{    X_{i,j}^{(m)} -  \E \big  (  X_{i,j}^{(m)} |  \overrightarrow{{\bf \eta}_{3k+1} }(j)   \big)   \big \}  Y^{(3(k+1))}_{\ell,j}   ( {\bf a}_{3(k+1) } (j)  )   \Big ]  dP_{\overrightarrow{{\bf \eta}_{\bf 3}} }  (  {\bf a }_{\bf 3}  )  
=  \E  \big ( X^{(m)}_{i,j}  X^{(m)}_{\ell,j}   \big ) \, ,
  \end{multline*}
where for the last equality we have used that since $ Y^{(3(k+1))}_{\ell,j}   ( {\bf a}_{3(k+1) } (j)  ) $ is a measurable function of $\overrightarrow{{\bf \eta}_{3k+2} }(j) $,  it is non correlated with $ \E \big  (  X_{i,j}^{(m)} |  \overrightarrow{{\bf \eta}_{3k+1} }(j)   \big) $.  

The other cases can be handled similarly.

\section{Concentration of the resolvent}
\subsection{The results}
The proof of Proposition  \ref{prop:lower-bound-X} is mainly based of the following concentration inequality for the resolvent: 
\begin{proposition}\label{prop:X-from-Xm}
There exist universal positive constants $c,C>0$ such that for every $z\in \mathbb C$ with $ \Im m(z)>0$ and $x\ge 0$,
$$
\Pbb\Bigl(
\bigl|\bigl\|
(zI_n-\mathbb X_n)^{-1}\bigr\|-\E\bigl\|(zI_n-\mathbb X_n)^{-1}
\bigr\|\bigr|
\ge
\xi_{n,m}(x,z)
\Bigr)
\le 4e^{-c x},
$$
where
$$
\xi_{n,m}(x,z)
=
C\left[
\left(
\frac{M\sqrt m}{\Im m(z)^2 n^{1/4}}
+
\frac{M^2m}{\Im m(z)^3\sqrt n}
\right)\sqrt x
+
\left(
\frac{M}{\Im m(z)^2}\sqrt{\frac{m}{n}}
+
\frac{M^2m}{\Im m(z)^3 n}
\right)x
+
\frac{\sqrt n \delta(m)}{\Im m(z)^2}
\right].
$$
\end{proposition}
As we shall see, this result is based on the following intermediate result involving $m$-dependent entries. 

\begin{proposition}\label{prop:conditional-concentration}
There exist a universal constant $c >0$ such that, for every $z\in \mathbb C$ with $\Im m(z)>0$ and $x\ge 0$,
$$
\Pbb\Bigl(
\bigl|\bigl\|(zI_n-\mathbb X_n^{(m)})^{-1}\bigr\|-\E\bigl\|(zI_n-\mathbb X_n^{(m)})^{-1}\bigr\|\bigr|
\ge \,\Theta_{n,m}(x,z)
\Bigr)
\le 4e^{-cx},
$$
where 
\beq \label{defthetaxz}
\Theta_{n,m}(x,z)
=
C\left[
\left(
\frac{M\sqrt m}{\Im m(z)^2 n^{1/4}}
+
\frac{M^2m}{\Im m(z)^3\sqrt n}
\right)\sqrt x
+
\left(
\frac{M}{\Im m(z)^2}\sqrt{\frac{m}{n}}
+
\frac{M^2m}{\Im m(z)^3 n}
\right)x
\right].
\eeq
\end{proposition}

\subsection{Proofs}
To simplify the notations, we shall drop the index $n$ in the definitions of the matrices $\mathbb X_n^{(m)}$, $\mathbb X_n$ and $\mathbb G_n$.

\subsubsection{Proof of Proposition \ref{prop:conditional-concentration}}
The proof of this result is based on a blocking approach and by working conditionally to a part of the random variables, in the spirit of \cite{Jirak}, where this idea is used to prove a
Berry-Esseen bound for functions of iid random variables.\\
Set
$
q_{n,m}=\left\lfloor \frac{n}{2m}\right\rfloor+1$, then for $1\le r\le q_{n,m}$ and $1\le j\le n$, define the even blocks
$$
\vec{\eta}_{2r}(j)=
\bigl(\varepsilon_{(2r-1)m+1,j},\ldots,\varepsilon_{n\wedge 2rm,j}\bigr),
$$
and the odd blocks
$$
\vec{\nu}_{2r}(j)=
\bigl(\varepsilon_{(2r-2)m+1,j},\ldots,\varepsilon_{n\wedge (2r-1)m,j}\bigr).
$$
We also introduce the initial block
$$
\vec{\nu}_{0}(j)=(\varepsilon_{1-m,j},\ldots,\varepsilon_{0,j}).
$$
Let
$$
\vec \eta_{\mathbf{2}}=\bigl(\vec{\eta}_{2r}(j)\bigr)_{1\le r\le q_{n,m},\ 1\le j\le n},
\qquad
\vec{\nu}_{\mathbf{2}}=\bigl(\vec{\nu}_0(j),\vec{\nu}_{2r}(j)_{1\le r\le q_{n,m}}\bigr)_{1\le j\le n},
$$
and define
$$
\mathcal F_m=\sigma(\vec \eta_{\mathbf{2}}).
$$
The matrix $\mathbb X^{(m)}$ may be written as
$$
\mathbb X^{(m)}
=
\frac1{\sqrt n}
\sum_{r=1}^{q_{n,m}}
\sum_{i=(2r-2)m+1}^{\,n\wedge 2rm}
\sum_{j=1}^{i}
X_{i,j}^{(m)}E_{i,j}.
$$
Accordingly, the row indices are partitioned into consecutive double blocks
$$
[1,2m],\ [2m+1,4m],\ \ldots,\ [2(r-1)m+1,2rm],\ \ldots,
$$
each of which consists of two adjacent sub-blocks of length $m$,
$$
\underbrace{[(2r-2)m+1,(2r-1)m]}_{\text{left block}}
\qquad\cup\qquad
\underbrace{[(2r-1)m+1,2rm]}_{\text{right block}}.
$$
The key point is that, in the proof, we shall alternately freeze one family of blocks and keep the other random. More precisely, we shall condition either on the even blocks through a realization of
$\vec\eta_{\mathbf 2}$, or on the odd blocks through a realization of $\vec\nu_{\mathbf 2}$.
To formalize this procedure, let $\mathbb P_{\vec\eta_{\mathbf 2}}$ and
$\mathbb P_{\vec\nu_{\mathbf 2}}$ denote the laws of $\vec\eta_{\mathbf 2}$ and
$\vec\nu_{\mathbf 2}$, respectively. For $1-m\le u\le n$, $1\le j\le n$, and deterministic
realizations $\mathbf a_{\mathbf 2}$ of $\vec\eta_{\mathbf 2}$ and $\mathbf b_{\mathbf 2}$ of
$\vec\nu_{\mathbf 2}$, define
$$
\xi_{u,j}^{a,b}
=
\begin{cases}
a_{u,j}, & \text{if }u\in\{(2r-1)m+1,\ldots,n\wedge 2rm\}
          \text{ for some }1\le r\le q_{n,m},\\
b_{u,j}, & \text{otherwise.}
\end{cases}
$$
Recall that
\[
X^{(m)}_{i,j}
=
\E \Big( X_{i, j}\,\Big|\, \sigma ( \varepsilon_{i-k,j},  0 \leq k \leq m ) \Big)
:= g_m  (  \varepsilon_{i-m,j} , \dots, \varepsilon_{i,j}),
\]
for some measurable real-valued function $g_m$, we then set
$$
X_{i,j}^{(m)}(\mathbf a_{\mathbf 2},\mathbf b_{\mathbf 2})
=
g_m\bigl(\xi_{i-m,j}^{a,b},\ldots,\xi_{i,j}^{a,b}\bigr),
\qquad 1\le j\le i\le n.
$$
Equivalently, this means that
$$
X_{i,j}^{(m)}(\mathbf a_{\mathbf 2},\mathbf b_{\mathbf 2})
=
\begin{cases}
g_m\bigl(b_{i-m,j},\ldots,b_{0,j},b_{1,j},\ldots,b_{i,j}\bigr),
& 1\le i\le m,\\
g_m\bigl(b_{i-m,j},\ldots,b_{(2r-1)m,j},a_{(2r-1)m+1,j},\ldots,a_{i,j}\bigr),
& (2r-1)m+1\le i\le n\wedge 2rm,\\
g_m\bigl(a_{i-m,j},\ldots,a_{2rm,j},b_{2rm+1,j},\ldots,b_{i,j}\bigr),
& 2rm+1\le i\le n\wedge (2r+1)m.
\end{cases}
$$
Note that by construction, $\mathbb X^{(m)}
=
\mathbb X^{(m)}(\vec \eta_{\mathbf 2},\vec \nu_{\mathbf 2})$ almost surely. For every $t>0$, the triangle inequality gives
$$
\Pbb\Bigl(
\bigl|\bigl\|(zI_n-\mathbb X^{(m)})^{-1}\bigr\|-\E\bigl\|(zI_n-\mathbb X^{(m)})^{-1}\bigr\|\bigr|\ge 2t
\Bigr)
\le I_1(t)+I_2(t),
$$
where
$$
I_1(t)
=
\E\Bigl[
\Pbb_{\mathcal F_m}\Bigl(
\bigl|
\bigl\|(zI_n-\mathbb X^{(m)})^{-1}\bigr\|-\E_{\mathcal F_m}\bigl\|(zI_n-\mathbb X^{(m)})^{-1}\bigr\|
\bigr|
\ge t
\Bigr)
\Bigr]
$$
and
$$
I_2(t)
=
\Pbb\Bigl(
\bigl|
\E_{\mathcal F_m}\bigl\|(zI_n-\mathbb X^{(m)})^{-1}\bigr\|-\E\bigl\|(zI_n-\mathbb X^{(m)})^{-1}\bigr\|
\bigr|
\ge t
\Bigr).
$$
\medskip
\noindent\textbf{Estimate of $I_1$.}\\
Fix a realization $\mathbf{a_2}$ of $\vec {\eta}_{\mathbf 2}$, and set $\mathbb X_a^{(m)}=\mathbb X^{(m)}(\mathbf{a_2},\vec{\nu}_{\mathbf{2}})$.
For $1\le r\le q_{n,m}$ and $1\le j\le n$, define
$$
Y_{r,j}(\mathbf{a_2})
=
\frac1{\sqrt n}
\sum_{i=\max\{j,(2r-2)m+1\}}^{n\wedge 2rm}
X_{i,j}^{(m)}(\mathbf{a_2},\vec{\nu}_{\mathbf{2}})\,E_{i,j}
\quad \text{and}\quad
\overset{\circ}{Y}_{r,j}(\mathbf{a_2}) = Y_{r,j}(\mathbf{a_2}) - \E[Y_{r,j}(\mathbf{a_2})].
$$
Then
$$
\mathbb X_a^{(m)}=\E \mathbb  X_a^{(m)} + \sum_{r=1}^{q_{n,m}}\sum_{j=1}^{n}\overset{\circ}{Y}_{r,j}(\mathbf{a_2}).
$$
Note that for fixed $\mathbf{a_2}$, the family $\bigl(\overset{\circ}{Y}_{r,j}(\mathbf{a_2})\bigr)_{1\le r\le q_{n,m},\ 1\le j\le n}$
is independent. Therefore, one can apply \cite[Proposition 5.6]{brailovskayaVanHandel} to prove that there exists a universal constant  $c>0$ such that
\begin{eqnarray}\label{eq:I1}
&\mathbb{P}\Bigg(
\Big|\|(z\mathbf{1}-\mathbb  X_a^{(m)})^{-1}\|-\mathbb{E}\|(z\mathbf{1}- \mathbb X_a^{(m)})^{-1}\|\Big|
\geq
\frac{\tilde{\sigma}(X_a^{(m)})}{(\operatorname{Im} z)^2}\sqrt{x}
+\left\{
\frac{R( \mathbb X_a^{(m)})}{(\operatorname{Im} z)^2}
+\frac{R( \mathbb X_a^{(m)})^2}{(\operatorname{Im} z)^3}
\right\}x \\
&\quad\qquad \qquad+
\left\{
\frac{R(\mathbb X_a^{(m)})^{1/2}\big(\mathbb{E}\| \mathbb X_a^{(m)}-\mathbb{E} \mathbb X_a^{(m)}\|\big)^{1/2}}{(\operatorname{Im} z)^2}
+
\frac{R(\mathbb X_a^{(m)})\big(\mathbb{E}\|\mathbb X_a^{(m)}-\mathbb{E} \mathbb X_a^{(m)}\|^2\big)^{1/2}}{(\operatorname{Im} z)^3}
\right\}\sqrt{x}
\Bigg)
\leq 2e^{-cx}.\nonumber
\end{eqnarray}
for any $x \geq 0$ and any $z \in \mathbb C$ with $\Im m (z) > 0$  and
$$
R(\mathbb X_a^{(m)})
=
\left\|\max_{1\le r\le q_{n,m},\ 1\le j\le n}\|\overset{\circ}{Y}_{r,j}(\mathbf{a_2})\|\right\|_\infty
\quad \text{and}\quad
\tilde{\sigma}(\mathbb X_a^{(m)})=
\sup_{\|u\|_2=\|v\|_2=1}
\left(
\sum_{r=1}^{q_{n,m}}\sum_{j=1}^{n}
\E\bigl|\langle u,\overset{\circ}{Y}_{r,j}(\mathbf{a_2}) v\rangle\bigr|^2
\right)^{1/2}.
$$
We now turn to the estimation of the quantities appearing in \eqref{eq:I1}. The same argument used to establish \eqref{boundonR} yields
\begin{equation}\label{eq:R-cond-proof}
R(\mathbb X_a^{(m)})\le CM\sqrt{\frac{m}{n}}.
\end{equation}
Let's now obtain a bound on $\tilde{\sigma}(\mathbb X_a^{(m)})$. Let $u = (u_1,...,u_n),v = (v_1,...,v_n)\in\mathbb C^n$ be unit vectors and set
$$
\gamma_{i,j}(u,v)
=
\begin{cases}
\overline u_i v_j+\overline u_j v_i,& i>j\\
\overline u_j v_j,& i=j
\end{cases}
,
$$
for $1\le j\le i\le n$. Since $ |X_{i,j}^{(m)}(\mathbf{a_2},\vec{\nu}_{\mathbf{2}})|\le M$, we obtain for $r = 1,...,q_{n,m}$ and $j=1,...,n$
\begin{align}\label{sumrjY}
\bigl|
\langle u,\overset{\circ}{Y}_{r,j}(\mathbf{a_2})v\rangle
\bigr|^2
&\le
\frac1n
\left(
\sum_{i=\max\{j,(2r-2)m+1\}}^{n\wedge 2rm}
2M\,|\gamma_{i,j}(u,v)|
\right)^2
\le
\frac{16M^2m}{n}
\sum_{i=\max\{j,(2r-2)m+1\}}^{n\wedge 2rm}
|\gamma_{i,j}(u,v)|^2.
\end{align}
Note that $ \sum_{1\le j\le i\le n}|\gamma_{i,j}(u,v)|^2\le 2$, therefore summing over $r$ and $j$ gives
\begin{equation}\label{eq:sigma-cond-proof}
\tilde{\sigma}(\mathbb X_a^{(m)})\le CM\sqrt{\frac{m}{n}}.
\end{equation}
We give now upper bounds of $ \E\|\mathbb X_a^{(m)}-\E\mathbb X_a^{(m)}\|$ and $\Bigl(\E\|\mathbb X_a^{(m)}-\E\mathbb X_a^{(m)}\|^2\Bigr)^{1/2}$. Fix unit vectors $u,v\in\Cbb^n$, then
$$
\langle u,(\mathbb X_a^{(m)}-\E\mathbb X_a^{(m)})v\rangle
=
\sum_{r=1}^{q_{n,m}}\sum_{j=1}^{n}
\langle u,\overset{\circ}{Y}_{r,j}(\mathbf{a_2})v\rangle,
$$
which is a sum of independent centered random variables. Therefore 
Hoeffding's inequality together with the upper bound \eqref{sumrjY} yields
\begin{equation}\label{eq:Hoeffding-concentration}
\Pbb_a\Bigl(
|\langle u,(\mathbb X_a^{(m)}-\E\mathbb X_a^{(m)})v\rangle|\ge s
\Bigr)
\le
2\exp\!\left(-\frac{cns^2}{M^2m}\right)
\end{equation}
for a universal constant $c>0$. Let $\mathcal N$ be a $1/4$-net of the unit sphere of $\Cbb^n$ such that
$
|\mathcal N|\le C^n 
$ (with $C \geq 1$). 
Using \eqref{eq:Hoeffding-concentration} along with a standard net argument (see \cite[p.130]{tao2023topics}), one gets 
$$
\Pbb(\|\mathbb X_a^{(m)}-\E\mathbb X_a^{(m)}\|\geq 2s)
\le
\Pbb(\max_{u,v\in\mathcal N}
|\langle u,(\mathbb X_a^{(m)}-\E\mathbb X_a^{(m)})v\rangle| \geq \frac{s}{2})\le
2C^{2n}\exp\!\left(-\frac{cns^2}{M^2m}\right).
$$
Choose $s_0=2M\sqrt{\frac{\log C}{c}}\,\sqrt m$,  then for $s\ge s_0$
$$
\Pbb_a\Bigl(
\|\mathbb X_a^{(m)}-\E\mathbb X_a^{(m)}\|\ge 2s
\Bigr)
\le
2\exp\!\left(-\frac{cns^2}{2M^2m}\right).
$$
It follows that 
\[
\E\|\mathbb X_a^{(m)}-\E\mathbb X_a^{(m)}\|^2
=
2\int_0^\infty y\,
\Pbb_a\bigl(\|\mathbb X_a^{(m)}-\E\mathbb X_a^{(m)}\|\ge y\bigr)\,dy,
\le
4s_0^2
+
4\int_{2s_0}^{\infty}
y\exp\!\left(-\frac{cny^2}{8M^2m}\right)\,dy
\le CM^2m.
\]
Then we can get the following estimates from the Cauchy-Schwarz inequality 
\begin{equation}\label{eq:moment-cond-proof}
\E\left\|
\mathbb X_a^{(m)}-\E\mathbb X_a^{(m)}
\right\|
\leq
\left(
\E\left\|
\mathbb X_a^{(m)}-\E\mathbb X_a^{(m)}
\right\|^2
\right)^{1/2}
\leq
CM\sqrt m.
\end{equation}
Applying Proposition~5.6 of \cite{brailovskayaVanHandel} to $\mathbb X_a^{(m)}$, and using
\eqref{eq:I1}, \eqref{eq:R-cond-proof}, \eqref{eq:sigma-cond-proof} and \eqref{eq:moment-cond-proof}, we obtain
$$
\Pbb_a\Bigl(
\bigl|
\|(zI_n-\mathbb X_a^{(m)})^{-1}\|
-
\E\|(zI_n-\mathbb X_a^{(m)})^{-1}\|
\bigr|
\ge
\frac12\Theta_{n,m}(x,z)
\Bigr)
\le 2e^{-cx},
$$
where
$
\Theta_{n,m}(x,z)
$ is given by \eqref{defthetaxz}. 
Integrating with respect to $\mathbb P_{\vec {\eta}_{\mathbf 2}}$, we obtain
\begin{equation}\label{eq:I1-final-clean}
I_1\bigl(\frac12\Theta_{n,m}(x,z)\bigr)\le 2e^{-cx}.
\end{equation}

\medskip
\noindent\textbf{Estimate of $I_2$.}
For every deterministic realization $\mathbf{b_2}$ of $\vec{\nu}_{\mathbf{2}}$, define
$
g_b(\mathbf{a_2})=\bigl\|(zI_n-\mathbb X^{(m)}(\mathbf{a_2},\mathbf{b_2}))^{-1}\bigr\|$, for $a\in \operatorname{supp}(\mathbb P_{ \vec {\eta}_{\mathbf 2}}).
$
Evaluating at the random even-block family $\vec {\eta}_{\mathbf 2}$, we obtain
$$
g_b(\vec {\eta}_{\mathbf 2})=\bigl\|(zI_n-\mathbb X^{(m)}(\vec {\eta}_{\mathbf 2},\mathbf{b_2}))^{-1}\bigr\|.
$$
We set $\Phi(\vec {\eta}_{\mathbf 2})=\int g_b(\vec {\eta}_{\mathbf 2})\,\mathbb dP_{\vec{\nu}_{\mathbf{2}}}(\mathbf{b_2})$. Since $ \mathbb X^{(m)}=\mathbb X^{(m)}(\vec {\eta}_{\mathbf 2},\vec{\nu}_{\mathbf{2}})$ almost surely, we have
$$
\Phi(\vec {\eta}_{\mathbf 2})=\E_{\mathcal F_m}\bigl\|(zI_n-\mathbb X^{(m)})^{-1}\bigr\|
\qquad\text{almost surely.}
$$
Therefore
$$
I_2(t)
=
\Pbb\Bigl(
\bigl|\Phi(\vec {\eta}_{\mathbf 2})-\E[\Phi(\vec {\eta}_{\mathbf 2})]\bigr|
\ge t
\Bigr).
$$
Fix $\mathbf{b_2}$ and for $0\le r\le q_{n,m}$,
$$
Z_{r,j}(\mathbf{b_2})
=
\frac1{\sqrt n}
\sum_{i=\max\{j,(2r-1)m+1\}}^{n\wedge(2r+1)m}
X_{i,j}^{(m)}(\vec {\eta}_{\mathbf 2},\mathbf{b_2})\,E_{i,j},
\qquad 1\le j\le n.
$$
Then
$$
\mathbb X^{(m)}(\vec {\eta}_{\mathbf 2},\mathbf{b_2})
=
\sum_{r=0}^{q_{n,m}}\sum_{j=1}^{n} Z_{r,j}(\mathbf{b_2}),
$$
and, for fixed $\mathbf{b_2}$, the family  of random matrices 
$
\bigl(Z_{r,j}(\mathbf{b_2})\bigr)_{0\le r\le q_{n,m},\ 1\le j\le n}
$
is independent. For $0\leq r \leq q_{n,m}$ denote $\overset{\circ}{Z}_{r,j}(\mathbf{b_2}) = Z_{r,j}(\mathbf{b_2})-\mathbb E Z_{r,j}(\mathbf{b_2})$, then the same argument used to establish \eqref{eq:R-cond-proof}, \eqref{eq:sigma-cond-proof} and \eqref{eq:moment-cond-proof} yield
$$
R(\mathbb X^{(m)}(\vec {\eta}_{\mathbf 2},\mathbf{b_2}))=\left\|\max_{0\le r\le q_{n,m},\ 1\le j\le n}\|\overset{\circ}{Z}_{r,j}(\mathbf{b_2})\|\right\|_\infty
\le CM\sqrt{\frac{m}{n}},
$$
$$
\tilde\sigma\bigl(\mathbb X^{(m)}(\vec {\eta}_{\mathbf 2},\mathbf{b_2})\bigr) = \sup_{\|u\|_2=\|v\|_2=1}
\left(
\sum_{r=0}^{q_{n,m}}\sum_{j=1}^{n}
\E\bigl|\langle u,\overset{\circ}{Z}_{r,j}(\mathbf{a_2}) v\rangle\bigr|^2
\right)^{1/2}
\le CM\sqrt{\frac{m}{n}},
$$
and
$$
\E\bigl\|
\mathbb X^{(m)}(\vec {\eta}_{\mathbf 2},\mathbf{b_2})-\E\mathbb X^{(m)}(\vec {\eta}_{\mathbf 2},\mathbf{b_2})
\bigr\|
 \leq 
\Bigl(
\E\bigl\|
\mathbb X^{(m)}(\vec {\eta}_{\mathbf 2},\mathbf{b_2})-\E\mathbb X^{(m)}(\vec {\eta}_{\mathbf 2},\mathbf{b_2})
\bigr\|^2
\Bigr)^{1/2}
\le CM\sqrt m,
$$
uniformly in $\mathbf{b_2}$.\\
In the following, we use the arguments from the proof of \cite[Proposition 5.6]{brailovskayaVanHandel}, in order to obtain
\begin{equation}\label{eq:bound-5.6}
    \E\exp\Bigl(
\lambda s\bigl(g_b(\vec {\eta}_{\mathbf 2})-\E[g_b(\vec {\eta}_{\mathbf 2})]\bigr)
\Bigr)
\le
\exp\!\left(
\frac{C\lambda^2A_{n,m}(z)^2}{1-\lambda B_{n,m}(z)}
\right),
\end{equation}
for every $0\le \lambda<1/(C\,B_{n,m}(z))$, $s\in \{ -1,1\}$, where
$$
A_{n,m}(z)
=
\frac{M\sqrt m}{\Im m(z)^2 n^{1/4}}
+
\frac{M^2m}{\Im m(z)^3\sqrt n}
\text{ and }
B_{n,m}(z)
=
\frac{M}{\Im m(z)^2}\sqrt{\frac{m}{n}}
+
\frac{M^2m}{\Im m(z)^3 n}
.
$$
Indeed, let $Z_{r,j}^0(\mathbf{b_2})$ be an independent copy of $Z_{r,j}(\mathbf{b_2})$ for $1\leq r\leq q, 1\leq j\leq n$. One can introduce the following random variable
\begin{eqnarray*}
&W
=
\frac{2}{\Im m(z)^4}
\sup_{\|v\|_2=\|w\|_2=1}
\sum_{r=1}^{q_{n,m}}\sum_{j=1}^{n}
\bigl|
\langle v,(Z_{r,j}(\mathbf{b_2})-Z_{r,j}^0(\mathbf{b_2}))w\rangle
\bigr|^2\\
&+
\frac{8R\bigl(\mathbb X^{(m)}(\Im m(z),b)\bigr)^2}{\Im m(z)^6}
\left\|
\sum_{r=1}^{q_{n,m}}\sum_{j=1}^{n}(Z_{r,j}(\mathbf{b_2})-Z_{r,j}^0(\mathbf{b_2}))^2
\right\|,
\end{eqnarray*}
together with
$$
a=\frac{16R\bigl(\mathbb X^{(m)}(\vec {\eta}_{\mathbf 2},\mathbf{b_2})\bigr)^2}{\Im m(z)^4}+\frac{64R\bigl(\mathbb X^{(m)}(\vec {\eta}_{\mathbf 2},\mathbf{b_2})\bigr)^4}{\Im m(z)^6}.
$$
The self-bounding estimate in the proof of \cite[Proposition 5.6]{brailovskayaVanHandel} gives
\begin{equation}\label{eq:mgfW-selfbounded}
\log \E e^{W/a}\le \frac{2}{a}\E W.
\end{equation}
The exponential Poincar\'e step used there gives, for $0\le\lambda<a^{-1/2}$,
$$
\log \E e^{\lambda (g_b(\vec {\eta}_{\mathbf 2})-\E[g_b(\vec {\eta}_{\mathbf 2})])}
\le
\frac{\lambda^2a}{1-\lambda^2a}\log\E e^{W/a}.
$$
Combining this upper bound with inequality \eqref{eq:mgfW-selfbounded} and using
$1-\lambda^2a\ge 1-\lambda\sqrt a$ yields
\begin{equation}\label{eq:positive-mgf}
\log \E e^{\lambda (g_b(\vec {\eta}_{\mathbf 2})-\E[g_b(\vec {\eta}_{\mathbf 2})])}
\le
\frac{2\lambda^2\E W}{1-\lambda\sqrt a}.
\end{equation}
The lower-tail part of the same proof gives, for $0\le\lambda<(2a)^{-1/2}$,
\begin{equation}\label{eq:negative-mgf}
\log \E e^{-\lambda (g_b(\vec {\eta}_{\mathbf 2})-\E[g_b(\vec {\eta}_{\mathbf 2})])}
\le
\frac{4\lambda^2\E W}{1-\lambda\sqrt{2a}}.
\end{equation}
But the following estimate from \cite[Lemma 5.8]{brailovskayaVanHandel} holds:
\begin{eqnarray*}
\E W
&\le
C\left(
\frac{\sigma_*\bigl(\mathbb X^{(m)}(\vec {\eta}_{\mathbf 2},\mathbf{b_2})\bigr)^2}{\Im m(z)^4}
+
\frac{R\bigl(\mathbb X^{(m)}(\vec {\eta}_{\mathbf 2},\mathbf{b_2})\bigr)\E\|\mathbb X^{(m)}(\vec {\eta}_{\mathbf 2},\mathbf{b_2})-\E  \mathbb X^{(m)}(\vec {\eta}_{\mathbf 2},\mathbf{b_2})\|}{\Im m(z)^4}\right.\\
&+
\left.\frac{R\bigl(\mathbb X^{(m)}(\vec {\eta}_{\mathbf 2},\mathbf{b_2})\bigr)^2\E\|\mathbb X^{(m)}(\vec {\eta}_{\mathbf 2},\mathbf{b_2})-\E  \mathbb X^{(m)}(\vec {\eta}_{\mathbf 2},\mathbf{b_2})\|^2}{\Im m(z)^6}
\right).
\end{eqnarray*}
Note that $(u+v+w)^2\ge u^2+v^2+w^2$ for any $u,v,w \geq 0$, then applying this inequality with 
$$
u=\frac{\sigma_*(X^{(m)}(\vec {\eta}_{\mathbf 2},\mathbf{b_2}))}{\Im m(z)^2},\quad
v=\frac{R(X^{(m)}(\vec {\eta}_{\mathbf 2},\mathbf{b_2}))^{1/2}\bigl(\E\|X^{(m)}(\vec {\eta}_{\mathbf 2},\mathbf{b_2})-\E X^{(m)}(\vec {\eta}_{\mathbf 2},\mathbf{b_2})\|\bigr)^{1/2}}{\Im m(z)^2},
$$
$$
w=\frac{R(X^{(m)}(\vec {\eta}_{\mathbf 2},\mathbf{b_2}))\bigl(\E\|X^{(m)}(\vec {\eta}_{\mathbf 2},\mathbf{b_2})-\E X^{(m)}(\vec {\eta}_{\mathbf 2},\mathbf{b_2})\|^2\bigr)^{1/2}}{\Im m(z)^3},
$$
yields
$$
\E W\le C\,A^2_{n,m}(z) .
$$
On the other hand,
$$
\sqrt a
\le
\frac{4R(X^{(m)}(\vec {\eta}_{\mathbf 2},\mathbf{b_2}))}{\Im m(z)^2}+\frac{8R(X^{(m)}(\vec {\eta}_{\mathbf 2},\mathbf{b_2}))^2}{\Im m(z)^3}
\le C\,B_{n,m}(z).
$$

Substitution in \eqref{eq:positive-mgf} and \eqref{eq:negative-mgf} and increasing the universal constant
$C$ if necessary, yields \eqref{eq:bound-5.6}. Since
$
\Phi(\vec {\eta}_{\mathbf 2})=\int g_b(\vec {\eta}_{\mathbf 2})\,\mathbb P_{\vec{\nu}_{\mathbf{2}}}(db),
$
Funini's theorem and Jensen's inequality gives for $s\in \{ -1,1\}$
$$
\exp\Bigl(
\lambda s\bigl(\Phi(\vec {\eta}_{\mathbf 2})-\E[\Phi(\vec {\eta}_{\mathbf 2})]\bigr)
\Bigr)
\le
\int
\exp\Bigl(
\lambda s\bigl(g_b(\vec {\eta}_{\mathbf 2})-\E[g_b(\vec {\eta}_{\mathbf 2})]\bigr)
\Bigr)\,\mathbb P_{\vec{\nu}_{\mathbf{2}}}(db).
$$
Taking $\E$ on both sides and using Fubini's theorem, we obtain
$$
\E\exp\Bigl(
\lambda s\bigl(\Phi(\vec {\eta}_{\mathbf 2})-\E[\Phi(\vec {\eta}_{\mathbf 2})]\bigr)
\Bigr)
\le
\exp\!\left(
\frac{C\lambda^2A_{n,m}(z)^2}{1-\lambda B_{n,m}(z)}
\right).
$$
By Chernoff's inequality and optimization on $\lambda$,
$$
I_2\Bigl(
C\bigl(A_{n,m}(z)\sqrt x+B_{n,m}(z)x\bigr)
\Bigr)
\le 2e^{-cx},
\qquad x\ge 0.
$$
Hence, by the definition of $\Theta_{n,m}(x,z)$, 
we conclude that
\begin{equation}\label{eq:I2-final-clean}
    I_2\bigl(\frac12\Theta_{n,m}(x,z)\bigr)\le 2e^{-cx},
\qquad x\ge 0.
\end{equation}

\medskip
\noindent\textbf{Final step.}
Combining \eqref{eq:I1-final-clean} and \eqref{eq:I2-final-clean}, we obtain
$$
\Pbb\Bigl(
\bigl|\bigl\|(zI_n-\mathbb X^{(m)})^{-1}\bigr\|-\E\bigl\|(zI_n-\mathbb X^{(m)})^{-1}\bigr\|\bigr|
\ge 2\,\Theta_{n,m}(x,z)
\Bigr)
\le 4e^{-cx},
$$
which proves the proposition.

\subsubsection{Proof of Proposition \ref{prop:X-from-Xm}}
Set
$$
\mathcal R(z)=\bigl\|(zI_n-\mathbb X)^{-1}\bigr\|,
\qquad
\mathcal R_m(z)=\bigl\|(zI_n-\mathbb X^{(m)})^{-1}\bigr\|.
$$
By the triangle inequality and inequalities \eqref{comparisonXXm} and \eqref{comparisonResXXm}, we derive 
\begin{align*}
|\mathcal R(z)-\E\mathcal R(z)|
&\le
|\mathcal R_m(z)-\E\mathcal R_m(z)|
+
|\mathcal R(z)-\mathcal R_m(z)|
+
|\E\mathcal R_m(z)-\E\mathcal R(z)| \\
& \leq  |\mathcal R_m(z)-\E\mathcal R_m(z)|
+
\frac{4\sqrt n \delta(m)}{\Im m(z)^2}.
\end{align*}
Therefore
$$
\Bigl\{
|\mathcal R(z)-\E\mathcal R(z)|
\ge
\Theta_{n,m}(x,z)+\frac{4\sqrt n \delta(m)}{\Im m(z)^2}
\Bigr\}
\subset
\Bigl\{
|\mathcal R_m(z)-\E\mathcal R_m(z)|
\ge
\Theta_{n,m}(x,z)
\Bigr\}.
$$
Taking probabilities and using the concentration inequality given in Proposition \ref{prop:conditional-concentration}  for $\mathbb X^{(m)}$, Proposition  \ref{prop:X-from-Xm} follows. 

 \subsection{Proof of Lemma \ref{lem:Gaussian-parameter-bounds}}\label{Appendix:free}

\begin{proof}[Proof of Lemma \ref{lem:Gaussian-parameter-bounds}]
In order to prove this lemma, we first need to bound $\|\Gamma_n\|$. For $h\in\mathbb Z$, let $\gamma (|h|) = {\rm Cov} (X_{0,0}, X_{h,0} )$. Since  $\gamma(-h)=\gamma(h)$ for $h\in\mathbb Z$, it follows that 
\begin{equation}
\label{eq:Gamma-Schur-bound}
\|\Gamma_n\|
\leq
\sup_{1\leq i\leq n}
\sum_{k=1}^{n}|\gamma(|i-k|)|
\leq
\sum_{h\in\mathbb Z}|\gamma(h)|
\leq
2\sigma_{\ast}^2.
\end{equation}
\textbf{Bound on } $\tilde{\sigma}_n$:\\
The columns of $\mathbb X_n$ are independent and the matrices $\mathbb X_n$ and $\mathbb G_n$ share the same covariance structure, then for $1\leq j\leq i\leq n$ and $1\leq \ell \leq k\leq n$, we have
\begin{equation}\label{eq:Gaussian-Cov}
\mathbb{E}[g_{i,j}g_{k,l}] = \begin{cases}
    (\Gamma_n)_{i,k} \text{ if } j=  \ell\\
    0 \text{ if } j\neq \ell
\end{cases}
.
\end{equation}
For unit vectors $u$ and $v$ Equations \eqref{defofGn} and \eqref{eq:Gaussian-Cov} give
$$
\mathbb E\left[
\left|\langle u,\mathbb G_n v\rangle\right|^2
\right]
=
\frac{1}{n}\sum_{j=1}^{n}\sum_{i,k=j}^{n} \overline{\langle u,E_{i,j}v\rangle} (\Gamma_{n})_{i,k}\langle u,E_{k,j}v\rangle
 = \frac{1}{n}\sum_{\ell=1}^{n}(\mathbf{c}^{(\ell)})^T \Gamma_n \mathbf{c}^{(\ell)},
$$
with $\mathbf{c}^{(j)} = (0,\ldots,0,\langle u,E_{j,j}v\rangle,\ldots,\langle u,E_{n,j}v\rangle)^T$, for $1\leq j\leq n$. 
Therefore, 
\begin{eqnarray}\label{eq:Gamma-quadratic-bound}
\mathbb E\left[
\left|\langle u,\mathbb G_n v\rangle\right|^2
\right] 
\leq \frac{\|\Gamma_n\|}{n}\sum_{\ell=1}^{n}\sum_{i=\ell}^n |\langle u,E_{i,\ell}v\rangle|^2 \leq \frac{2\|\Gamma_n\|\|u\|^2\|v\|^2}{n}.
\end{eqnarray}
Since the vectors $u$ and $v$ are unit vectors, we conclude taking the supremum on $u,v$ and combining \eqref{eq:Gamma-Schur-bound} and \eqref{eq:Gamma-quadratic-bound}.\\
\smallskip
\\
\textbf{Bound on } $\sigma_n$ :\\
Denote $E_{ij}v = (v_{i,1}^{(j)},...,v_{i,n}^{(j)})$ and
$\mathbf v_r^{(j)}
:=
(0,\ldots,0,v_{j,r}^{(j)},v_{j+1,r}^{(j)},\ldots,v_{n,r}^{(j)})^{\mathsf T}$ for $1\leq i,j,r\leq n$ with $j\leq i$. Then for any unit vector $v$ equations \eqref{defofGn} and \eqref{eq:Gaussian-Cov} yield
\begin{eqnarray*}
\left\langle
v,\mathbb E[\mathbb G_n^2]v
\right\rangle
= \frac{1}{n}\sum_{j=1}^{n}\sum_{i,k=j}^{n}(\Gamma_n)_{i,k}\langle
E_{i,j}v, E_{k,j}v\rangle =
\frac{1}{n}\sum_{j=1}^{n}\sum_{r=1}^n (\mathbf v_r^{(j)})^T\Gamma_n \mathbf v_r^{(j)}.
\end{eqnarray*}
Then
\begin{align*}
\left\langle
v,\mathbb E[\mathbb G_n^2]v
\right\rangle
 \leq \frac{\|\Gamma_n\|}{n}\sum_{\ell=1}^{n}\sum_{r=1}^n \|\mathbf v_r^{(\ell)}\|^2
\leq \frac{\|\Gamma_n\|}{n}
\sum_{1\leq \ell\leq i\leq n}
\|E_{i,\ell}v\|^2.
\end{align*}
A direct computation gives $\sum_{1\leq \ell\leq i\leq n}\|E_{i,\ell}v\|^2=n\|v\|^2$. It follows that
$$
\sigma_n^2
=
\left\|\mathbb E[\mathbb G_n^2]\right\|
\leq
\|\Gamma_n\|,
$$
and we conclude using \eqref{eq:Gamma-Schur-bound}.\\
\smallskip
\\
\textbf{Bound on } $v_n$:\\
Using the identity satisfied by $v_n$ given in \cite[\textsection 1.3.1]{BandeiraBoedihardjoVanHandel} and combining it with the representation \eqref{eq:Gaussian-matrix-representation}, we obtain
$$
v_n^2
=
\sup_{\mathrm{Tr}[T^T T]=1}
\sum_{k,\ell=1}^{n}
\left|
\operatorname{Tr}(A_{(k,\ell),n} T)
\right|^2 = \frac{1}{n}\sup_{\mathrm{Tr}[T^T T]=1}\sum_{k,\ell=1}^{n}
\left|
\sum_{i=\ell}^{n}
(\Gamma_n^{1/2})_{i,k}\operatorname{Tr}(E_{i,\ell} T)
\right|^2.
$$
Denote by $\mathbf{t}^{(\ell)}=(0,\ldots,0, \operatorname{Tr}(E_{\ell,\ell} T),\ldots, \operatorname{Tr}(E_{n,\ell} T))^T$. Then applying the same argument used to establish
\eqref{eq:Gamma-quadratic-bound}, with $\mathbf{c}^{(\ell)}$ replaced by $\mathbf{t}^{(\ell)}$, yields
$$
v_n^2
\leq
\frac{\|\Gamma_n\|}{n}
\sup_{\mathrm{Tr}[T^T T]=1}
\sum_{1\leq \ell\leq i\leq n}
\left|
\mathrm{Tr}( E_{i,\ell}T)
\right|^2\leq
\frac{2\|\Gamma_n\|}{n}
\sup_{\mathrm{Tr}[T^T T]=1}
\mathrm{Tr}[T^T T].
$$
Since $\mathrm{Tr}[T^T T]=1$, we conclude by using \eqref{eq:Gamma-Schur-bound}.
\end{proof}
\section{Some insights for the extension to the two-directional causal model}
\label{appendix:twodimensional-extension}
In this section we explain how the proof of Proposition~\ref{propcomparaisonmomentsp} has to be modified when the entries are generated by the two-directional causal model \eqref{eq:Xkl-def-2D}:
\[
X_{k,\ell}
=
f\big(\varepsilon_{k-i,\ell-j}: i,j\ge 0\big),
\qquad k,\ell\in\mathbb Z,
\]
instead of the one-directional model \eqref{eq:Xkl-def-1D}. The aim is to construct, in this more general setting, a suitable Gaussian Wigner-type proxy whose covariance structure matches that of the finite-memory approximation and whose spectrum can then be compared with ${\rm sp}(\mathbb X_n)$. The main difference is that the finite-memory approximation now depends on a two-dimensional window of innovations. Consequently, the one-dimensional blocking and freezing procedure used in the proof of Proposition~\ref{propcomparaisonmomentsp} must be replaced by a two-dimensional version. For this, we let  
\[
I^{(m)}_{k,\ell}  = [\![ (k-1)m+1, km]\!] \times  [\![ (\ell-1)m+1, \ell m]\!]  \, , 
\]
\[
I^{(m)}_{k, \cdot}  = [\![ (k-1)m+1, km]\!] \times  {\mathbb Z} \text{ and } I^{(m)}_{ \cdot, \ell}  = {\mathbb Z} \times
 [\![ (\ell-1)m+1, \ell m]\!]  \, .  \]
Define
\begin{align*}
\overrightarrow{{\bf \eta}_{k, \ell} } = \{  \varepsilon_{u,v} : (u,v) \in I^{(m)}_{k,\ell} \} \, , \, \overrightarrow{{\bf \eta}_{k, \cdot} } = \{  \varepsilon_{u,v} : (u,v) \in I^{(m)}_{k,\cdot} \} \, , \, \overrightarrow{{\bf \eta}_{ \cdot, \ell} } =
 \{  \varepsilon_{u,v} : (u,v) \in I^{(m)}_{\cdot, \ell} \} \, , 
\end{align*}
and
\begin{align*}
{{\bf a}_{k, \ell} } = \{  a_{u,v} : (u,v) \in I^{(m)}_{k,\ell} \} \, , \, {{\bf a}_{k, \cdot} } = \{ a_{u,v} : (u,v) \in I^{(m)}_{k,\cdot} \} \, , \, {{\bf a}_{ \cdot, \ell} } =
 \{  a_{u,v} : (u,v) \in I^{(m)}_{\cdot, \ell} \} \, .
\end{align*}
Fix two positive integers $k$ and $\ell$. We consider the $3m\times 3m$ macro-block $I^{(3m)}_{k,\ell}$, which is decomposed into nine sub-blocks of size $m\times m$. The purpose of the following construction is to localize the dependence inside the central $2m\times 2m$ part of this macro-block. More precisely, when the finite-memory window associated with an entry $(i,j)$ crosses one of the boundary strips of the macro-block, the corresponding innovations are replaced by deterministic values. This is the two-dimensional analogue of the boundary-freezing procedure used in the one-directional case. The boundary strips outside this square are frozen and are denoted by the deterministic blocks ${\bf a}_{3k-3,\cdot}$, ${\bf a}_{3k,\cdot}$, ${\bf a}_{\cdot,3\ell-3}$ and ${\bf a}_{\cdot,3\ell}$, depending on which side of the macro-block is crossed. This leads to nine modified arrays, corresponding to the nine sub-blocks of the $3m\times 3m$ macro-block. The detailed definitions are given below:

\vspace{0.5cm}
\noindent
$\bullet$ If $(i,j)  \in  [\![ (3k-3)m+1, (3k -2) m]\!] \times  [\![ (3\ell-3)m+1, (3\ell-2)  m]\!] $, let 
\begin{align*}
Y^{(1,1)}_{i, j}  ( {{\bf a}_{ 3k-3, \cdot} } ;  {{\bf a}_{ \cdot, 3 \ell - 3} } )   & = g_m \Big (  \big (  \varepsilon_{u,v}  \big )_{ (3k-3)m+1 \leq u \leq i  \atop  (3\ell-3)m+1 \leq v \leq  j} 
 ,  \big (  a_{u,v}  \big )_{ i-m \leq u \leq i  \atop   j-m \leq v \leq (3\ell-3)m}  ,  \big (  a_{u,v}  \big )_{  i-m \leq u \leq (3k-3)m \atop  (3\ell-3)m+1 \leq v \leq  j}  \Big ) \\
 & : =  g_{m,i,j}^{(1,1)} \big (  \overrightarrow{{\bf \eta}_{3k-2, 3\ell-2} }  ,   {{\bf a}_{ 3k-3, \cdot} } ,  {{\bf a}_{ \cdot, 3 \ell - 3} }  \big ) \, , 
 \end{align*}
otherwise set $Y^{(1,1)}_{i, j}  ( {{\bf a}_{ 3k-3, \cdot} } ;  {{\bf a}_{ \cdot, 3 \ell - 3} } )  =0$. Then define
\[
{\mathbb Y}^{(1,1)}_{k, \ell}  ( {{\bf a}_{ 3k-3, \cdot} } ;  {{\bf a}_{ \cdot, 3 \ell - 3} } )  = \Big (  Y^{(1,1)}_{i, j}  ( {{\bf a}_{ 3k-3, \cdot} } ;  {{\bf a}_{ \cdot, 3 \ell - 3} } )  \Big )_{  (i,j) \in I^{(3m)}_{k, \ell}  } \, .
\]
$\bullet$ If $(i,j)  \in  [\![ (3k-3)m+1, (3k -2) m]\!] \times  [\![ (3\ell-2)m+1, (3\ell-1)  m]\!] $, let 
\begin{align*}
Y^{(1,2)}_{i, j}  ( {{\bf a}_{ 3k-3, \cdot} } )   & = g_m \Big (  \big (  \varepsilon_{u,v}  \big )_{ (3k-3)m+1 \leq u \leq i  \atop  j-m \leq v \leq  j} 
 ,  \big (  a_{u,v}  \big )_{ i-m \leq u \leq   (3k-3)m  \atop   j-m \leq v \leq  j}  \Big ) \\
 & : =  g_{m,i,j}^{(1,2)} \big (  \overrightarrow{{\bf \eta}_{3k-2, 3\ell-2} }  ,   \overrightarrow{{\bf \eta}_{3k-2, 3\ell-1} } ,  {{\bf a}_{ 3k-3, \cdot} }  \big ) \, , 
 \end{align*}
otherwise set $Y^{(1,2)}_{i, j}  ( {{\bf a}_{ 3k-3, \cdot} }  )  =0$. Then define
\[
{\mathbb Y}^{(1,2)}_{k, \ell}  ( {{\bf a}_{ 3k-3, \cdot} } )  = \Big (  Y^{(1,2)}_{i, j}  ( {{\bf a}_{ 3k-3, \cdot} }  )  \Big )_{  (i,j) \in I^{(3m)}_{k, \ell}  } \, .
\]
$\bullet$ If $(i,j)  \in  [\![ (3k-3)m+1, (3k -2) m]\!] \times  [\![ (3\ell-1)m+1, 3\ell  m]\!] $, let 
\begin{align*}
Y^{(1,3)}_{i, j}  ( {{\bf a}_{ 3k-3, \cdot} } ;  {{\bf a}_{ \cdot, 3 \ell } } )     & = g_m \Big (  \big (  \varepsilon_{u,v}  \big )_{ (3k-3)m+1 \leq u \leq i  \atop  j-m \leq v \leq  (3\ell-1)m} 
 ,  \big (  a_{u,v}  \big )_{ i-m \leq u \leq   (3k-3)m  \atop   j-m \leq v \leq  j } ,  \big (  a_{u,v}  \big )_{(3k-3)m+1 \leq u \leq i  \atop  (3\ell-1)m +1 \leq v \leq  j} \Big ) \\
 & : =  g_{m,i,j}^{(1,3)} \big (     \overrightarrow{{\bf \eta}_{3k-2, 3\ell-1} } ,  {{\bf a}_{ 3k-3, \cdot} } , {{\bf a}_{ \cdot, 3 \ell } } \big ) \, , 
 \end{align*}
otherwise set $Y^{(1,3)}_{i, j}  ( {{\bf a}_{ 3k-3, \cdot} } ;  {{\bf a}_{ \cdot, 3 \ell } }   )   =0$. Then define
\[
{\mathbb Y}^{(1,3)}_{k, \ell}  ( {{\bf a}_{ 3k-3, \cdot} }  ;  {{\bf a}_{ \cdot, 3 \ell } } )   = \Big (  Y^{(1,3)}_{i, j}  ( {{\bf a}_{ 3k-3, \cdot} } ;  {{\bf a}_{ \cdot, 3 \ell } } )      \Big )_{  (i,j) \in I^{(3m)}_{k, \ell}  } \, .
\]
$\bullet$ If $(i,j)  \in  [\![ (3k-2)m+1, (3k -1) m]\!] \times  [\![ (3\ell-3)m+1, (3\ell-2)  m]\!] $, let 
\begin{align*}
Y^{(2,1)}_{i, j}  ( {{\bf a}_{ \cdot, 3 \ell - 3} } )   & = g_m \Big (  \big (  \varepsilon_{u,v}  \big )_{ i-m \leq u \leq i  \atop  (3\ell-3)m+1 \leq v \leq  j} 
 ,  \big (  a_{u,v}  \big )_{ i-m \leq u \leq i  \atop   j-m \leq v \leq (3\ell-3)m}    \Big ) \\
 & : =  g_{m,i,j}^{(2,1)} \big (  \overrightarrow{{\bf \eta}_{3k-2, 3\ell-2} }, \overrightarrow{{\bf \eta}_{3k-1, 3\ell-2} }   ,    {{\bf a}_{ \cdot, 3 \ell - 3} }  \big ) \, , 
 \end{align*}
otherwise set $Y^{(2,1)}_{i, j}  (  {{\bf a}_{ \cdot, 3 \ell - 3} } )  =0$. Then define
\[
{\mathbb Y}^{(2,1)}_{k, \ell}  (   {{\bf a}_{ \cdot, 3 \ell - 3} } )  = \Big (  Y^{(2,1)}_{i, j}  (   {{\bf a}_{ \cdot, 3 \ell - 3} } )  \Big )_{  (i,j) \in I^{(3m)}_{k, \ell}  } \, .
\]
$\bullet$ If $(i,j)  \in  [\![ (3k-2)m+1, (3k -1) m]\!] \times  [\![ (3\ell-2)m+1, (3\ell-1)  m]\!] $, let 
\[
Y^{(2,2)}_{i, j}   = g_m \Big (  \big (  \varepsilon_{u,v}  \big )_{ i-m \leq u \leq i  \atop  j-m \leq v \leq  j} 
   \Big )  \, , 
\]
otherwise set $Y^{(2,2)}_{i, j}    =0$. Then define
\[
{\mathbb Y}^{(2,2)}_{k, \ell}  = \Big (  Y^{(2,2)}_{i, j}   \Big )_{  (i,j) \in I^{(3m)}_{k, \ell}  } \, .
\]
$\bullet$ If $(i,j)  \in  [\![ (3k-2)m+1, (3k -1) m]\!] \times  [\![ (3\ell-1)m+1, 3\ell  m]\!] $, let 
\begin{align*}
Y^{(2,3)}_{i, j}  ( {{\bf a}_{ \cdot, 3 \ell} } )   & =  g_m \Big (  \big (  \varepsilon_{u,v}  \big )_{ i-m \leq u \leq i  \atop  j-m \leq v \leq  (3\ell-1)m} 
 ,  \big (  a_{u,v}  \big )_{ i-m \leq u \leq i  \atop   (3\ell-1)m +1 \leq v \leq j  }  \Big ) \\
 & : =   g_{m,i,j}^{(2,3)} \big (  \overrightarrow{{\bf \eta}_{3k-2, 3\ell-1} }, \overrightarrow{{\bf \eta}_{3k-1, 3\ell-1} }   ,    {{\bf a}_{ \cdot, 3 \ell } }  \big ) \, , 
 \end{align*}
otherwise set $Y^{(2,3)}_{i, j}  (  {{\bf a}_{ \cdot, 3 \ell } } )  =0$. Then define
\[
{\mathbb Y}^{(2,3)}_{k, \ell}  (   {{\bf a}_{ \cdot, 3 \ell } } )  = \Big (  Y^{(2,3)}_{i, j}  (   {{\bf a}_{ \cdot, 3 \ell } } )  \Big )_{  (i,j) \in I^{(3m)}_{k, \ell}  } \, .
\]
$\bullet$ If $(i,j)  \in  [\![ (3k-1)m+1, 3k  m]\!] \times  [\![ (3\ell-3)m+1, (3\ell-2)  m]\!] $, let 
\begin{align*}
Y^{(3,1)}_{i, j}  ( {{\bf a}_{ 3k, \cdot } } ;  {{\bf a}_{ \cdot, 3 \ell - 3} } )   & = g_m \Big (  \big (  \varepsilon_{u,v}  \big )_{ i-m \leq u \leq (3k-1)m  \atop  (3\ell-3)m+1 \leq v \leq  j} 
 ,  \big (  a_{u,v}  \big )_{ i-m \leq u \leq (3k-1)m  \atop   j-m \leq v \leq (3\ell-3)m}     ,  \big (  a_{u,v}  \big )_{ (3k-1)m +1 \leq u \leq  i \atop   j-m \leq v \leq j}  \Big ) \\
 & : =  g_{m,i,j}^{(3,1)} \big (  \overrightarrow{{\bf \eta}_{3k-1, 3\ell-2} },  {{\bf a}_{ 3k, \cdot } } ,    {{\bf a}_{ \cdot, 3 \ell - 3} }  \big ) \, ,  
 \end{align*}
otherwise set $Y^{(3,1)}_{i, j}   ( {{\bf a}_{ 3k, \cdot } } ;  {{\bf a}_{ \cdot, 3 \ell - 3} } )  =0$. Then define
\[
{\mathbb Y}^{(3,1)}_{k, \ell}   ( {{\bf a}_{ 3k, \cdot } } ;   {{\bf a}_{ \cdot, 3 \ell - 3} } )  = \Big (  Y^{(3,1)}_{i, j}   ( {{\bf a}_{ 3k, \cdot } } ;  {{\bf a}_{ \cdot, 3 \ell - 3} } )  \Big )_{  (i,j) \in I^{(3m)}_{k, \ell}  } \, .
\]
$\bullet$ If $(i,j)  \in  [\![ (3k-1)m+1, 3k  m]\!] \times  [\![ (3\ell-2)m+1, (3\ell-1)  m]\!] $, let 
\begin{align*}
Y^{(3,2)}_{i, j}  ( {{\bf a}_{ 3k, \cdot } } )   & = g_m \Big (  \big (  \varepsilon_{u,v}  \big )_{ i-m \leq u \leq (3k-1)m  \atop j-m \leq v \leq  j} 
     ,  \big (  a_{u,v}  \big )_{ (3k-1)m +1 \leq u \leq  i \atop   j-m \leq v \leq j}  \Big ) \\
 & : =  g_{m,i,j}^{(3,2)} \big (  \overrightarrow{{\bf \eta}_{3k-1, 3\ell-2} },   \overrightarrow{{\bf \eta}_{3k-1, 3\ell-1} },  {{\bf a}_{ 3k, \cdot } }  \big ) \, ,  
 \end{align*}
otherwise set $Y^{(3,2)}_{i, j}   ( {{\bf a}_{ 3k, \cdot } }  )  =0$. Then define
\[
{\mathbb Y}^{(3,2)}_{k, \ell}   ( {{\bf a}_{ 3k, \cdot } } )  = \Big (  Y^{(3,2)}_{i, j}   ( {{\bf a}_{ 3k, \cdot } }  )  \Big )_{  (i,j) \in I^{(3m)}_{k, \ell}  } \, .
\]
$\bullet$ If $(i,j)  \in  [\![ (3k-1)m+1, 3k  m]\!] \times  [\![ (3\ell-1)m+1, 3\ell  m]\!] $, let 
\begin{align*}
Y^{(3,3)}_{i, j}  ( {{\bf a}_{ 3k, \cdot } } ;  {{\bf a}_{  \cdot, 3 \ell } } )   & = g_m \Big (  \big (  \varepsilon_{u,v}  \big )_{ i-m \leq u \leq (3k-1)m  \atop j-m \leq v \leq  (3\ell-1)m} 
     ,  \big (  a_{u,v}  \big )_{ (3k-1)m +1 \leq u \leq  i \atop   j-m \leq v \leq  (3\ell-1)m} ,  \big (  a_{u,v}  \big )_{i-m \leq u \leq  i \atop    (3\ell-1)m +1 \leq v \leq j}  \Big ) \\
 & : =  g_{m,i,j}^{(3,3)} \big (  \overrightarrow{{\bf \eta}_{3k-1, 3\ell-1} },  {{\bf a}_{ 3k, \cdot } }  ,  {{\bf a}_{  \cdot, 3 \ell } }  \big ) \, ,  
 \end{align*}
otherwise set $Y^{(3,3)}_{i, j}   ( {{\bf a}_{ 3k, \cdot } } ,  {{\bf a}_{  \cdot, 3 \ell } }  )  =0$. Then define
\[
{\mathbb Y}^{(3,3)}_{k, \ell}   ( {{\bf a}_{ 3k, \cdot } };  {{\bf a}_{  \cdot, 3 \ell } }  )  = \Big (  Y^{(3,3)}_{i, j}   ( {{\bf a}_{ 3k, \cdot } } ;   {{\bf a}_{  \cdot, 3 \ell } } )  \Big )_{  (i,j) \in I^{(3m)}_{k, \ell}  } \, .
\]
We now center the nine matrices introduced above. For each $(u,v)\in\{1,2,3\}^2$, we denote by ${\mathbb M}^{(u,v)}_{k,\ell}(\cdot) := \E\big[{\mathbb Y}^{(u,v)}_{k,\ell}(\cdot)\big]$ 
the corresponding mean matrix, where the expectation is taken over the remaining random innovations, the frozen boundary values being kept fixed. We then set ${\widetilde {\mathbb Y}}^{(u,v)}_{k, \ell}   (  \cdot  )  : =   {\mathbb Y}^{(u,v)}_{k, \ell}   (  \cdot  )  - {\mathbb M}^{(u,v)}_{k, \ell}   (  \cdot ) $. For fixed deterministic boundary blocks, the randomness in the $3m\times 3m$ macro-block indexed by $(k,\ell)$ is carried only by the four interior innovation
 $(   \overrightarrow{{\bf \eta}_{3k-2, 3\ell-2} },   \overrightarrow{{\bf \eta}_{3k-2, 3\ell-1} },   \overrightarrow{{\bf \eta}_{3k-1, 3\ell-2} },   \overrightarrow{{\bf \eta}_{3k-1, 3\ell-1} })$.  Since these four blocks are disjoint from the corresponding interior blocks of other macro-blocks, the centered matrices associated with different pairs $(k,\ell)$ are independent after conditioning on the frozen boundary values. The next step is to reorganize the nine centered matrices into fewer components that will later be replaced by Gaussian matrices with matching covariances. The role of this decomposition is analogous to the conditional-expectation split used in the one-dimensional proof: each middle column block is decomposed into a conditional expectation part and a centered remainder. We first
combine them in the column direction. This produces six centered matrices, denoted by
\[\begin{aligned}
{\mathbb A}^{(1,1)}_{k, \ell}  ( {{\bf a}_{ 3k-3, \cdot} } ;  {{\bf a}_{ \cdot, 3 \ell - 3} } )  &= {\widetilde {\mathbb Y}}^{(1,1)}_{k, \ell}  ( {{\bf a}_{ 3k-3, \cdot} } ;  {{\bf a}_{ \cdot, 3 \ell - 3} } )  + \E \Big ( 
{\widetilde {\mathbb Y}}^{(1,2)}_{k, \ell}  ( {{\bf a}_{ 3k-3, \cdot} } ) |  \overrightarrow{{\bf \eta}_{3k-2, 3\ell-2} } \Big ) \, ,
\\
{\mathbb A}^{(1,2)}_{k, \ell} ( {{\bf a}_{ 3k-3, \cdot} } ;  {{\bf a}_{ \cdot, 3 \ell } } )  &= {\widetilde {\mathbb Y}}^{(1,3)}_{k, \ell}  ( {{\bf a}_{ 3k-3, \cdot} } ;  {{\bf a}_{ \cdot, 3 \ell } } )  + {\widetilde {\mathbb Y}}^{(1,2)}_{k, \ell}  ( {{\bf a}_{ 3k-3, \cdot} } )  - \E \Big ( 
{\widetilde {\mathbb Y}}^{(1,2)}_{k, \ell}  ( {{\bf a}_{ 3k-3, \cdot} } ) |  \overrightarrow{{\bf \eta}_{3k-2, 3\ell-2} } \Big ) \, ,
\\
{\mathbb A}^{(2,1)}_{k, \ell}  (  {{\bf a}_{ \cdot, 3 \ell - 3} } )  &= {\widetilde {\mathbb Y}}^{(2,1)}_{k, \ell}  (   {{\bf a}_{ \cdot, 3 \ell - 3} } )  + \E \Big ( 
{\widetilde {\mathbb Y}}^{(2,2)}_{k, \ell}   |  \overrightarrow{{\bf \eta}_{3k-2, 3\ell-2} } ,  \overrightarrow{{\bf \eta}_{3k-1, 3\ell-2} }  \Big ) \, ,
\\
{\mathbb A}^{(2,2)}_{k, \ell}  (  {{\bf a}_{ \cdot, 3 \ell } } )  &= {\widetilde {\mathbb Y}}^{(2,3)}_{k, \ell}  (  {{\bf a}_{ \cdot, 3 \ell } } )  + {\widetilde {\mathbb Y}}^{(2,2)}_{k, \ell}  - \E \Big ( 
{\widetilde {\mathbb Y}}^{(2,2)}_{k, \ell}   |  \overrightarrow{{\bf \eta}_{3k-2, 3\ell-2} } ,  \overrightarrow{{\bf \eta}_{3k-1, 3\ell-2} }  \Big ) \, ,
\\
{\mathbb A}^{(3,1)}_{k, \ell}  ( {{\bf a}_{ 3k, \cdot} } ;  {{\bf a}_{ \cdot, 3 \ell - 3} } )  &= {\widetilde {\mathbb Y}}^{(3,1)}_{k, \ell}  ( {{\bf a}_{ 3k, \cdot} } ;  {{\bf a}_{ \cdot, 3 \ell - 3} } )  + \E \Big ( 
{\widetilde {\mathbb Y}}^{(3,2)}_{k, \ell}  ( {{\bf a}_{ 3k, \cdot} } ) |  \overrightarrow{{\bf \eta}_{3k-1, 3\ell-2} } \Big ) \, ,
\\
{\mathbb A}^{(3,2)}_{k, \ell} ( {{\bf a}_{ 3k, \cdot} } ;  {{\bf a}_{ \cdot, 3 \ell } } )  &= {\widetilde {\mathbb Y}}^{(3,3)}_{k, \ell}  ( {{\bf a}_{ 3k, \cdot} } ;  {{\bf a}_{ \cdot, 3 \ell } } )  + {\widetilde {\mathbb Y}}^{(3,2)}_{k, \ell}  ( {{\bf a}_{ 3k, \cdot} } )  - \E \Big ( 
{\widetilde {\mathbb Y}}^{(3,2)}_{k, \ell}  ( {{\bf a}_{ 3k, \cdot} } ) |  \overrightarrow{{\bf \eta}_{3k-1, 3\ell-2} } \Big ) \, .
\end{aligned}
\]
We shall next combine these six matrices once more, now in the row direction, in order to obtain four centered matrices
${\mathbb B}^{(u,v)}_{k,\ell}(\cdot)
$.
These four matrices will play the role of the final uncorrelated block components in the two-dimensional construction.
\[
\begin{aligned}
{\mathbb B}^{(1,1)}_{k, \ell} ( {{\bf a}_{ 3k-3, \cdot} } ;  {{\bf a}_{ \cdot, 3 \ell - 3} } )  &= {\mathbb A}^{(1,1)}_{k, \ell} ( {{\bf a}_{ 3k-3, \cdot} } ;  {{\bf a}_{ \cdot, 3 \ell - 3} } )  + \E \Big ( 
{\mathbb A}^{(2,1)}_{k, \ell}  ( {{\bf a}_{ \cdot , 3 \ell -3} } ) |  \overrightarrow{{\bf \eta}_{3k-2, 3\ell-2} } \Big ) \, ,
\\
{\mathbb B}^{(2,1)}_{k, \ell} ( {{\bf a}_{ 3k, \cdot} } ;  {{\bf a}_{ \cdot, 3 \ell - 3} } )  &= {\mathbb A}^{(3,1)}_{k, \ell} ( {{\bf a}_{ 3k, \cdot} } ;  {{\bf a}_{ \cdot, 3 \ell - 3} } )  + {\mathbb A}^{(2,1)}_{k, \ell} (   {{\bf a}_{ \cdot, 3 \ell - 3} } ) - \E \Big ( 
{\mathbb A}^{(2,1)}_{k, \ell}  ( {{\bf a}_{ \cdot , 3 \ell -3} } ) |  \overrightarrow{{\bf \eta}_{3k-2, 3\ell-2} } \Big ) \, ,
\\
{\mathbb B}^{(1,2)}_{k, \ell} ( {{\bf a}_{ 3k-3, \cdot} } ;  {{\bf a}_{ \cdot, 3 \ell } } ) & = {\mathbb A}^{(1,2)}_{k, \ell} ( {{\bf a}_{ 3k-3, \cdot} } ;  {{\bf a}_{ \cdot, 3 \ell } } )  + \E \Big ( 
{\mathbb A}^{(2,2)}_{k, \ell}  ( {{\bf a}_{ \cdot , 3 \ell } } ) |  \overrightarrow{{\bf \eta}_{3k-2, 3\ell-1} }, \overrightarrow{{\bf \eta}_{3k-2, 3\ell-2} } \Big ) \, ,
\\
{\mathbb B}^{(2,2)}_{k, \ell} ( {{\bf a}_{ 3k, \cdot} } ;  {{\bf a}_{ \cdot, 3 \ell} } )  &= {\mathbb A}^{(3,2)}_{k, \ell} ( {{\bf a}_{ 3k, \cdot} } ;  {{\bf a}_{ \cdot, 3 \ell } } )  + {\mathbb A}^{(2,2)}_{k, \ell}  ( {{\bf a}_{ \cdot , 3 \ell } } ) -  \E \Big ( 
{\mathbb A}^{(2,2)}_{k, \ell}  ( {{\bf a}_{ \cdot , 3 \ell } } ) |  \overrightarrow{{\bf \eta}_{3k-2, 3\ell-1}},\overrightarrow{{\bf \eta}_{3k-2, 3\ell-2} }  \Big )  \, .
\end{aligned}
\]
By construction, the random variables 
\[\Big ({\mathbb B}^{(1,1)}_{k, \ell} ( {{\bf a}_{ 3k-3, \cdot} } ;  {{\bf a}_{ \cdot, 3 \ell - 3} } ) , {\mathbb B}^{(2,1)}_{k, \ell} ( {{\bf a}_{ 3k, \cdot} } ;  {{\bf a}_{ \cdot, 3 \ell - 3} } ) , {\mathbb B}^{(1,2)}_{k, \ell} ( {{\bf a}_{ 3k-3, \cdot} } ;  {{\bf a}_{ \cdot, 3 \ell } } ) , {\mathbb B}^{(2,2)}_{k, \ell} ( {{\bf a}_{ 3k, \cdot} } ;  {{\bf a}_{ \cdot, 3 \ell } } )  \Big )_{k, \ell \geq 1} \]
are independent for fixed boundary values. Moreover, direct but tedious computations prove that,  for $k $ and $\ell$ fixed, the four random matrices are uncorrelated. This is due to the property of the conditional expectation, the independent between 
the vectors  $ \overrightarrow{{\bf \eta}_{3k-2, 3\ell-2} }$, $   \overrightarrow{{\bf \eta}_{3k-2, 3\ell-1} }$,  $ \overrightarrow{{\bf \eta}_{3k-1, 3\ell-2} }$,  $ \overrightarrow{{\bf \eta}_{3k-1, 3\ell-1} }$ and the following well known fact: If $X$, $Y$, $Z$ are three random vectors such that
$X$ is independent of $\sigma(Z)$ and $\sigma(Y)\subset\sigma(Z)$, then, for any measurable fonction $g$ such that $\Vert g(X,Y) \Vert_1 < \infty$, $
\mathbb{E}(g(X,Y)|Z)=\mathbb{E}(g(X,Y)|Y)\text{ a.s.}$

We now introduce Gaussian matrices whose covariance structures match those of the
four uncorrelated block components defined above. To this end, define the vectorization operator ${\rm Vect} : {\mathcal M}_{3m,3m} \rightarrow {\mathbb R}^{9m^2}$ by 
\[
{\rm Vect} \big (   ( M_{i,j})_{1 \leq i \leq 3m \atop 1 \leq j \leq 3m } \big )  = \big ( M_{1,1}, \dots,   M_{3m,1}, M_{1,2}, \dots,   M_{3m,2}, \dots, M_{1,3m}, \dots,   M_{3m,3m} \big )^t \, .
\]
Let 
\[
\Gamma^{(1,1)}_{k, \ell} ( {{\bf a}_{ 3k-3, \cdot} } ;  {{\bf a}_{ \cdot, 3 \ell - 3} } )  = {\rm Cov } \Big ( {\rm Vect} \big (  {\mathbb B}^{(1,1)}_{k, \ell} ( {{\bf a}_{ 3k-3, \cdot} } ;  {{\bf a}_{ \cdot, 3 \ell - 3} } )  \big ) \Big )
\]
\[
\Gamma^{(2,1)}_{k, \ell} ( {{\bf a}_{ 3k, \cdot} } ;  {{\bf a}_{ \cdot, 3 \ell - 3} } )  = {\rm Cov } \Big ( {\rm Vect} \big (  {\mathbb B}^{(2,1)}_{k, \ell} ( {{\bf a}_{ 3k, \cdot} } ;  {{\bf a}_{ \cdot, 3 \ell - 3} } )  \big ) \Big )
\]
\[
\Gamma^{(1,2)}_{k, \ell} ( {{\bf a}_{ 3k-3, \cdot} } ;  {{\bf a}_{ \cdot, 3 \ell } } )  = {\rm Cov } \Big ( {\rm Vect} \big (  {\mathbb B}^{(1,2)}_{k, \ell} ( {{\bf a}_{ 3k-3, \cdot} } ;  {{\bf a}_{ \cdot, 3 \ell } } )  \big ) \Big )
\]
\[
\Gamma^{(2,2)}_{k, \ell} ( {{\bf a}_{ 3k, \cdot} } ;  {{\bf a}_{ \cdot, 3 \ell } } )  = {\rm Cov } \Big ( {\rm Vect} \big (  {\mathbb B}^{(2,2)}_{k, \ell} ( {{\bf a}_{ 3k, \cdot} } ;  {{\bf a}_{ \cdot, 3 \ell } } )  \big ) \Big ) \, .
\]
Let $(Z^{(1,1)}_{i,j })_{i,j\in {\mathbb Z}}$, $(Z^{(1,2)}_{i,j })_{i,j\in {\mathbb Z}}$, $(Z^{(2,1)}_{i,j })_{i,j\in {\mathbb Z}}$ and $(Z^{(2,2)}_{i,j })_{i,j\in {\mathbb Z}}$ be a mutually independent family of iid standard Gaussian random variables that are also independent of the random field $(\varepsilon_{u,v})_{u,v \in {\mathbb Z}}$. Define the following square random matrix of size $3m$:  
\[
{\mathbb N}^{(1,1)}_{k, \ell} ( {{\bf a}_{ 3k-3, \cdot} } ;  {{\bf a}_{ \cdot, 3 \ell - 3} } )  = {\rm Vect}^{-1} \Big ( \big [ \Gamma^{(2,1)}_{k, \ell} ( {{\bf a}_{ 3k, \cdot} } ;  {{\bf a}_{ \cdot, 3 \ell - 3} } )  \big ]^{1/2}
 {\rm Vect} \big (   ( Z^{(1,1)}_{i,j})_{(3k-3)m+1 \leq i \leq 3km \atop (3\ell-3)m+1 \leq j \leq 3 \ell m } \big )   \Big )  \, .
\]
We define similarly the Gaussian matrices ${\mathbb N}^{(2,1)}_{k, \ell} ( {{\bf a}_{ 3k, \cdot} } ;  {{\bf a}_{ \cdot, 3 \ell - 3} } ) $, ${\mathbb N}^{(1,2)}_{k, \ell} ( {{\bf a}_{ 3k-3, \cdot} } ;  {{\bf a}_{ \cdot, 3 \ell } } ) $ and ${\mathbb N}^{(2,2)}_{k, \ell} ( {{\bf a}_{ 3k, \cdot} } ;  {{\bf a}_{ \cdot, 3 \ell } } ) $ with the obvious modifications. For fixed boundary values, the four Gaussian matrices are mutually independent.  

To reconstruct the original finite-memory block, we must also restore the deterministic centering terms removed from the nine matrices $\mathbb Y^{(r,s)}_{k,\ell}$, $r,s\in\{1,2,3\}$. For $u,v\in\{1,2\}$, we denote by $\widehat{\mathbb M}^{(u,v)}_{k,\ell}$ the deterministic mean correction associated with $\mathbb B^{(u,v)}_{k,\ell}$ and denote by
\[
\overrightarrow{{\bf \eta}_{{\bf 3},\cdot}}
:=
\Big\{
\overrightarrow{{\bf \eta}_{3k,\cdot}}
:
k\in\mathbb Z
\Big\}, \;
\overrightarrow{{\bf \eta}_{\cdot,{\bf 3}}}
:=
\Big\{
\overrightarrow{{\bf \eta}_{\cdot,3\ell}}
:
\ell\in\mathbb Z
\Big\}, \; {\bf a}_{{\bf 3},\cdot}
:=
\Big\{
{\bf a}_{3k,\cdot}
:
k\in\mathbb Z
\Big\},
\quad \text{and} \quad
{\bf a}_{\cdot,{\bf 3}}
:=
\Big\{
{\bf a}_{\cdot,3\ell}
:
\ell\in\mathbb Z
\Big\}.
\]
Finally, let 
\[
{\mathbb X}^{(m)} ( {{\bf a}_{ \bf{3}, \cdot} } ;  {{\bf a}_{ \cdot, \bf 3 } } )  = \frac{1}{\sqrt{n}} \sum_{k,\ell=1}^{k_{n,m}}  \sum_{(i,j) \in I^{(3m)}_{k, \ell}  \atop j\leq i }
\sum_{u,v=1}^2 \Big [  {\mathbb B}^{(u,v)}_{k, \ell} ( {{\bf a}_{ 3k-3(2-u), \cdot} } ;  {{\bf a}_{ \cdot, 3 \ell - 3 (2-v)} } ) +  \widehat{\mathbb M}^{(u,v)}_{k, \ell} ( {{\bf a}_{ 3k-3(2-u), \cdot} } ;  {{\bf a}_{ \cdot, 3 \ell - 3 (2-v)} } )  \Big ]_{i,j} E_{i,j} \, ,
\]
and note that ${\mathbb X}^{(m)} = {\mathbb X}^{(m)} (  \overrightarrow{{\bf \eta}_{ \bf 3, \cdot} } ;   \overrightarrow{{\bf \eta}_{ \cdot,\bf 3 } } )$. Also note that throughout this section, every sum over $(i,j)\in I^{(3m)}_{k,\ell}$ is understood to be restricted to $(i,j)\in I^{(3m)}_{k,\ell}\cap[\![1,n]\!]^2$, together with the condition $j\le i$. We then define  
\[
{\mathbb Z}^{(m)} ( {{\bf a}_{ \bf 3, \cdot} } ;  {{\bf a}_{ \cdot, \bf 3 } } )  = \frac{1}{\sqrt{n}} \sum_{k,\ell=1}^{k_{n,m}}  \sum_{(i,j) \in I^{(3m)}_{k, \ell} \atop j\leq i }
\sum_{u,v=1}^2 \Big [  {\mathbb N}^{(u,v)}_{k, \ell} ( {{\bf a}_{ 3k-3(2-u), \cdot} } ;  {{\bf a}_{ \cdot, 3 \ell - 3 (2-v)} } ) +  \widehat{\mathbb M}^{(u,v)}_{k, \ell} ( {{\bf a}_{ 3k-3(2-u), \cdot} } ;  {{\bf a}_{ \cdot, 3 \ell - 3 (2-v)} } )  \Big ]_{i,j} E_{i,j} \, ,
\]
 and set ${\mathbb Z}^{(m)} = {\mathbb Z}^{(m)} (  \overrightarrow{{\bf \eta}_{\bf 3, \cdot} } ;   \overrightarrow{{\bf \eta}_{ \cdot, \bf 3 } } )$. The comparison between $\mathbb X^{(m)}$ and $\mathbb Z^{(m)}$, corresponding to the term $Q_1$ in Section \ref{Subsection:Step2}, can now be carried out by the same interpolation argument as in the one-directional case. The essential difference is that an independent block now contains $m^2$ entries rather than $m$ entries. Accordingly, the block-size estimates appearing in \eqref{firstboundtrace} must be adjusted by replacing the one-dimensional block cardinality $m$ with the two-dimensional cardinality $m^2$. Next, we write that 
\begin{multline*}
{\mathbb Z}^{(m)} = \frac{1}{\sqrt{n}} \sum_{k,\ell=0}^{k_{n,m}}  \sum_{(i,j) \in I^{(3m)}_{k, \ell} \atop j\leq i } \Big \{ 
 \Big [  {\mathbb N}^{(1,1)}_{k+1, \ell} (\overrightarrow{{\bf \eta}_{ 3k, \cdot} } ; \overrightarrow{{\bf \eta}_{ \cdot, 3 \ell - 3} } ) +  \widehat{\mathbb M}^{(1,1)}_{k+1, \ell} (\overrightarrow{{\bf \eta}_{ 3k, \cdot} } ; \overrightarrow{{\bf \eta}_{ \cdot, 3 \ell - 3} } )  
 +   {\mathbb N}^{(1,2)}_{k+1, \ell} (\overrightarrow{{\bf \eta}_{ 3k, \cdot} } ; \overrightarrow{{\bf \eta}_{ \cdot, 3 \ell } } ) +  \widehat{\mathbb M}^{(1,2)}_{k+1, \ell} (\overrightarrow{{\bf \eta}_{ 3k, \cdot} } ; \overrightarrow{{\bf \eta}_{ \cdot, 3 \ell } } ) \\
  +   {\mathbb N}^{(2,1)}_{k, \ell} (\overrightarrow{{\bf \eta}_{ 3k, \cdot} } ; \overrightarrow{{\bf \eta}_{ \cdot, 3 \ell -3 } } ) +  \widehat{\mathbb M}^{(2,1)}_{k, \ell} (\overrightarrow{{\bf \eta}_{ 3k, \cdot} } ; \overrightarrow{{\bf \eta}_{ \cdot, 3 \ell-3 } } )  
  +   {\mathbb N}^{(2,2)}_{k, \ell} (\overrightarrow{{\bf \eta}_{ 3k, \cdot} } ; \overrightarrow{{\bf \eta}_{ \cdot, 3 \ell  } } ) +  \widehat{\mathbb M}^{(2,2)}_{k, \ell} (\overrightarrow{{\bf \eta}_{ 3k, \cdot} } ; \overrightarrow{{\bf \eta}_{ \cdot, 3 \ell } } )  \Big ]_{i,j} 
 \Big \}  E_{i,j} \, ,
\end{multline*} 
so that,  conditionally on the frozen-boundary field, $\mathbb{Z}^{(m)}$ is expressed as a sum of independent Gaussian block contributions with the same conditional means and covariances as the corresponding contributions to $\mathbb X^{(m)}$. Following the strategy used in the one-directional case, we do not compare $\mathbb Z^{(m)}$ directly with the final Gaussian matrix. Instead, we introduce an intermediate Gaussian proxy, denoted by $\mathbb W^{(m)}$ where we first condition on vertical boundary strips $\overrightarrow{{\bf \eta}_{\cdot,{\bf 3}}}$. This can be suitably approximated by a Wigner-type random matrix:
\[
{\mathbb W}^{(m)}  :=  \frac{1}{\sqrt{n}} \sum_{k,\ell=1}^{k_{n,m}}  \sum_{(i,j) \in I^{(3m)}_{k, \ell} \atop j\leq i }     \Big [  {\mathbb G}^{(1)}_{k,\ell} (  \overrightarrow{{\bf \eta}_{  \cdot, 3 \ell -3} })+ \widetilde{\mathbb{M}}^{(1)}_{k,\ell} (  \overrightarrow{{\bf \eta}_{  \cdot, 3 \ell -3} }) +  {\mathbb G}^{(2)}_{k,\ell} (  \overrightarrow{{\bf \eta}_{  \cdot, 3 \ell } } ) +\widetilde{\mathbb{M}}^{(2)}_{k,\ell} (  \overrightarrow{{\bf \eta}_{  \cdot, 3 \ell } })\Big ]_{i,j} E_{i,j} \, ,
\]
where the conditional mean corrections $\widetilde{\mathbb{M}}^{(u)}_{k,\ell}$ are obtained by regrouping the matrices $\widehat{\mathbb M}^{(u,v)}_{k,\ell}$ in the same way as the centered Gaussian contributions in $\mathbb Z^{(m)}$. When $ \overrightarrow{{\bf \eta}_{  \cdot, 3 } }  = {\bf a}_{  \cdot, 3 } $, the random matrices  $ \big (   {\mathbb G}^{(1)}_{k,\ell} ( {\bf a}_{  \cdot, 3 \ell } )  \big )_{k, \ell}$  and $ \big (   {\mathbb G}^{(2)}_{k,\ell} ( {\bf a}_{  \cdot, 3 \ell -3} )  \big )_{k, \ell}$ form two independent sequences of iid and Gaussian random variables. Now, rearranging the terms in ${\mathbb W}^{(m)}$, we write
\[
{\mathbb W}^{(m)}  :=  \frac{1}{\sqrt{n}} \sum_{k=1}^{k_{n,m}} \sum_{\ell=0}^{k_{n,m}}  \sum_{(i,j) \in I^{(3m)}_{k, \ell} \atop j\leq i }     \Big [  {\mathbb G}^{(1)}_{k,\ell+1} (  \overrightarrow{{\bf \eta}_{  \cdot, 3 \ell } }) +\widetilde{\mathbb{M}}^{(1)}_{k,\ell+1} (  \overrightarrow{{\bf \eta}_{  \cdot, 3 \ell } })+  {\mathbb G}^{(2)}_{k,\ell} (  \overrightarrow{{\bf \eta}_{  \cdot, 3 \ell } } )+\widetilde{\mathbb{M}}^{(2)}_{k,\ell} (  \overrightarrow{{\bf \eta}_{  \cdot, 3 \ell } })\Big ]_{i,j} E_{i,j} \, .
\]
Using the independence by blocks, the passage from $\mathbb W^{(m)}$ to the final Gaussian Wigner-type matrix $\mathbb G^{(m)}$ is obtained by applying the same regrouping and Gaussian replacement procedure in the vertical direction  
\[
{\mathbb G}^{(m)}  :=  \frac{1}{\sqrt{n}} \sum_{k=1}^{k_{n,m}} \sum_{\ell=0}^{k_{n,m}}  \sum_{(i,j) \in I^{(3m)}_{k, \ell} \atop j\leq i }     \Big [  {\mathbb N}^{(m)}_{k,\ell} \Big ]_{i,j} E_{i,j} \, ,
\]
where $\mathbb N^{(m)}_{k,\ell}$s are the resulting Gaussian blocks including the corresponding deterministic mean correction. By construction, $\mathbb G^{(m)}$ is Gaussian and, after the two successive
regroupings in the horizontal and vertical directions, its covariance structure coincides with that of the finite-memory process. 

The remaining estimates follow the same pattern as in the one-dimensional argument. The difference is that an independent macro-block now contains order $m^2$ scalar entries, rather than order $m$. Thus, in each block-size estimate (including the bounds used in the interpolation comparison and in the subsequent Gaussian replacement), the one-dimensional block cardinality $m$ is replaced by $m^2$. The truncation range in each coordinate remains $m$, so the truncation error is still controlled by $\delta(m)$. Taking $m=c^{-1}(2+\kappa)\log n$ as in the one-dimensional proof and inserting the resulting estimates into the two spectral-inclusion bounds gives the rate stated in Theorem~\ref{thm:Hausdorff distance2D}.

\section*{Acknowledgement}
\qquad M. Banna and I.M. Dabo acknowledge support from M. Banna's research fund at NYU Abu Dhabi. M. Banna was partially supported by the Fédération de Recherche d’Animation Bézout during her June 2026 visit to Université Gustave Eiffel as a visiting professor, during which substantial progress was made on this project.

 The research, development of the mathematical formulation and proofs, and the preparation of the manuscript were conducted by the three authors. ChatGPT was used for proofreading and language editing.

%{ The authors declare that they have no conflict of interest. No datasets were generated or analyzed during the current study.}

\bibliographystyle{abbrv}
\bibliography{ref}
\end{document}